\documentclass[a4paper,11pt,reqno]{article}

\usepackage{a4wide, fullpage,enumerate}
\usepackage{amsmath,amssymb,amsthm,amsfonts}
\usepackage{color}
\usepackage{hyperref}

\theoremstyle{plain}
\newtheorem{prop}{Proposition}[section]
\newtheorem{theorem}[prop] {Theorem}

\theoremstyle{definition}
\newtheorem{lemma}[prop]{Lemma}
\newtheorem{definition}[prop]{Definition}
\newtheorem{cor}[prop]{Corollary}
\newtheorem{assumption}[prop]{Assumption} 

\theoremstyle{remark}
\newtheorem*{remark}{Remark}
\newtheorem{example}{Example}[section]

\newcommand{\N}{\mathbb{N}} 
\newcommand{\R}{\mathbb{R}} 
\newcommand{\Z}{\mathbb{Z}} 
\newcommand{\C}{\mathbb{C}} 
\renewcommand{\P}{\mathbb{P}} 
\newcommand{\E}{\mathbb{E}} 
\newcommand{\Hi}{\mathcal{H}} 
\newcommand{\meas}{\mathfrak{M}} 
\newcommand{\prob}{\mathfrak{M}_{1,+}} 

\newcommand{\dd}{\mathrm{d}} 
\newcommand{\eps}{\varepsilon}
\newcommand{\la}{\langle} 
\newcommand{\ra}{\rangle} 
\newcommand{\trans}{\mathrm{T}} 
\newcommand{\id}{\mathrm{id}} 

\newcommand{\vect}[1]{\boldsymbol{#1}} 

\DeclareMathOperator{\supp}{supp} 
 \DeclareMathOperator {\ran}{ran} 
\DeclareMathOperator{\ex}{ex} 

\newcommand{\noemi}[1]{\textcolor{black}{\textrm{#1}}}

\title{On the notion(s) of duality for Markov processes}
\date{\today}
\author{Sabine Jansen, Ruhr-Universit{\"a}t Bochum, and Noemi Kurt, TU Berlin }
	
\begin{document}

\maketitle

\begin{abstract}
	We provide a systematic study of the notion of duality of Markov processes with respect to a function. We discuss the relation of this notion with duality with respect to a measure as studied in Markov process theory and potential theory and give functional analytic results including existence and uniqueness criteria and a comparison of the spectra of dual semi-groups. The analytic framework builds on the notion of dual pairs, convex geometry, and Hilbert spaces. In addition, we formalize the notion of pathwise duality as it appears in population genetics and interacting particle systems. We discuss the relation of duality with rescalings, stochastic monotonicity, intertwining, symmetries, and quantum many-body theory, reviewing known results and establishing some new connections. \\
	
{\small \noindent \textit{Keywords}:  Markov processes, duality, semi-groups, cone duality, time reversal, interacting particle systems\\
	\noindent
\emph{MSC 2010 Subject classification:} Primary: 60J25. Secondary: 46N30, 47D07, 60J05.}

\end{abstract}

\tableofcontents
\section{Introduction}
Duality of Markov processes with respect to a duality function has first appeared in the literature in the late 40s and early 50s \cite{levy48, KarlinMcGregor, Lindley}, and has been formalized and generalized over the following decades \cite{Siegmund76, holley-stroock79, CoxRoesler84, CliffordSudbury85, EthierKurtz86, Vervaat88, SudburyLloyd95}. It has since been applied in a variety of fields ranging from interacting particle systems, queueing theory, SPDEs, and superprocesses to mathematical population genetics. In spite of this wide interest and applicability, there are so far few attempts at developing a systematic theory of duality of Markov processes with respect to a function, unlike for the different although related notion of duality with respect to a \emph{measure}, for which there exists a rather complete theory, see \cite{chung-walsh, getoor} for recent surveys. Overviews of the method of duality with respect to a function generally focus on certain aspects or applications to particular fields \cite{DawsonGreven, Liggett05, EthierKurtz86, DiaconisFill, Moehle99, Sudbury, asmussen}, and presentations of the manifold connections to fundamental structures or properties of Markov processes, such as time reversal, stochastic monotonicity, symmetries, or conserved quantities, are often restricted to specific problems. The interest in a general theory of duality has further increased in recent years, but even basic questions such as giving necessary and sufficient conditions for the existence of a dual process of a given Markov process have not yet been fully resolved -- ``finding dual processes is something of a black art'' \cite[p. 519]{etheridge_leshouches}. It has however seen substantial development for example in the work by M\"ohle, illuminating relations with symmetries and conserved quantities \cite{Moehle99, moehle-cones}, and by Giardina, Redig, Kurchan and Vafayi \cite{RedigEtAl09} which presents a deep connection with symmetries and representations of Lie algebras, using quantum mechanics formalisms. On the other side, the lookdown construction \cite{DonnellyKurtz1, DonnellyKurtz2} has triggered new and powerful applications. However, a unified treatment of the theory presenting fundamental connections is still missing. 

The present paper is on the one hand a survey of the notion of duality of Markov processes with respect to a function, and on the other hand also presents new results in this field.  A particular focus is on the question of existence and uniqueness of dual processes (Section~\ref{sec:fa}). These are formulated in functional analytic language on the level of Markovian semi-groups, and relate the problem to the invariance of certain (convex) sets and to the existence of certain unique integral representations via the concept of cone duality. We also formalize the notion of pathwise duality, which is of particular importance in applications (Section~\ref{sec:pathwise}). Moreover, connections with time reversal (duality with respect to a measure, Section~\ref{sec:diagonal}), stochastic monotonicity (Section~\ref{sec:monotonicity}), intertwinings and symmetries (Section~\ref{sec:hilbert}) are discussed.

The aim of the paper is to give an overview of the theoretical background of the concept of duality of Markov processes, and to present fundamental connections in a unified way. We hope that it provides a reference for probabilists applying duality techniques to various situations, who might be interested in the fundamental principles of this theory. We also try to assist understanding of certain results from mathematical physics, Hilbert space theory or quantum mechanics for researchers who might not be familiar with the jargon of these fields. Last but not least, we hope that this article triggers new research in this multifaceted and widely applicable area of probability theory.

\subsection{Setting and definitions}

In the following, $X$ and $Y$ 
are Markov processes with stationary transition probabilities and state spaces $E$ and $F$. 
The state spaces are assumed to be Polish and are endowed with the Borel $\sigma$-algebras $\mathcal{B}(E)$ and $\mathcal{B}(F)$.  
For our purpose a Markov process  with state space $E$ is a collection $X = (\Omega,\mathcal{F},(X_t)_{t\geq 0}, \{\P_x\}_{x\in E})$ consisting of a measure space $(\Omega,\mathcal{F})$, measurable maps $X_t: (\Omega,\mathcal{F})  \to (E ,\mathcal{B}(E))$, and probability measures $\P_x$ on $(\Omega, \mathcal{F})$ such that: 
\begin{itemize} \itemsep0pt 
	\item for all $x\in E$, $X_0 = x$, $\P_x$-a.s.; 
	\item for every Borel-measurable bounded function $f:E\to \R$, the map $x\mapsto \P_x f(X_t)$ is Borel-measurable as well; 
	\item the process satisfies the Markov property with respect to the natural filtration \\
	$\mathcal{F}_t^0 = \sigma( X_s,\, 0 \leq s\leq t)$. 
\end{itemize} 
We do not assume that the strong Markov property holds, and unless explicitly stated otherwise, we do not assume any regularity of the sample paths. Thus our processes have less structure than commonly assumed in the theory of Markov processes, as exposed in classical textbooks such as \cite{dynkin, EthierKurtz86, RogersWilliams}. The reason is that we will only need properties determined by the finite-dimensional distributions. 
The basic concept we are interested in is the following: 

\begin{definition}[Duality with respect to a duality function] \label{def:duality1} 
	Let $X = (\Omega_1,\mathcal{F}_1,(X_t)_{t\geq 0}, \{ \P_x\}_{x\in E} )$ and 
	$Y = (\Omega_2,\mathcal{F}_2,(X_t)_{t\geq 0}, \{ \P^y\}_{y\in F} )$
	 be two Markov processes with respective 
	state spaces $E$ and $F,$ and $H:E\times F\to \R$ a bounded, measurable function. Then $X$ and $Y$ are \emph{dual with respect to $H$} 
	if and only if 
	 for all $x\in E$, $y\in F$ and $t\geq 0$ 
	\begin{equation} \label{eq:duality-def}
		\E_x H(X_t,y) = \E^y H(x,Y_t).
	\end{equation}
\end{definition}
\begin{remark}
Here and throughout this paper we assume boundedness of the duality function $H$ for the sake of simplicity of the exposition. Of course duality can in principle be defined for suitable unbounded functions as well.
\end{remark}
From now on we drop the explicit mention of the underlying measure spaces and simply speak of Markov processes $(X_t)$ and $(Y_t)$. Throughout $\P_x$ refers to the law of $X$ started in $x$ and $\P^y$ to the law of the dual process started in $y$. 

Let  $(P_t)_{t\geq 0}$ and $(Q_t)_{t\geq 0}$ denote the semi-groups of $(X_t)$ and $(Y_t),$ respectively, that is, $P_tf(x)=\E_xf(X_t),$ and similarly for $Q_t.$
A different way of writing the duality formula is 
\begin{equation}\label{eq:duality_semigroup}
P_tH(\cdot,y)(x)=Q_tH(x,\cdot)(y).\end{equation}
Note that duality implies that for every $t,$
\begin{equation}P_tH(\cdot,y)(x)=P_s[P_{t-s}H(\cdot,y)](x)=P_s [Q_{t-s}H(\cdot,\cdot)(y)](x),\quad 0\leq s\leq t,\end{equation}
where in the first equality we have used the Chapman-Kolmogorov equation, and in the second the duality property. Note that $P_t$ always acts of $H$ as a function of the first, $Q_t$ as a function of the second variable. Assume now that $(X_t)$ and $(Y_t)$ have generators $L^X$ and $L^Y$ with domains $D(L^X)$ and $D(L^Y)$ respectively, and that $H(x,\cdot)\in D(L^Y), H(\cdot, y)\in D(L^X).$  Equation \eqref{eq:duality_semigroup} then implies
$$L^X H(\cdot, y)(x)=L^YH(x,\cdot)(y)\;\;x\in E, y\in F.$$ The converse is true as well, under certain conditions:

\begin{prop}\label{prop:duality_generator}
Let $(X_t),(Y_t)$ be Markov processes with generators $L^X, L^Y,$ let $H:E\times F\to\R$ be bounded and continuous. If $H(x,\cdot),P_tH(x,\cdot)\in D(L^Y)$ for all $x\in E, t\geq 0$ and $H,(\cdot, y), Q_tH(\cdot,y)\in D(L^X)$ for all $y\in F, t\geq 0,$ and if
\begin{equation}\label{eq:duality_generators}
L^XH(\cdot, y)(x)=L^YH(x,\cdot)(y)\quad \forall x\in E, y\in F,\end{equation}
then $(X_t)$ and $(Y_t)$ are dual with respect to $H.$
\end{prop}

\begin{proof}
Let $u_1(t,x,y):= \E_x(X_t,y)=P_tH(x,y),$ and $u_2(t,x,y):= \E^y(x,Y_t)=Q_tH(x,y).$ Note first that by Fubini $P_t Q_sH(x,y)=Q_sP_t H(x,y)$ for $s,t\geq 0, x\in E, y\in F,$ therefore, by our assumptions, $P_t L^Y H(x,\cdot)(y)=L^YP_tH(x,\cdot)(y).$  Using the Kolmogorov forward equation and \eqref{eq:duality_generators}, we get
\begin{equation*}
\begin{split}
\frac{d}{dt}u_1(t,x,y)=&P_t L^X H(\cdot,y)(x)=P_t L^YH(x,\cdot)(y)
=L^Y u_1(t,x,y).
\end{split}
\end{equation*}
Since also $\frac{d}{dt}u_2(t,x,y)=L^Yu_2(t,x,y)$ and $u_1(0,x,y)=H(x,y)=u_2(0,x,y)$ for all $x\in E, y\in F,$ the claim follows from the uniqueness of the solution of the initial value problem associated with $L^Y$ (see \cite[Thm. 1.3]{dynkin}).
\end{proof}

There is another notion of duality. In Section~\ref{sec:diagonal} we shall 
see that we can think of it, roughly, as a specialisation of the previous definition to 
diagonal duality functions. 
The definition is usually given in terms of the \emph{resolvent} $R_\lambda(x,A):=\int_0^\infty \exp(-\lambda t) P_t(x,A) \dd t$, and allows processes with a finite life-time corresponding to sub-Markov semi-groups. 

\begin{definition}[Duality with respect to a measure] \label{def:duality2}
	Let $(X_t)$, $(Y_t)$ be two (sub)-Markov processes  with common  
	state space $E$, semi-groups $(P_t)$, $(Q_t)$, and resolvents $R_\lambda$, $\hat R_\lambda$.
	Then $(X_t)$ and $(Y_t)$ 	are said to be \emph{in duality} with respect to the $\sigma$-finite measure $\mu$ 
	if  (i) for all $\lambda >0$ and all non-negative $f,g\in L^\infty(E)$, 
	\begin{equation} \label{eq:meas-duality}
		\int (R_\lambda f)(x) g(x) \mu(\dd x) 
				= \int f(x) (\hat R_\lambda  g)(x) \mu (\dd x),
	\end{equation}
	and (ii) the resolvents are absolutely continuous with respect to $\mu$, $R_\lambda(x,\cdot) \ll \mu(\cdot)$,$\hat R_\lambda(y,\cdot)\ll \mu(\cdot)$ for all $\lambda>0$ and $x,y\in E$. If only (i) holds, then $(X_t)$ and $(Y_t)$ are said to be \emph{in weak duality} with respect to $\mu$. 
\end{definition}

If $(P_t) = (Q_t)$ and Eq.~\eqref{eq:meas-duality} holds,  then $\mu$ is called a \emph{symmetrizing measure} for $(X_t)$. When $(X_t)$ and $(Y_t)$ have right-continuous paths, we can replace the resolvents in eq.~\eqref{eq:meas-duality} by the semi-groups. 

This type of duality will not be in the focus of this paper, and we refer the reader to \cite{blumenthal-getoor, chung-walsh, getoor} and references therein for detailed accounts of this theory.

\begin{remark}[Time reversal] When $\mu$ is a probability measure and $(X_t)$ and $(Y_t)$ are Markov processes (not only sub-Markov), the previous notion coincides with the usual notion of time reversal with respect to a probability measure $\mu$. Similarly, in this case a symmetrizing probability measure is the same as a reversible measure. 
\end{remark}

\begin{remark}[Feynman-Kac corrections]
Definition \ref{def:duality1} can be generalized in the following way. Let $H:E_1\times E_2\to \R, F:E_1\to\R, G:E_2\to\R$ be bounded, measurable such that $\int_0^T|F(X_t)|\dd t<\infty, \int_0^T |G(Y_t)|\dd t<\infty,$ for $T>0.$ We say that $(X_t), (Y_t)$ are dual with respect to $(H,F,G),$ if for every $x\in E_1, y\in E_2,$
$$\E_x\left[\left| H(X_T,y)\exp\left(\int_0^T F(X_t)\dd t\right)\right|\right]<\infty,$$
$$\E^y\left[\left| H(x,Y_T)\exp\left(\int_0^T G(Y_t)\dd t\right)\right|\right]<\infty,$$
and
$$\E_x\left[ H(X_T,y)\exp\left(\int_0^T F(X_t)\dd t\right)\right]=\E^y\left[ H(x,Y_T)\exp\left(\int_0^T G(Y_t)\dd t\right)\right]$$
(see \cite{EthierKurtz86} for a discussion of such dualities in the context of Martingale problems). In the present article, however, we will not use this more general definition, and restrict ourselves to duality in the sense of \eqref{def:duality1}.
\end{remark}

\subsection{Examples}

We now give some examples of dual processes and typical duality functions, and hint at some applications of this concept. The list of examples is very far from being exhaustive. It is meant as a motivating illustration of the wide use of duality before addressing more theoretical questions. In order to keep the exposition short, we will restrict ourselves to simple examples that don't necessitate much notation, as for example interacting particle systems, and we don't try to find the most general setting for the various types of duality functions.

\begin{example}[Siegmund duality]
Assume $E=F$, and let $\leq$ be a partial order on $E.$ For $x,y\in E$ let $H(x,y):=\mathbf{1}_{\{x\leq y\}}.$ Two processes $(X_t), (Y_t)$ on $E$ are dual with respect to this duality function if and only if 
\begin{equation}
\label{def:siegmund_dual}
\P_x(X_0\leq Y_t)=\P^y(X_t\leq Y_0).
\end{equation}
Of course, exchanging the roles of $X$ and $Y$ we could equivalently choose $H(x,y)=\mathbf{1}_{\{x\geq y\}}.$ This is a classical duality and occurs in many contexts. For example, it was observed by L\'evy (1948) \cite{levy48}, that Brownian motion reflected at 0 and Brownian motion absorbed at 0 are dual with respect to this duality function. It was applied in different fields such as of queuing theory, \cite{Lindley, asmussen}, birth and death processes \cite{KarlinMcGregor}, or interacting particle systems \cite{CliffordSudbury85}. Siegmund \cite{Siegmund76} proved that it holds for stochastically monotone Markov processes (cf. Section \ref{sec:monotonicity}). It is therefore sometimes called ``Siegmund duality'', a name that we will also adopt throughout this paper for duality with respect to $\mathbf{1}_{\{x\leq y\}}$ or $\mathbf{1}_{\{x\geq y\}}.$ This type of duality is related to time reversal in a sense that it reverses the role of entrance and exit laws \cite{CoxRoesler84}, and to other forms of duality such as Wall duality \cite{Dette_etal}, or strong stationary duality \cite{DiaconisFill, Fill92} that will not be discussed here in detail.
\end{example}

\begin{example}[Coalescing dual, interacting particle systems] Duality has found a particularly wide application in the field of interacting particle systems \cite{Spitzer70, HolleyLiggett, Liggett05, Griffeath, Harris}. We focus here on a simple setup.
 Consider $E=\{0,1\}^G$ for some graph $G$.  A partial order on $E$ is given by 
$x_1\leq x_2\Leftrightarrow x_1(i)\leq x_2(i)\forall i\in G.$ Let $(X_t)$ denote the voter model on $G,$ that is the interacting particle system where a particle at site $i$ flips from 0 to 1 with rate given by the number of type 1 neighbours, and from 1 to 0 at rate given by the number of type 0 neighbours. It is well-known that the voter model is dual to a system of coalescing random walks in the following way: Let $F=\{A: A \mbox{ is a finite subset of } G\}.$ We consider the Markov process $(A_t)$ with values in $F$ with dynamics such that each $i\in A$ is removed at rate 1 and replaced by one of its neighbours $j\in G$ if $j\notin A.$ If $j\in A$, then the particles coalesce, that means, $i$ is removed from $A,$ and $j$ remains. Then $(X_t)$ and $(A_t)$ are dual with respect to
$$H(x,A)=\prod_{i\in A}(1-x_i),$$
see \cite{holley-stroock79}. Identifying $(X_t)$ and $(B_t)$ with $B_t=\{i\in G: X_t(i)=1\},$ this can be written as $H(A,B)=\mathbf{1}_{\{A\cap B=\emptyset\}}.$
Conversely $(A_t)$ can as well be  interpreted as a particle system $(Y_t),$ taking values in $E,$ by setting $Y_t(i)=1_{\{i\in A_t\}}.$ Then the duality function becomes
$$H(x,y)=\prod_{i:x_i=1}(1-y_i).$$
This is equivalent to 
\begin{equation}
\label{def:coal_dual}H(x,y)=\mathbf{1}_{\{x\wedge y=\mathbf{0}\}},
\end{equation}
where $x\wedge y$ denotes the component-wise minimum. All of these forms are usually referred to as \emph{coalescing duals}, cf. \cite{Liggett05, SudburyLloyd95}.\\
Note that this special case is actually equivalent to a Siegmund duality:
We have $x\wedge y=\mathbf{0}$ if and only if $x\leq \mathbf 1-y$ componentwise, since $x,y\in\{0,1\}^G.$ Therefore, for spin systems, $(X_t)_t$ and $(Y_t)_t$ are dual with respect to the function (\ref{def:siegmund_dual}) if and only if $(X_t)_t$ and $(\mathbf 1-Y_t)_t$ are dual with respect to (\ref{def:coal_dual}). 
\end{example}

\begin{example}[Moment duality]
 Take $E= [0,1]$ and $F=\N_0,$ and let 
$$H(x,n):=x^n.$$
For obvious reasons, dualities with respect to this function are called moment dualities. There are many examples, such as the classical duality between the Wright-Fisher diffusion and the block-counting process of Kingman's coalescent, which can be easily checked by applying Propostion \ref{prop:duality_generator} the generators given by $L^{WF} f(x) =\frac{1}{2}x(1-x) f''(x), x\in [0,1]$ (generator of the Wright-Fisher diffusion) and $L^{KC} f(n)= {n\choose 2}(f(n-1)-f(n)), n\in\N$ (generator of the block-counting process).\\
The importance of this kind of duality stems of course from the fact that a probability measure on $[0,1]$ is uniquely determined by its moments (Hausdorff moment problem). The notion of moment duality clearly makes sense for any finite subset $E$ of $\R.$  In Appendix B we provide an example of a moment duality with $E=[-1,1],$ compare also Definition \ref{def:q-moment}. Moment duality can, at least formally, be defined for any $E\subseteq \R.$ However, in general this duality might not determine the one-dimensional distributions of the process completely.\\
There are many connections between moment and coalescing duality. The coalescing duality of interacting particle systems can be cast in the form of a moment duality by writing, for $x=(x^i)_{i=1,...,N}, y=(y^i)_{i=1,...,N}\in\{0,1\}^N,$
\begin{equation}1_{\{x\wedge y=0\}}=\prod_{i=1}^N(1-x^i)^{y^i}.\end{equation}
In order to generalize this, we note that 
\begin{equation}\label{eq:coal_moment}\prod_{i=1}^N(1-x^i)^{y^i}=\prod_{i=1}^N(1-\mathbf 1_{\{x^i=y^i=1\}})=\prod_{i=1}^N\left(1-\int_{\{0,1\}}x^i\delta_{y^i}(\dd x^i)\right).\end{equation}
Let now
\begin{equation}
Z_t:=\frac{1}{N}\sum_{i=1}^N\delta_{ Y_t^i},
\end{equation}
and assume that both $(X_t)$ and $(Y_t)$ are \emph{exchangeable} for all $t\geq 0,$ that is, $\mathcal L (X_t^1,...,X_t^N)=\mathcal L(X^{\pi(1)}_t,...,X^{\pi(N)}_t)$ for all permutations $\pi$ of $\{1,...,N\}.$ Then we see from \eqref{eq:coal_moment} that the duality with respect to $H(x,y)=1_{\{x\wedge y=0\}}$ is equivalent to duality with respect to 
\begin{equation}
\tilde{H}(x,z):=\prod_{i=1}^N \int_{\{0,1\}}x^iz(\dd x^i)
\end{equation}
if $z=\frac{1}{N}\sum_{i=1}^N\delta_{y^i}.$ This generalizes moment dualities to the context of measure-valued processes, see for example \cite{DawsonKurtz, BertoinLeGall, DonnellyKurtz1, DawsonGreven, EthierKrone}. Moment dualities can be used to show uniqueness of solutions of Martingale problems, see \cite{EthierKurtz86}, section 4, which has been applied for example in the context of SPDEs, \cite{AthreyaTribe}.
\end{example}

\begin{example}[Annihilating dual]
Going back to the situation $E=F=\{0,1\}^G,$ we can choose 
$$H(x,y):=\mathbf{1}_{\{|x\wedge y|\mbox{ is odd}\}}=\frac{1}{2}(1-(-1)^{|x\wedge y|}),$$
or in the set-valued notation, $H(A,B)=\mathbf 1_{\{|A\cap B|\mbox{ odd }\}}.$
It is well-known that the voter model and annihilating random walk are dual with respect to this duality function \cite{Griffeath, SudburyLloyd95}. Related dualities are used for example in \cite{AthreyaSwart,JansenKurt}.
\end{example}

\begin{example}[Laplace dual] Here, the duality function is of the form $e^{-\langle x,y\rangle},$ where $\langle\cdot,\cdot\rangle$ denotes a scalar product, or more generally just a bilinear form. This duality is related to moment dualities, variants of it have been used in order to prove uniqueness of a solution of certain SPDE \cite{Mytnik1}. 
\end{example}

\section{Duality with respect to a measure} \label{sec:diagonal}

In this section we elucidate the relationship between duality in the sense of Def.~\ref{def:duality1} and duality in the sense of Def.~\ref{def:duality2}. 
The starting point is the following observation: if $(X_t)$ is a Markov chain with discrete state space and a reversible probability measure $\mu$, then $(X_t)$ is self-dual with duality function $H(x,y):= \mu(x)^{-1} \delta_{x,y}$, provided $\mu(x) >0$ for all $x\in E$.  This is the ``cheap'' self-duality function of~\cite{RedigEtAl09}, where it serves as the starting point for more interesting dualities. Propositions~\ref{prop:diagonal-duality} and \ref{prop:diagonal-polish} address the situation when the measure $\mu$  does not have full support or the state space is not discrete. We refer the reader to \cite[Section 5]{DiaconisFill}, \cite[Section 5.1]{carmona-petit-yor98}, and \cite{krs07} for further comparisons of dualities.

We start with an example in discrete state spaces. Fix $L\in \N$ and consider a simple random walk on $\{0,1,\ldots,L+1\}$ with absorption at $0$ and $L+1$. Any reversible measure $\mu$ must be invariant, hence satisfies $\mu(x) = 0$ for $x=1,\ldots,L$. On the other hand, a straight-forward computation shows that the diagonal function 
$H(x,x)=1$ if $x=1,\ldots,L$, $H(0,0)=H(L+1,L+1)=0$, and $H(x,y) =0$ if $x\neq y$, is a self-duality function. Clearly $H(x,x)$ is not of the form $1/\mu_\mathrm{rev}(x)$ for a reversible measure $\mu_\mathrm{rev}$. Nevertheless, the measure $\mu(x) = 1/H(x,x) =1$ on $\{1,\ldots,L\}$ is a duality measure for a sub-Markov process, namely, the simple random walk killed at the boundary. Note that a duality measure for strictly sub-Markov kernels need not be invariant, but only \emph{excessive}, $(\mu P)(y) \leq \mu(y)$. 

\begin{remark}
The previous example should be seen in the original context of the notion of duality with respect to a measure~\cite{hunt57}: the initial motivation came from potential theory, where   Markov processes  killed upon reaching the boundary of the domain play an important role.
\end{remark}

The next proposition generalizes the relation explained in the previous example. 
The key subtlety concerns the invariance of subsets of the state space: if $\mu$ is a reversible measure of a Markov process $(X_t)$, then $(X_t)$ cannot leave $\supp \mu$ -- it can, however, jump from $E\backslash \supp \mu$ into $\supp \mu$. The situation is the other way round if $H(x,x)$ is a diagonal self-duality function of $(X_t)$: in this case $(X_t)$ may very well leave $\supp H$, but it cannot leave $E\backslash \supp H$, which is therefore a trap for $(X_t)$. 

\begin{prop}\label{prop:diagonal-duality}
	Let $(X_t)$ and $(Y_t)$ be Markov processes with identical discrete state space $E$. 
	\begin{enumerate} 
		\item Suppose that $(X_t)$ and $(Y_t)$ are dual with respect to a diagonal duality function $H(x,x)$. Then $|H(x,x)|$ is a duality function too. Define the measure 
		\begin{equation} \label{eq:muh}
			\mu(x):= \begin{cases} 
					1/|H(x,x)|, &\quad H(x,x)\neq 0,\\ 
					0, & \quad H(x,x) =0.
				\end{cases}
		\end{equation} 
		Then $E\setminus \supp \mu$ is a trap for $(X_t)$ and $(Y_t)$, and the sub-Markov processes 
		\begin{equation} \label{eq:hatxt}
			\hat X_t:= \begin{cases} 
					X_t,&\quad X_t \in \supp \mu,\\ 
					\dagger, &\quad X_t \notin \supp \mu, 
				\end{cases}, \qquad 
		\hat Y_t:= \begin{cases} 
					Y_t,&\quad Y_t \in \supp \mu,\\ 
					\dagger, &\quad Y_t \notin \supp \mu.  
				\end{cases} 
		\end{equation} 
		are in duality with respect to the measure $\mu$. 

		\item 	Conversely, let $\mu$ be a measure on $E$ and let $H(x,x) \geq 0$ be the diagonal function given by Eq.~\eqref{eq:muh}. Suppose that $E\setminus \supp \mu$ is a trap for $(X_t)$ and $(Y_t)$, and $(\hat X_t)$ and $(\hat Y_t)$ defined as in Eq.~\eqref{eq:hatxt} are in duality with respect to $\mu$. Then $(X_t)$ and $(Y_t)$ are dual with respect to $H$. 
	\end{enumerate} 
\end{prop}

\begin{proof}
	We start with the proof of 2. 
	We need to check that for all $x,y\in E$ and all $t\geq 0$, 
	\begin{equation} \label{eq:hdidu}
		P_t(x,y)H(y,y) = Q_t(y,x) H(x,x).
	\end{equation}
	The weak duality of $(\hat X_t)$ and $(\hat Y_t)$ with respect to the measure $\mu$ tells us that 
	\begin{equation*} 
		\forall x,y\in \supp \mu:\ \mu(x)P_t(x,y) = \mu(y) Q_t(y,x). 
	\end{equation*} 
	Dividing by $\mu(x)$ and $\mu(y)$ on both sides, we find that Eq.~\eqref{eq:hdidu} holds when $x$ and $y$ are both in $\supp \mu$. If $x$ and $y$ are both outside $\supp \mu$, Eq.~\eqref{eq:hdidu} obviously holds since in this case $H(x,x) = H(y,y)=0$. 
	If $x\in \supp \mu$ but $y\notin \supp \mu$, then $H(y,y)=0$ and, because $(Y_t)$ cannot leave $E\setminus \supp \mu$, $Q_t(y,x) =0$. Thus Eq.~\eqref{eq:hdidu} holds. The symmetric case $x\notin \supp \mu$, $y \in \supp \mu$ can be treated in a similar way. 
	
	Proof of 1.: 
	We know that Eq.~\eqref{eq:hdidu} holds for all $x,y\in E$. Therefore, if $H(y,y) = 0$ but $H(x,x) \neq 0$, we have $Q_t(y,x) =0$. Thus $(Y_t)$ cannot leave $E\setminus \supp \mu = \{y\mid H(y,y) =0\}$ and $(\hat Y_t)$ gives a well-defined Sub-Markovian process; the analogous statement for $(X_t)$ follows in a similar way. Moreover, when $x$ and $y$ are both in $\supp \mu$, we obtain from Eq.~\eqref{eq:hdidu} that 
	\begin{equation} \label{eq:hdu}
		\forall x,y\in \supp \mu:\  H(x,x)^{-1} P_t(x,y) = H(y,y)^{-1} Q_t(y,x)
	\end{equation}
	which proves the claim when $H$ is non-negative. 
	Next, suppose that there are $x$ and $y$ such that $H(x,x)H(y,y)<0$, i.e., $H(x,x)$ and $H(y,y)$ have opposite signs. Then Eq.~\eqref{eq:hdidu} shows that for all $t$, $P_t(x,y)=0$ and $Q_t(y,x) =0$. 
	Put differently, in Eq.~\eqref{eq:hdidu} either both sides have the same sign or both sides vanish; as a consequence we can take absolute values and deduce that $|H(x,x)|$ is a diagonal duality function. 
\end{proof} 

Next let us look at non-discrete state spaces. In this case expressions such as $1/\mu(x)$ do not make sense and it is not clear what a good diagonal function $H$ would be. A way around this issue is to replace the candidate diagonal duality function by a family of functions $H_\lambda$ such that, formally, as $\lambda\to \infty$, the family converges to the correct object -- much in the same way as a Dirac measure can be approximated by a family of Gaussian functions with variances going to~$0$. To this aim recall the following: in finite state spaces, for strongly continuous Markov chains, the resolvent $R_\lambda$ satisfies $\lambda R_\lambda \to \id$ as $\lambda \to \infty$. Thus if we set $H_\lambda(x,y) = \lambda R_\lambda(x,y)/\mu(y)$, then $H_\lambda \to \mathrm{diag}(1/\mu(x))$. For non-discrete state spaces, $1/\mu(x)$ does not make sense as a function, but $R_\lambda(x,y)/\mu(y)$ can be interpreted as a Radon-Nikod{\'y}m derivative: if $(X_t)$ and $(\hat X_t)$ are in duality with respect to $\mu$, there is a function $r_\lambda(x,y)$ such that the resolvents satisfy
\begin{equation} \label{eq:rlambda} 
	R_\lambda(x,\dd y) = r_\lambda(x,y) \mu(\dd y),\quad \hat R_\lambda(x,\dd y) = 	 r_\lambda(y,x) \mu(\dd y). 
\end{equation} 
for all $x,y \in E$. Moreover we may assume that for all $\lambda$ the functions $x\mapsto r_\lambda(x,y)$ 
and $y\mapsto r_\lambda(x,y)$ are $\lambda$-excessive for the resolvents $(R_\alpha)_{\alpha>0}$ and $(\hat R_\alpha)_{\alpha>0}$ respectively,\footnote{A non-negative measurable function $f$ is $\lambda$-excessive for $(R_\alpha)$ if (i) for all $\beta>0$, $\beta R_{\lambda + \beta} f \leq f$, and (ii) $\beta R_{\lambda+ \beta} f \to f$ as $\beta \to \infty$ pointwise on $E$, see e.g.~\cite{blumenthal-getoor}[Chapter III.4].} and under this assumption the functions $r_\lambda(x,y)$ are uniquely determined by the resolvents \cite{blumenthal-getoor}[Chapter VI.1].


\begin{prop} \label{prop:diagonal-polish}
	Let $\mu$ be a $\sigma$-finite measure on $E$. Let $(X_t)$ and $(\hat X_t)$ be Markov processes with c{\`a}dl{\`a}g paths. Suppose that $(X_t)$ and $(\hat X_t)$ are in duality with respect to $\mu$ and let $r_\lambda:E\times E\to [0,\infty)$ be the unique function that satisfies Eq.~\eqref{eq:rlambda} and turns $x\mapsto r_\lambda(x,y)$ and $y\mapsto r_\lambda(x,y)$ into $\lambda$-excessive functions for $(R_\alpha)$ and $(\hat R_\alpha)$ respectively. Then 
	\begin{equation} \label{eq:rlambda-dual}  
		\E_x [r_\lambda(X_t,y)] = \E^y[ r_\lambda(x,\hat X_t)]
	\end{equation} 
	for all $\lambda>0$, $x,y \in E$ and $t>0$, i.e., every $r_\lambda$ is a duality function for $(X_t)$ and $(Y_t)$. 
\end{prop}

The proof uses that  two $\lambda$-excessive functions that are equal $\mu$-almost everywhere are in fact equal everywhere~\cite{blumenthal-getoor}\footnote{We thank the anonymous referee for pointing out the relevance of the arguments in~\cite{blumenthal-getoor}[Chapter VI].}. It is analogous to the proof of Theorem~1.16 in~\cite{blumenthal-getoor}[Chapter VI.1], which shows that  under additional regularity assumptions, that in Eq.~\eqref{eq:rlambda-dual} we may replace $t$ by the hitting times $T_B$ and $\hat T_B$ of Borel sets $B\subset E$.

\begin{proof}[Proof of Prop.~\ref{prop:diagonal-polish}]
	We note first that if the processes have c{\`a}dl{\`a}g paths, then for all $t>0$ and all non-negative, bounded functions $f$ and $g$, 
	\begin{equation} \label{eq:pt-dual} 
		\int_E (P_t f)(x) g(x) \mu(\dd x) = \int_E f(x) (P_t g)(x) \mu(\dd x). 
	\end{equation} 
	Indeed, let $L(t)$ and $R(t)$ be the left and right-hand sides of this equation. The definition of duality with respect to $\mu$ says that $L$ and $R$ have the same Laplace transforms. As a consequence, $L(t) = R(t)$ for Lebesgue-almost all $t>0$. Now, since $(X_t)$ has c{\`a}dl{\`a}g paths, dominated convergence shows that for bounded continuous  functions $f$, and every $x\in E$, the map $[0,\infty) \ni t \mapsto (P_t f)(x) = \E_x f(X_t)$ is c{\`a}dl{\`a}g as well; similarly for $(\hat P_t g)(x)$. Thus for non-negative, bounded, continuous $f$ and $g$, $L(t)$ and $R(t)$ are c{\`a}dl{\`a}g, and equality Lebesgue-almost everywhere implies equality for all $t>0$.  Eq.~\eqref{eq:pt-dual} holds for continuous $f$ and $g$ and extends to bounded non-negative $f$ and $g$ by a density argument. 
	
	Now let $f,g\in L^\infty(E)$, non-negative. We integrate Eq.~\eqref{eq:rlambda-dual} against $f(x) g(y) \mu(\dd x) \mu(\dd y)$. The left-hand side becomes, using $r_\lambda(x,y) \mu(\dd y) = R_\lambda(x,\dd y)$, 
	\begin{align*} 
		\int P_t(x,\dd x') r_\lambda(x',y) f(x) g(y) \mu(\dd y) 
			& = \int\mu(\dd x) f(x)  P_t(x,\dd x') R_\lambda( x',\dd y) g(y)  \\
			& = \int  \mu(\dd x) f(x) \bigl( P_t R_\lambda g)(x)=\la f, P_t R_\lambda g\ra_{L^2(E,\mu)}. 
	\end{align*} 
	Similarly, the right-hand side becomes, using $r_\lambda(x,y) \mu(\dd x) = \hat R_\lambda(y, \dd x)$ and the duality of the semi-groups with respect to $\mu$, 
	\begin{equation*} 
		\int \mu( \dd y) g(y) \bigl( \hat P_t \hat R_\lambda f)(y) = \la \hat P_t \hat R_\lambda f, g \ra_{L^2(E,\mu)} = \la f, R_\lambda P_t g\ra_{L^2(E,\dd \mu)}. 
	\end{equation*}
	Thus the left-hand side and the right-hand side of Eq.~\eqref{eq:rlambda-dual}, integrated against $f(x) g(y) \mu(\dd x) \mu(\dd y)$, are equal, for all non-negative $f,g\in L^\infty(E)$. It follows that for all non-negative $f$ and $\mu$-almost all $x\in E$, 
\begin{equation*}
	 \int_E \Bigl( \int_E  P_t(x,\dd x') r_\lambda(x',y)\Bigr)  g(y) \mu(\dd y) 
	= \int_E \Bigl( \int_E P_t(y,\dd y') r_\lambda(x,y')  \Bigr) g(y) \mu(\dd y).
\end{equation*}
Both sides are $\lambda$-excessive functions of $x$, therefore the functions agree for all $x$. It follows that Eq.~\eqref{eq:rlambda-dual} holds for all $x$ and $\mu$-almost all $y$; the identity extends to all $y\in E$ by using again $\lambda$-excessivity. 
\end{proof}

\section{Functional analytic theory} \label{sec:fa}

This section presents an analytic framework for duality of Markov processes. In Prop.~\ref{prop:bilin}, duality is recast as duality of operators with respect to a bilinear form; convex geometry enters in Section \ref{sec:cones}. As an application, we present abstract criteria for the existence and uniqueness of a dual to a given process with respect to a given duality function $H$. The main results are Propositions ~\ref{prop:uniqueness} and~\ref{prop:gen-form}, and Theorems~\ref{prop:non-deg} and  \ref{prop:lift}. We also give a  criterion under which the dual of a Feller semi-group is a Feller semi-group (Theorem~\ref{thm:feller}) and show that reversible processes with non-degenerate duality functions have the same spectrum (Theorem~\ref{thm:unitary}).

All results of this section can be formulated in terms of the semi-groups only, and all Markov semi-groups will be on Polish spaces equipped with their Borel $\sigma$-algebra. By Markov semi-group we mean a family of kernels $P_t(x,A)$, $t\geq 0$, such that for all $x$, $P_0(x,\cdot) = \delta_x$,  $x\mapsto P_t(x,A)$ is Borel-measurable for all measurable $A$, $P_t(x,\cdot) =1$ is a probability measure, and the Chapman-Kolmogorov equations hold. 
We do not assume any additional regularity such as existence of realizations with c{\`a}dl{\`a}g paths. We write $\meas(E)$ for bounded signed measures and $\prob(E)$ for probability measures on $E$. 

\subsection{Duality of semi-groups with respect to a bilinear form. Uniqueness}

The natural setting for duality of Markov processes are  dual pairs as encountered in the treatment of weak topologies and locally convex spaces \cite[Chapter V.7]{reed-simon}.

\begin{definition} [Left and right null spaces, non-degenerate bilinear form]
	Let $V,W$ be vector spaces and $B:V\times W \to \R$ a bilinear form. 
	The \emph{left} and \emph{right null spaces} of $B$ are the subspaces  
	\begin{align*}
		\mathcal{N}_\mathrm{L} = \{ x \in V \mid B(x,\cdot) = 0 \}, \quad 
		\mathcal{N}_\mathrm{R} = \{ y \in W \mid B(\cdot,y) = 0 \}. 
	\end{align*}
	The form $B$ is called \emph{non-degenerate} if $\mathcal{N}_\mathrm{L} = \{0\}$ 
	and $\mathcal{N}_\mathrm{R}=\{0\}$. 
\end{definition} 

When $B(\cdot,\cdot)$ is non-degenerate, it is common to use a scalar product notation $B(\cdot,\cdot) = \la \cdot,\cdot\ra$, and the triple $(V,W,\la\cdot,\cdot\ra)$ is referred to as a \emph{dual pair}. 

\begin{definition}[Duality of operators with respect to a bilinear form] 
	Let $V,W$ be vector spaces, $B:V\times W \to \R$ a  bilinear form 
	and $T:V \to V$, $S: W\to W$ linear operators. 
	Then $T$ is called \emph{dual} to $S$ with respect to $B$ if 
	\begin{equation*}
		\forall f \in V,\, g\in W:\ B(Tf, g) = B(f,S g). 
	\end{equation*}
\end{definition}


The dual with respect to a non-degenerate form, if it exists, is unique. Standard examples include the usual dual spaces, Hilbert spaces, and the pairing between measures and functions.

\begin{example}[``The'' dual] Let $X$ be a Banach space and $X'$ the dual space, i.e., the space of continuous linear functionals from 
	$X$ to $\R$. Then $\la \varphi,x\ra := \varphi(x)$ defines a non-degenerate bilinear form on $X'\times X$. Every bounded operator $S:W\to W$ has a dual operator $S'$, often simply called ``the'' dual, sometimes also the \emph{Banach space adjoint}. 
\end{example}

\begin{example}[Adjoint operator in a Hilbert space] 
	Let $\mathcal{H}$ some real Hilbert space. The scalar product $\la \cdot,\cdot\ra$ is a non-degenerate bilinear form on $\Hi \times \Hi$, and every bounded operator $A$ has a unique dual operator $A^*$, the (Hilbert space) \emph{adjoint} of $A$. 
\end{example}

\begin{example}[Pairing of measures and functions]
	Let $E$ be a Polish space, endowed with its Borel $\sigma$-algebra, $L^\infty(E)$ the bounded measurable functions  and $\meas(E)$ the space of finite signed measures. We equip $L^\infty(E)$ with the supremum norm $||\cdot||_\infty$ and the signed measures with the total variation norm 
	\begin{equation} \label{eq:totvar}
		||\mu||:= \sup_A \Bigl( |\mu(A)| +|\mu(E\backslash A)|\Bigr) .
	\end{equation}
	 Then $\la \mu, f\ra:= \int_E f\dd \mu$ defines a non-degenerate bilinear form on $\meas(E)\times L^\infty(E)$. Let $P(x,A)$ be a transition kernel, acting on functions via $(Pf)(x) = \int_E P(x,\dd x') f(x')$ and on measures via $(P^* \mu)(A) = (\mu P)(A) = \int_E \mu(\dd x) P(x,A)$. $P$ and $P^*$ are dual with respect to $\la \cdot,\cdot\ra$. 
\end{example}


\begin{prop}[Duality of Markov processes as duality of operators] \label{prop:bilin}
	Let  $(X_t)$, $(Y_t)$ be Markov processes with Polish state spaces $E$,$F$, and 
	semi-groups 	$(P_t)$, $(Q_t)$. Let $H:E\times F \to \R$ be a bounded measurable function 
	and 
	\begin{equation*}
		B_H(\mu,\nu):= \int_{E\times F} H(x,y) \mu(\dd x) \nu(\dd y).  
	\end{equation*}
	Then $(X_t)$ and $(Y_t)$ are dual with respect to  $H$ if and only if for all 
	$t >0$, 
	\begin{equation*}
		\forall \mu \in \mathfrak{M}(E),\ \forall \nu \in \mathfrak{M}(F):\ 
			B_H(P_t^* \mu,\nu) = B_H(\mu, Q_t^* \nu). 
	\end{equation*}
	i.e., if and only if for all $t>0$, $P_t^*$ and $Q_t^*$ are dual with respect to $B_H$. 
\end{prop}

\begin{proof}
	By definition, $\int f(x) (P_t^*\mu)(\dd x) = \E_\mu f(X_t)= \int [\E_x f(X_t)] \mu(\dd x)$. 
	If the processes are dual with respect to $H$, then 
	\begin{align*}
		B_H(P_t^*\mu,\nu) & =\int H(x,y) (P_t^*\mu) (\dd x) \nu(\dd y) \\
			& = \int \Bigl( \E_x H(X_t,y) \Bigr)  \mu(\dd x) \nu (\dd y) \\
			& = \int \Bigl( \E^y H(x, Y_t) \Bigr) \mu (\dd x) \nu (\dd y) \quad \text{by 
				Eq.~\eqref{eq:duality-def} }\\
			& = \int H(x,y) \mu(\dd x) (Q_t^* \nu) (\dd y) = B_H(\mu, Q_t^* \nu).  
	\end{align*}
	Conversely, if the semi-groups are dual with respect to the bilinear form $B$, 
	then for $\mu = \delta_x$ , $\nu = \delta_y$, 
	\begin{equation*} 
		\E_x H(X_t, y) = B_H( P_t^* \delta_x, \delta_y) = B_H(\delta_x,Q_t^* \delta_y) 
			= \E^y H(x,Y_t). \qedhere
	\end{equation*}
\end{proof}

Most of the standard duality functions have non-degenerate associated bilinear forms. 

\begin{prop} \label{lem:non-degenerate}
The following duality functions $H:E\times F\to \R$ give rise to non-degenerate bilinear forms $B_H$: 
	\begin{enumerate}\itemsep0pt
		\item Siegmund duality: $E = F = \R$, $H(x,y):= \mathbf{1}_{\{x\leq y\}}$.  
		\item Laplace duality: $E= F = [0,\infty)$, $H(x,\lambda):= \exp(-\lambda x)$. 
		\item Moment duality: $E=[0,1]$, $F=\N_0$, 	$H(x,n) = x^n$.  
		\item Coalescing dual: $\Lambda$ a finite set, $E=F=\mathcal{P}(\Lambda)$ power set, 
		 $H(A,B) = \mathbf{1}_{\{A\cap B= \emptyset\}}$. 
		\item Annihilating dual: $E=F= \mathcal{P}(\Lambda)$, $H(A,B) = (-1)^{|A\cap B|}$. 
	\end{enumerate}	 
\end{prop} 

Prop~\ref{lem:non-degenerate} is proven in Appendix~\ref{app:standard}. The proposition is of interest because non-degeneracy implies uniqueness of the dual. Another closely related condition for uniqueness is that the family of functions $H(x,\cdot)$, $x\in E$ separates $\prob(F)$, see \cite{EthierKurtz86}; in \cite{Swart}, the duality is called \emph{informative} if both $H(x,\cdot)$, $x\in E$ and $H(\cdot,y)$, $y \in F$ are separating (in $\prob(F)$ and $\prob(E)$, respectively).

\begin{definition} 
	Let $E$ be a Polish space and $\mathcal{X} \subset L^\infty(E)$. We call $\mathcal{X}$ \emph{separating} if $\mathcal{X}$ \emph{separates} $\prob(E)$, i.e., 
	if two probability measures $P,Q\in \prob(E)$ satisfy $\int_E g \dd P = \int_E g \dd Q$ for all $g \in \mathcal{X}$, then $P=Q$. 	
\end{definition} 

We have the following implications. 

\begin{prop} \label{prop:uniqueness} 
	Let $(P_t)$ be a Markov semi-group in $E$. Suppose that one of the following conditions holds:
	\begin{enumerate} \itemsep0pt
		\item The family of functions $H(x,\cdot)$, $x\in E$, is separating.
		\item The family of functions $\int_E \mu(\dd x) H(x,\cdot)$, $\mu \in \prob(E)$, is separating.
		\item The right null space of $B_H$ is $\{0\}$. 
		\item $B_H$ is non-degenerate. 
	\end{enumerate} 
	Then the dual semi-group, if it exists, is unique. 
\end{prop} 
\noindent We note the implications $(4)\Rightarrow (3) \Rightarrow (2)$,  $(1) \Rightarrow (2)$, $(2)\Rightarrow$ uniqueness, and leave the elementary proof to the reader.

Despite Proposition~\ref{lem:non-degenerate}, not all duality functions are associated with non-degenerate bilinear forms. The following proposition describes a situation which is typical for pathwise duality, where processes often initially live on bigger spaces and we construct $(X_t)$ and $(Y_t)$ by forgetting a part of the information, see Section \ref{sec:pathwise}. 

\begin{prop} \label{lem:lift-deg}
	Let $(X_t)$ and $(Y_t)$ be two Markov processes with Polish state spaces $E$ and $F$, dual with respect to  $H:E\times F\to \R$. 
	Suppose that $(Y_t)$ can be lifted to a bigger space, i.e., there is a Polish space $G$, a measurable map $\pi:G\to F$ that is surjective but not injective, and a Markov process $(Z_t)$ on $G$ such that 
	$(\pi(Z_t))_{t\geq 0}$ is equivalent to  $(Y_t)$. Let $\tilde H(x,z):= H(x,\pi(z))$.  Then $(Z_t)$ is dual to $(X_t)$ with duality function $\tilde H$, and $B_{\tilde H}$ is degenerate. 
\end{prop} 
Here ``equivalence'' means that whenever $\pi(Z_0)$ and $Y_0$ have the same distribution, all finite-dimensional distributions of $\pi(Z_t)$ and $(Y_t)$ agree.

\begin{proof}
	Since $\pi$ is not injective, there are $z$ and $z'$ such that $z\neq z'$ and  $\pi(z) =\pi(z')$. It follows that $\delta_z - \delta_{z'}$ is in the right null space of $B_{\tilde H}$, and $B_{\tilde H}$ is degenerate. 
	Moreover, for every $z\in G$ and $x\in E$, we have 
	\begin{equation*}
		\E_x \tilde H(X_t,z) = \E_x H(X_t,\pi(z)) = \E^{\pi(z)} H(x,Y_t) = \E^z \tilde H(x,Z_t),
	\end{equation*}
	hence $(X_t)$ and $(Z_t)$ are dual with duality function $\tilde H$. 
\end{proof} 

We will come back to this situation in Section~\ref{sec:cones}, where we will characterize dualities that can be obtained as stochastic lifts of non-degenerate dualities. 

\subsection{Existence. Feller semi-groups} 

As mentioned in the previous section, any bounded operator $T$ in a Banach space has a unique dual operator $T'$, and any bounded operator in a Hilbert space has a unique adjoint operator $T^*$. For general dual pairs $(V,W,\la \cdot,\cdot\ra)$, however, it might be difficult to determine whether a given concrete operator has a dual.

The  existence of a dual Markov process adds a layer of difficulty, as it is not enough to ask whether $P_t^*$ has a $B_H$-dual \emph{operator} $T_t$ in $\meas(F)$: we also need to know whether the dual operator is of the form $T_t \mu = \int \mu(\dd x) Q_t(x,\cdot)$ with $Q_t(x,A)$ the transition kernels of some Markov process.
As it turns out, the existence of a dual \emph{operator} semi-group is tied to the invariance of some \emph{linear} subspace $\mathcal{V}$; the existence of a dual \emph{Markov} semi-group is tied to the stronger requirement that some \emph{convex} subset $\mathcal{V}_{1,+} \subset \mathcal{V}$ be invariant under $(P_t)$ \cite{krs07,moehle-cones}. 

\begin{definition}  \label{def:vwspaces}
	Fix $H:E\times F\to \R$ bounded and measurable. We define 
	\begin{align*}
		 \mathcal{V}_{1,+} & := \bigl \{ f \in L^\infty (E) \big|\,  \exists \nu \in \prob(F):\,  \forall x\in E\  f(x) = \int_E H(x,y) \nu(\dd y) \bigr \} \\
		\mathcal{W}_{1,+} & := \bigl\{ g \in L^\infty(F)\big|\, \exists \mu \in \prob(E):\ \forall y\in F\  g(y) = \int_E H(x,y) \mu(\dd x)  \}.
	\end{align*} 
	The spaces $\mathcal{V}$ and $\mathcal{W}$ are defined in a similar way, replacing $\prob(E)$ and $\prob(F)$ by $\meas (E)$ and $\meas(F)$ respectively. 
\end{definition} 
Thus $\mathcal{V}_{1,+} \subset \mathcal{V}\subset L^\infty(E)$ and $\mathcal{W}_{1,+} \subset \mathcal{W}\subset L^\infty(F)$. Clearly $\mathcal{V}_{1,+}$ and $\mathcal{W}_{1,+}$ are convex, and $\mathcal{V}$ and $\mathcal{W}$ are linear subspaces that arise as the linear hulls of $\mathcal{V}_{1,+}$ and $\mathcal{W}_{1,+}$. When $E$ and $F$ are finite state spaces, we may think of $H$ as a matrix, and the four spaces correspond to linear and convex combinations of the rows and columns of the matrix. 

We have the following necessary condition for the existence of a dual. 
\begin{prop}\label{prop:gen-form}
	Let $E$ and $F$ be Polish state spaces, $H:E\times F\to \R$ measurable and bounded, and $(P_t)$ a Markov semi-group in $E$. Suppose that $(P_t)$ has a Markov semi-group in $F$ dual with respect to $H$. Then $(P_t)$ leaves $\mathcal{V}_{1,+}$ and $\mathcal{V}$ invariant, and any dual semi-group leaves $\mathcal{W}_{1,+}$ and $\mathcal{W}$ invariant. 
\end{prop} 

\begin{proof}
	Let $(Q_t)$ be a dual Markov semi-group and $f(\cdot) = \int_E H(\cdot,y)\nu(\dd y) \in \mathcal{V}_{1,+}$, $\nu \in \meas(F)$. Then, for all $t>0$ and $x\in E$,  
	\begin{equation*}
		(P_t f)(x) = B_H(P_t ^*\delta_x,\nu) =B_H( \delta_x, Q_t^* \nu) = 
			\int_F H(x,y) (Q_t^* \nu) (\dd y) 
	\end{equation*}
	thus $P_t f \in \mathcal{V}_{1,+}$ and $\mathcal{V}_{1,+}$ is invariant under $(P_t)$. 
	Since $\mathcal{V}$ is the linear hull of $\mathcal{V}_{1,+}$,  the invariance of $\mathcal{V}$ follows from the invariance of $\mathcal{V}_{1,+}$. 
	We can invert the roles played by $E$ and $F$, $(P_t)$ and $(Q_t)$, and apply what we have just proven to $(Q_t)$: this gives that any dual $(Q_t)$ must leave $\mathcal{W}_{1,+}$ and $\mathcal{W}$ invariant. 
\end{proof}

A list of spaces $\mathcal{V}_{1,+}$ and $\mathcal{V}$ for some common duality functions is given in Tables \ref{table:v1} and \ref{table:v}. We note that invariance of $\mathcal{V}$ is related to regularity properties (e.g., does the semi-group map continuous functions to continuous functions?), while the invariance of $\mathcal{V}_{1,+}$ is sometimes associated with monotonicity properties. We shall come back to the latter aspect in Section~\ref{sec:monotonicity}. 

\begin{remark}
For finite state spaces and non-degenerate duality function, $\mathcal{V}$ is the image of $\R^F$ under an invertible matrix, hence $\mathcal{V}=\R^E$. As a consequence, $\mathcal{V}$ is automatically invariant. A trick around the invariance of $\mathcal{V}_{1,+}$ is to define the dual Markov process in an artificially doubled state space $\tilde F$, where it \emph{always} exists, see e.g. the paragraph after Eq. (2.3) in \cite{holley-stroock79}.
\end{remark}

\begin{table}
	\centering
	\begin{tabular}{cccl}
		\hline\hline
		$H(x,y)$ & $x\in E$ & $y \in F$ & $\mathcal{V}_{1,+} = \{\int_F H(\cdot,y) \nu(\dd y) \mid \nu \in \prob(F)\}$ \\
		\hline \hline
		$ \mathbf{1}_{\{x \geq y\}}$ & $x\in \R$ & $y \in \R$ & monotone increasing, right-continuous  \\
		& & & functions $f$ with $\lim_{-\infty} f =0$ and $\lim_{+\infty} f = 1$ \\
		\hline
	$x^n$ & $x \in [0,1]$ & $n \in \N_0$  & absolutely monotone functions $f$ with \\
		 &&& $f(1) =1$ \\
		 \hline
		 $y^n$ & $n \in \N_0$ & $y \in [0,1]$ & completely monotone sequences with $f(0)=1$\\
		 	\hline
		 	$\exp(-xy)$ & $ x\in [0,\infty)$ & $y \in [0,\infty)$ & completely monotone functions with $f(0)=1$\\
		 	\hline
	\end{tabular} 
	\vspace{.2cm}
	\caption{\small  \label{table:v1}  List of convex subsets $\mathcal{V}_{1,+} \subset L^\infty(E)$ associated with some common duality functions. For definitions of absolute and complete monotonicity, see Sections~VII.2, VII.3 and XIII.4 in \cite{feller2}.}  
\end{table}

\begin{table} 
	\centering
	\begin{tabular}{cccl}
		\hline\hline
		$H(x,y)$ & $x\in E$ & $y \in F$ & $\mathcal{V}= \{\int_F H(\cdot,y) \nu(\dd y) \mid \nu \in \meas(F)\}$ \\
		\hline \hline
		$ \mathbf{1}_{\{x \geq y\}}$ & $x\in \R$ & $y \in \R$ & bounded, right-continuous functions $f$ with  \\
		& & &  bounded variation, $\lim_{-\infty} f =0$, and $\lim_{+\infty} f \in \R$ \\
		\hline
	$x^n$ & $x \in [0,1]$ & $n \in \N_0$  & functions $f\in C([0,1])$ with an analytic extension  \\
		 &&& to the complex open unit disk that is in the Wiener\\
		 &&& algebra; subset of a $H^\infty$ Hardy space.\\
		 \hline
		 $\exp(-xy)$ & $ x\in [0,\infty)$ & $y \in [0,\infty)$ & subset of the functions with a bounded continuous \\
		 &&& extension $F$ to the complex half-plane $\Re z \geq 0$,\\
		 &&& analytic in $\Re z>0$\\
		 \hline
	\end{tabular} 
	\vspace{.2cm}
	\caption{\small \label{table:v}  List of linear subspaces $\mathcal{V} \subset L^\infty(E)$ associated with some common duality functions; see \cite{hoffman} for more on the Wiener algebra and Hardy spaces.}  
\end{table}

When both space and time are discrete, the invariance of $\mathcal{V}_{1,+}$ is actually sufficient for the existence of a dual.  Recall that in the discrete case duality is equivalent to 
\begin{equation} \label{eq:didu}
	\forall x\in E \ \forall y\in F:\ \sum_{x'\in E} P(x,x') H(x',y) = \sum_{y'\in F} Q(y,y') H(x,y'), 
\end{equation}
where $P$ and $Q$ are the transition matrices of Markov chains. The signed measures on $F$ correspond to $\ell^1(F)$, and $\mathcal{V}$ becomes the image of $\ell^1(F)$ under the linear map with matrix $H$. 

\begin{prop}[Discrete time and discrete state space] \label{prop:discrete} 
	Let $E$ and $F$ be countable spaces, endowed with the discrete topology, and $H:E\times F\to \R$ bounded. Let $P$ be a stochastic $E\times E$-matrix. Then 
	\begin{enumerate} 
		\item The problem~\eqref{eq:didu}			
			admits a solution $(Q(y,z) )_{y,z\in F}$ with $\sum_{z\in F} |Q(y,z)|<\infty$ 
		for all $y\in F$ if and only if $\mathcal{V}$ is invariant under $P$. 
		\item $Q$ can be chosen as a stochastic matrix if and only if $\mathcal{V}_{1,+}$ is invariant under $P$.    
	\end{enumerate} 
\end{prop}

\begin{proof}
	1. The necessity of the invariance of $\mathcal{V}$ is proven as in Prop.~\ref{prop:gen-form}. 
	For the sufficiency, given $y \in F$, let $h_y \in \mathcal{V}$ be the function  $h_y(x):= H(x,y)$ and let $(Q(y,z))_{z\in F}$ be a vector in $\ell^1(F)$ such that 
	$P h_y = \sum_{z\in F} Q(y,z) h_z$. Such a vector exists because of the invariance of $\mathcal{V}$ under $P$, and one can check that the matrix $Q$ defined in this way solves Eq.~\eqref{eq:didu}. 

	2. The necessity of the invariance of $\mathcal{V}_{1,+}$ is proven as in Prop.~\ref{prop:gen-form}. For the sufficiency, we proceed as in 1., noting that $h_y\in \mathcal{V}_{1,+}$ so that we can chose $Q(y,\cdot) \in \prob(F)$ such that $P h_y = \sum_{z\in F} Q(y,z) h_z$. 
\end{proof} 

For non-discrete state spaces and continuous time, the situation is more complicated, as we need to be able to choose the dual transition kernel $Q_t(x,\dd y)$ in such a way that the Chapman-Kolmogorov equations hold and that $x\mapsto Q_t(x,B)$ is measurable for every measurable $B\subset F$. The next theorem addresses non-degenerate dualities; in this case non-degeneracy ensures that there is a unique possible choice $Q_t(x,B)$ and the Chapman-Kolmogorov equations automatically hold. When $F$ is not discrete, we impose an additional condition to ensure measurability. 

\begin{assumption} \label{ass:w-dense} 
	The duality function $H$ is such that the associated linear space $\mathcal{W}$ from Definition~\ref{def:vwspaces} satisfies the following: 
	Every measurable function $g\in L^\infty(F)$ is the pointwise limit of a uniformly bounded sequence of functions $(g_n)$ from $\R \mathbf{1} + \mathcal{W}$.
\end{assumption} 

Clearly, if $H$ satisfies Assumption \ref{ass:w-dense}, then $\mathcal{W}$ is separating. The converse requires additional conditions. We recall the definition of continuous functions vanishing at infinity, useful in locally compact spaces: $f\in C_0(F)$ if  and only if $f$ is continuous and 
 for every $\eps>0$ there is a compact set $K\subset F$ such that $|f(x)|\leq \eps$ for $x\in F\setminus K$.  

\begin{lemma} \label{lem:pwdense}
	Suppose that $\mathcal{W}$ is separating on $\meas(F)$ and in addition one of the following conditions holds:
	\begin{enumerate} \itemsep0pt
	\item  $F$ is Polish and locally compact, and $\mathcal{W}\subset C_0(F)$.  
	\item  \label{ass:mvs}
	$\mathcal{W}$ contains a subset $\mathcal{W}'$ that is separating, closed under multiplication ($f,g\in \mathcal{W}'\Rightarrow fg\in \mathcal{W}'$), and generates the full Borel $\sigma$-algebra on $F$ ($\sigma(\mathcal{W}') = \mathcal{B}(F)$). 
	\end{enumerate} 
	Then Assumption \ref{ass:w-dense} holds.  
\end{lemma} 
Remember that for a locally compact space to be Polish as well, it is necessary and sufficient that its topology has a countable base; moreover every such space is $\sigma$-compact. 

\begin{proof}[Proof of Lemma \ref{lem:pwdense}]
	We first prove 2. Suppose that $\mathcal{W}'$ is closed under multiplication.  The functional monotone class theorem tells us that the closure of $\R\mathbf{1}+ \mathcal{W}'$ under  pointwise monotone limits of sequences is dense in $L^\infty(F, \sigma(\mathcal{W}')) = L^\infty(F,\mathcal{B}_F) = L^\infty(F)$. 	

	Next we prove 1. First we note that $\mathcal{W}$ is dense in $C_0(F)$ with respect to uniform convergence. Indeed, if this was not the case, the Hahn-Banach theorem would allow us to find a continuous functional $\varphi: C_0(F) \to \R$ vanishing on the closure of $\mathcal{W}$. By the Riesz-Markov theorem, there is a bounded signed measure $\mu$ such that $\varphi(f) = \int_E f\dd \mu$ for all $f$, and in particular $\int_E f\dd \mu = 0$ for all $f\in \mathcal{W}$. But this contradicts that $\mathcal{W}$ separates $\meas(F)$. Thus $\mathcal{W}$ is dense in $C_0(F)$.  Since for Polish spaces $C_0(F)$ generates the full Borel $\sigma$-algebra and $C_0(F)$ is closed under multiplication, the proof is concluded with a monotone class theorem just as in the proof of item 1. 
\end{proof} 

\begin{remark}
	A naive reasoning suggests that if $\mathcal{W'}$ is separating, then $\sigma(\mathcal{W}) = \mathcal{B}(F)$. Indeed, if the inclusion $\sigma(\mathcal{W'})\subset \mathcal{B}_F$ is strict, it is tempting to think that there must be some probability measure $P$ on $(F,\sigma(\mathcal{W'}))$ that has two distinct extensions $Q_1$, $Q_2$ to $(F,\mathcal{B}(F))$. We would have $\int g \dd Q_1 = \int g \dd Q_2$ for every $g \in \mathcal{W'}$ but $Q_1 \neq Q_2$, contradicting that $\mathcal{W'}$ is separating. 
	
	The following example, which we owe to M. Scheutzow, shows however that strict inclusion of $\sigma$-algebras in general does \emph{not} imply the existence of a probability measure with non-unique extensions: Let $F= [0,1]$, $\mathcal{B}$ the Borel $\sigma$-algebra, and $\mathcal{P}$ the collection of subsets of $[0,1]$. A Borel measure $P$ is either purely discrete, in which case it has a unique extension to $\mathcal{P}$, or has a continuous component, in which case it has no extension to $\mathcal{P}$. Thus even though the inclusion $\mathcal{B} \subset \mathcal{P}$ is strict, there is no probability measure on $\mathcal{B}$ with more than one extension to $\mathcal{P}$; 
	this is why in item 2. of Lemma~\ref{lem:pwdense} we explicitly require that $\sigma(\mathcal{W}') = \mathcal{B}(F)$. 
\end{remark}

The standard duality functions examined in Prop.~\ref{lem:non-degenerate} satisfy Assumption \ref{ass:w-dense}. Note that for the moment and Laplace duals, the families $\mathcal{W}'= \{H(x,\cdot) \mid x \in E\}$ are closed under multiplication -- for example, $x_1^n x_2^n = (x_1 x_2)^n$, whence $H(x_1,\cdot)H(x_2,\cdot) =H(x_1x_2,\cdot)$.

\begin{theorem} \label{prop:non-deg} 
	  Let $E,F$ be Polish spaces, $H:E\times F\to \R$ bounded and measurable, 
	and $(X_t)$ a Markov process with state space $E$ and semi-group $(P_t)$. 
	Suppose that either $\mathcal{W}$ is separating and $F$ discrete, or Assumption \ref{ass:w-dense} holds. Then $(P_t)$ has a dual Markov semi-group with respect to $H$  if and only if $\mathcal{V}_{1,+}$ is invariant, and the dual with respect to $H$ is unique. 
\end{theorem} 

See also \cite[Prop. 2.5]{moehle-cones}.
Before we come to the proof of Theorem~\ref{prop:non-deg} and Lemma~\ref{lem:pwdense}, we formulate a result on Feller semi-groups. 
Recall that 
$(P_t)$ is a \emph{Feller semi-group} on $C_0(E)$ if it maps $C_0(E)$ to itself and it is strongly continuous in $C_0(E)$, i.e.,  $||P_t f - f||_\infty \to 0$ as $t\to 0$, for every $f\in C_0(E)$. 

\begin{theorem} \label{thm:feller}
	Assume that $E$ and $F$ are Polish and locally compact. 
	Suppose that $H\in C_0(E\times F)$ and $B_H$ is non-degenerate. Then $(P_t)$ has a  Markov semi-group with respect to $H$ if and only if $\mathcal{V}_{1,+}$ is invariant, and the dual  is unique. Furthermore, the dual is a Feller semi-group if and only if $(P_t)$ is.   
\end{theorem} 
\begin{proof}[Proof of Theorem \ref{prop:non-deg}] 
	The uniqueness follows from Prop. \ref{prop:uniqueness}. For the existence, let $y\in F$ and $t \geq 0$; note $f_y(\cdot):=H(\cdot,y) = \int_F H(\cdot,z) \delta_y(\dd z) \in \mathcal{V}_{1,+}$. 
	Because of the invariance of $\mathcal{V}_{1,+}$ under $(P_t)$, there is a probability measure $\nu=\nu_{t,y}$ on $F$ such that $P_t f_ y (x) = \int_F H(x,y') \nu(\dd y')$.  The measure $\nu$ is unique because $\mathcal{W}$ is separating. We set $Q_t(y,B):= \nu_{t,y}(B)$. By construction, for every $t$ and $y$,  $Q_t(y,\cdot)$ is a probability measure, and we have, for all $t,x,y$,
	\begin{equation*}
		\int_E P_t(x,\dd x') H(x',y) = \int_F H(x,y') Q_t(y',\dd y).  
	\end{equation*}
	In order to show that $(Q_t)$ is a Markov semi-group dual to $(P_t)$, it remains to check that $y\mapsto Q_t(y,B)$ is measurable and that the Chapman-Kolmogorov equations hold. 
	We start with the measurability. 
	Let $g(\cdot) = \int_E \mu(\dd x) H(x,\cdot)  \in \mathcal{W}$, $\mu$ a signed measure on $E$. Then 
	\begin{equation*}
		(Q_t g)(y) = \int_E (\mu P_t)(\dd x) H(x,y). 
	\end{equation*}
	Since, by Fubini, the right-hand side is a measurable function of $y$, we find that if $g\in \mathcal{W}$, then $Q_t g$ is Borel-measurable. Let $B\subset F$ be a  Borel-measurable set. By Assumption~\ref{ass:w-dense}, there is a sequence of functions $(g_n)$ in $\R \mathbf{1} + \mathcal{W}$ such that $\sup_{n\in \N} ||g_n||_\infty < \infty$ and $g_n \to \mathbf{1}_B$ pointwise. By dominated convergence, we find that $Q_t(\cdot,B)$ is the pointwise limit of the measurable functions $Q_t g_n$; in particular, $Q_t(\cdot,B)$ is measurable. 
	
	Fix $s,t\geq 0$. For every $\mu\in \meas(E)$ and $\nu\in \meas(F)$, we have 
	\begin{equation*}
		B_H(\mu, Q_{t+s}^* \nu) = B_H(P_{t+s}^* \mu, \nu) = B_H(P_t^* P_s^* \mu, \nu) = B_H(\mu, Q_s^* Q_t^* \nu),  
	\end{equation*}
	thus $Q_s^* Q_t^* \nu$ and $Q_{t+s}^* \nu$ have the same integrals against functions from $\mathcal{W}$. Since $\mathcal{W}$ is separating, it follows that $Q_s^*Q_t^* = Q_{t+s}^*$.  
\end{proof} 

For the proof of Theorem~\ref{thm:feller} we need two lemmas. 
	
\begin{lemma} \label{lem:felleraux1} 
	Let $E$ and $F$ be Polish and locally compact, and $H\in C_0(E\times F)$. Then 
	 $\mathcal{V}\subset C_0(E)$. Moreover, for every $\eps>0$, there is a compact set $K\subset E$ such that for all $f\in \mathcal{V}$ and $x\in E\setminus K$, $|f(x)|\leq \eps$. 
\end{lemma} 

\begin{proof}
	If $f(\cdot)= \int_F H(\cdot,y) \nu( \dd y)\in \mathcal{V}$, then 
	\begin{equation*}
		\lim_{x\to x_0} f(x) = \lim_{x\to x_0} \int_F H(x,y) \nu(\dd y) = \int_F H(x_0,y) \nu(\dd y) = f(x_0).
	\end{equation*}
	The inversion of limits and integrals is justified by the dominated convergence theorem (recall that $H$ is bounded and $\nu$ is a finite signed measure). Thus  every function in $\mathcal{V}$ is continuous. Furthermore, given $\eps>0$, let $K\subset E\times F$ compact such that $|H(x,y)|\leq\eps$ for $(x,y)\notin K$. Let $\pi_E(x,y):= x$ the projection from $E\times F$ to $E$. For $x\notin \pi_E(K)$, 
	$|f(x)| \leq \eps ||\nu||$. Moreover $\pi_E(K)$, as the image of a compact set under a continuous map, is itself compact. Thus $f \in C_0(E)$ and we have checked that $\mathcal{V}\subset C_0(E)$. Since the compact set $\pi_E(K)$ depends on $H$ alone and not on $f$, the uniformity statement of the lemma follows. 
\end{proof} 

Our last lemma is a variant  of the well-known fact that integral operators are typically compact operators~\cite{werner}.

\begin{lemma} \label{lem:felleraux2}
	Let $E$ and $F$ be Polish and locally compact, 
	 and $H\in C_0(E\times F)$. The following holds:
	\begin{enumerate}
		\item Let $(\nu_n)$ be a sequence in $\prob(F)$ converging weakly to $\nu \in \prob(F)$. 
			Then \\
			$\int_F H(\cdot,y) \nu_n(\dd y) \to \int_F H(\cdot,y) \nu(\dd y)$  \emph{uniformly} in $E$. 
		\item $\mathcal{V}_{1,+} \subset C_0(E)$  is relatively compact.
		\item If $F$ is compact, then $\mathcal{V}_{1,+}$ is compact. 
	\end{enumerate}  
\end{lemma} 

Here weak convergence of measures means, as usual, that for every bounded continuous function $g$, 
$\int_F g \dd  \nu_n \to \int_F g \dd \nu$. 

\begin{remark} The following example shows that $\mathcal{V}_{1,+}$ is in general not closed. Let $E=F=\N$ and $H(m,n) =\delta_{m,n} n^{-1}$. Then 
	\begin{equation*}
		\mathcal{V}_{1,+} = \{ (a_n)_{n\in \N} \mid a_n \geq 0,\ \sum_{n=1}^\infty n a_n =1\}. 
	\end{equation*}
 The closure of $\mathcal{V}_{1,+}$ is the set of non-negative sequences $(a_n)$ with $\sum_{n=1}^\infty n a_n \leq 1$.
\end{remark}

\begin{proof}
	 Suppose that $\nu_n\to  \nu$ weakly. By the continuity of $H$ and the definition of weak convergence, we see that $\int_F H(\cdot,y)\nu_n(\dd y) \to \int_F H(\cdot,y)\nu(\dd y)$ pointwise, as $n\to \infty$. In order to show that the convergence is uniform, we are going to mimick the proof of the compactness of integral operators with continuous integral kernel, using the Arzel{\`a}-Ascoli theorem as in~\cite{werner}.  

First we show that  $\mathcal{V}_{1,+}$ is relatively compact in $C_0(E)$. 
Let $d$ be a metric on $E$ generating the topology. $H$ 
is uniformly continuous on every compact set $K\subset E\times F$; the uniform continuity extends to all of $E\times F$ because $H$ vanishes at infinity. 
Fix $\eps>0$ and let $\delta>0$ such that for all $y\in F$ and $x,x'\in E$ with $d(x,x')\leq \delta$ we have $|H(x,y)- H(x',y)| \leq \eps$. Let $\nu\in\prob(F)$. Then, for $x,x'\in  E$ with $d(x,x')\leq \delta$, 
\begin{equation*}
	\Bigl|\int_F H(x,y) \nu(\dd y) - \int_F H(x',y)\nu(\dd y)\Bigr| \leq 	 \sup_{y\in F} \bigl| H(x,y) - H(x',y)\bigr| 
				\leq  \eps. 
\end{equation*}
Thus $\mathcal{V}_{1,+}$ is an equicontinuous family of functions. It is also uniformly bounded, by $\sup_{E\times F} |H|$. If $E$ and $F$ are compact, we can conclude right away that $\mathcal{V}_{1,+}$ is relatively compact. 

For the general case, by Lemma~\ref{lem:felleraux1}, for every $j\in \N$ there is a compact set $K_j\subset E$ such that for all $f\in \mathcal{V}_{1,+}$, $\sup_{E\backslash K_j} |f| \leq 1/j$. In particular, every function in $\mathcal{V}_{1,+}$ vanishes outside $\cup_{j\in \N} K_j$. We may assume without loss of generality that $K_j \subset K_{j+1}$ for all $j$. 
Let  $g_n(\cdot) = \int_F H(\cdot,y) \nu_n(\dd y)$ be a sequence in $\mathcal{V}_{1,+}$. 
There is a function $g\in C(K_1)$ and a subsequence $g_{\phi_1(n)}$, $\phi_1:\N\to\N$ strictly increasing, converging to $g$ uniformly on $K_1$. Iterating the procedure, we find that there are strictly increasing maps $(\phi_j)$ and  a continuous function $g\in C(\cup_{j\in \N} K_j)$ such that for every $j\in \N$, 
\begin{equation*}
	\lim_{n\to \infty} \sup_{x\in K_j} \Bigl| g_{\phi_j\circ\cdots \circ \phi_1(n)} (x) - g (x)\Bigr| =0. 
\end{equation*}
Let $n_k:= \phi_k \circ\cdots \circ \phi_1(k)$ be the diagonal subsequence. Then, for every $j\in \N$, 
\begin{equation*}
	\lim_{k\to \infty} \sup_{x\in K_j} \Bigl| g_{n_k}(x) - g(x) \Bigr| = 0.
\end{equation*}
It follows that for $x \in \cup_{r\in \N} K_r \setminus K_j$, $|g(x)| \leq 1/j$. Extend $g$ by setting $g(x):= 0$ for $x\in E\setminus \cup_j K_j$. An $\eps/3$-argument shows that $g_{n_k}$ converges to $g$, uniformly in all of $E$, which also implies that $g\in C_0(E)$ (we had not yet checked that $g$ is continuous in all of $E$). This proves that $\mathcal{V}_{1,+}$ is relatively compact. 

Next, suppose that $(\nu_n)$ converges weakly to some probability measure $\nu$ on $F$. Let $g(\cdot):=\int_F H(\cdot,y) \nu(\dd y)$; we know that $g_n\to g$ pointwise. If the convergence is not uniform, because of the relative compactness of $\mathcal{V}_{1,+}$, we can find a subsequence $(g_{n_k})$ and a function $h \in \mathcal{V}_{1,+}$ with $h\neq g$ such that $g_n\to h$ uniformly. Since uniform convergence implies pointwise convergence, it follows that $g=h$, contradiction. Thus $g_n\to g$ uniformly. 

It remains to check that $\mathcal{V}_{1,+}$ is closed when $F$ is compact. Let $(g_n) = (\int_F H(\cdot,y) \nu_n(\dd y)) $ be a sequence in $\mathcal{V}_{1,+}$ converging uniformly to some function $g\in C_0(E)$. If $F$ is compact, we conclude from Prohorov's theorem that there is a subsequence $(\nu_{n_j})$ converging weakly to some probability measure $\nu$ on $F$, and $g = \int_F H(\cdot,y) \nu(\dd y) \in\mathcal{V}_{1,+}$. 
\end{proof} 

\begin{proof}[Proof of Theorem~\ref{thm:feller}] 
	The existence of a unique dual $(Q_t)$ follows from Theorem~\ref{prop:gen-form}.  
	 Suppose that $(P_t)$ has a dual Markov semi-group $(Q_t)$ that is strongly continuous in $C_0(E)$. 
	 By Prop.~\ref{prop:gen-form}, the set $\mathcal{V}$ is invariant under $(P_t)$. Since $P_t \mathbf{1} = \mathbf{1}$ and $||P_t f||_\infty \leq ||f||_\infty$, the closure of $\R\mathbf{1} + \mathcal{V}$ is invariant as well. By Lemma~\ref{lem:felleraux1}, this closure is all of $C_0(E)$, thus $C_0(E)$ is invariant under $(P_t)$. 

Next, let $\nu\in \prob(F)$. For every $g\in C_0(E)$, we have 
\begin{equation*}
	\bigl|\int_F g(y) (Q_t^*\nu)(\dd y) - \int_F g(y) \nu(\dd y) \bigr|
		= \bigl|\int_F (Q_t g)(y) \nu(\dd y) - \int_F g(y) \nu(\dd y) \bigr| \leq ||Q_t g - g||_\infty \to 0
\end{equation*}
 as $t\to 0$. Therefore $Q_t^* \nu \to \nu$ vaguely and, since $\nu$ is a probability measure, $Q_t^* \nu \to \nu$ weakly. 
 Let $f(\cdot):=\int_F H(\cdot,y) \nu(\dd y)$. We deduce from Lemma~\ref{lem:felleraux2} that as $t\to 0$, 
 	\begin{equation*}
		P_t f (x)= \int P_t(x,\dd x')H(x',y) \nu(\dd y) 
			= \int H(x,y)(Q_t^*\nu)(\dd y) \to \int H(x,y) \nu(\dd y) = f(x)
	\end{equation*}
	uniformly in $x\in E$. It follows that for every $f\in \R \mathbf{1} + \mathcal{V}$, $||P_t f - f|| \to 0$ as $t \to 0$. 
	The density of $\R \mathbf{1} + \mathcal{V}$ in $C_0(E)$ and the uniform bound $||P_t f||_\infty \leq ||f||_\infty$, allow us to conclude that $\lim_{t\to 0} ||P_t f- f|| =0$ for all $f\in C_0(E)$. 
	Thus if $(Q_t)$ is a Feller semi-group, so is $(P_t)$. Inverting the roles of $E$ and $F$, $(P_t)$ and $(Q_t)$, we see that the converse follows from the same arguments. 
\end{proof}

\subsection{Cone duality and lifts of non-degenerate dualities} \label{sec:cones}

Here we discuss the notion of cone duality. It was introduced in~\cite{krs07} and further discussed in~\cite{moehle-cones}. Both references were primarily interested in non-degenerate dualities. This section's main result, Theorem~\ref{prop:lift} shows that the notion develops  particular clout for the understanding of a certain type of degenerate dualities. 

The natural framework for cone duality is convex geometry and Choquet theory~\cite{phelps}. Recall that if $C$ is a convex set in some vector space, a point $x\in C$ is \emph{extremal} if it cannot be written as a convex combination of any two distinct points from $C$. The set of extremal points of $C$ is denoted $\ex C$. A \emph{simplex} in $\R^n$ is a non-empty convex, compact set such that every point in $C$ can be written in a unique way as a convex combination of the extremal points of $C$ -- a filled closed triangle is a simplex, but a square is not. We are interested in the following generalization of a simplex:

\begin{definition}[Unique integral representation] \label{def:uir}
Let $C\subset L^\infty(E)$ be a non-empty convex subset of $C$. We say $C$ admits 
 a \emph{unique integral representation} if there is a $\sigma$-algebra $\mathcal{F}$ on  $\ex C$ such 
\begin{itemize} \itemsep0pt
	\item For every $x\in E$, the evaluation map $\ex C\to \R$, $e\mapsto e(x)$ is measurable.
	\item Every function $f\in C$ can be represented as 
	\begin{equation} \label{eq:integral-rep}
		f(x) = \int_{\ex C} e(x) \mu(\dd e) 
	\end{equation} 
	for a unique probability measure $\mu$ on $(\ex C,\mathcal{F})$.
	\item $C$ contains all functions of the form~\eqref{eq:integral-rep}.
\end{itemize} 
\end{definition}

\begin{definition}[Cone duality] 	
	Let $E$ be a Polish space and $C\subset L^\infty(E)$ a convex set with a unique integral representation on $(\ex C,\mathcal{F})$.
	Let $(P_t)$ be a Markov semi-group on $E$ and $(Q_t)$ a semi-group on $(\ex C,\mathcal{F})$. Then 
	$(Q_t)$ is the \emph{cone dual} of $(P_t)$ if it is dual with respect to the duality function $H_C(x,e):= e(x)$. 
\end{definition} 
Thus $Q_t(e,\cdot)$ is the measure on $\ex C$ such that $(P_t e)(x) = \int  Q_t(e,\dd e')e'(x)$. 
The cone dual, if it exists, is unique. A necessary condition for the existence of the cone dual is the invariance of $C$ under $(P_t)$. If this condition is satisfied, the only potential source of problems is the measurability of $e\mapsto Q_t(e,B)$. 

\begin{remark}
	The name ``cone duality'' can be motivated in different ways. Recall that a cone, in the usual sense of convex geometry, is a set $\mathcal{C}$ such that if $x\in \mathcal{C}$, then the whole ray of points $t x$, $t \geq 0$ is in $\mathcal{C}$. Formally, this resembles the definition of invariance of a subset, $f\in C\Rightarrow \forall t \geq 0:\, P_t f\in C$;  M{\"o}hle~\cite{moehle-cones} therefore calls an invariant set a cone for the semi-group. Another motivation \cite[Section 1.5]{krs07} is to look at  the set 
	\begin{align*}
		\mathcal{C}&= \bigl \lbrace t f\mid f\in C, t\geq 0 \bigr \rbrace \\
			&= \Bigl \lbrace x\mapsto \int_{\ex C} e(x)\mu(\dd e)\mid \mu\ \text{finite  measure on }\ex C\Bigr \rbrace,
	\end{align*} 
which is a cone in the usual sense and plays an important role in Choquet theory. Klebaner, R{\"o}sler and Sagitov seemed further motivated by the role played by the cone of non-negative harmonic functions in the construction of the Martin boundary and in the interpretation of Doob's $h$-transform \cite[Section 1.1]{krs07}.  
\end{remark} 

When $H: E\times F\to \R$ is a duality function such that $\mathcal{W}$ is separating,  the cone dual and the usual dual can be identified via the bijection $F\to \ex \mathcal{V}_{1,+}$, $y\mapsto H(\cdot,y)$, as done in~\cite{krs07} for the Siegmund dual and the associated cone dual.  
%
The situation becomes more interesting when $\mathcal{V}_{1,+}$  has a unique integral representation even though $\mathcal{W}$ is not separating.  In this case the notions differ: the cone dual and the usual dual have different state spaces, and  the cone dual is unique, while the usual dual in general is not. 

We will restrict to $H\in C_0(E\times F)$ and compact $F$, so that by Lemma~\ref{lem:felleraux2}, $\mathcal{V}_{1,+}$ is a compact convex subset of $C_0(E)$ (in the uniform topology). Let $\mathcal{F}$ be the Borel $\sigma$-algebra on $\ex \mathcal{V}_{1,+}$ corresponding to the uniform topology. $\mathcal{V}_{1,+}$ has a unique integral representation with respect to $\mathcal{F}$ if and only if it is a \emph{Choquet simplex} \cite{phelps}. 

\begin{lemma}\label{lem:pimatrix} 
	Let $E$ be Polish and locally compact,  $F$ a compact separable Hausdorff space, and   $H\in C_0(E\times F)$. Suppose in addition that $\mathcal{V}_{1,+}$ a Choquet simplex. 
For $y \in F$, let 
 $\Pi(y,\cdot)$ be the unique probability measure on $\ex \mathcal{V}_{1,+}$ such that for all $x\in E$,
	\begin{equation} \label{eq:pidef}
		H(x,y) = \int_{\ex \mathcal{V}_{1,+}} e(x) \Pi(y,\dd e).  
	\end{equation}
	Then $F \ni y\mapsto \Pi(y,B)$ is measurable, for all measurable $B \subset \ex\mathcal{V}_{1,+}$.  
\end{lemma} 
Thus $\Pi$ is a transition kernel from $F$ to $\ex \mathcal{V}_{1,+}$

 \begin{example} 
Let $E= \{1,2\}$, $F=\{1,2,3,4\}$, and $H$  given by the matrix 
\begin{equation} \label{eq:exasim}
	H= \begin{pmatrix} 
			2 & 0 & 1 & 0 \\
			0 & 2 & 1  & 2
	\end{pmatrix}
\end{equation}
Then $\mathcal{V}_{1,+}$ is a simplex with two extremal points, the column vectors $(2,0)^\trans$ and $(0,2)^\trans$. The transition kernel $\Pi$ can be identified with the matrix 
\begin{equation*}
		\Pi =\begin{pmatrix} 
				1 & 0 \\
				0 & 1 \\
				1/2 & 1/2 \\
				0 & 1 
			\end{pmatrix}. 
\end{equation*}
\end{example}

\begin{theorem} \label{prop:lift} 
	Let $E$ be a Polish and locally compact space, $F$ a compact separable Hausdorff space, $H\in C_0(E\times F)$, 
	and $(P_t)$ a Markov semi-group in $E$. Suppose that $\mathcal{V}_{1,+}$ is a Choquet simplex invariant under $(P_t)$. Then  
	\begin{enumerate} \itemsep0pt
		\item The cone dual exists and is unique. It is defined on $\ex \mathcal{V}_{1,+}$ with the uniform topology and the Borel $\sigma$-algebra; $\ex \mathcal{V}_{1,+}$ is Polish. 
		\item $(P_t)$ has at least one dual Markov semi-group $(Q_t)$. 
		\item Let $R_t(x, \dd e)$ be the semi-group of the cone dual and $(Q_t)$ a Markov semi-group in $F$. Then $(Q_t)$ is dual to $(P_t)$ if and only if
		\begin{equation} \label{eq:pilift} 
			\int_F Q_t(y, \dd y') \Pi(y', B) = \int_{\ex \mathcal{V}_{1,+}} \Pi(y,\dd e) R_t(e, B)
		\end{equation}
		for all $y\in F$ and all measurable $B \subset \ex \mathcal{V}_{1,+}$. 
	\end{enumerate} 
\end{theorem}

Eq.~\eqref{eq:pilift} essentially says that any dual Markov process $(Y_t)$ and the cone dual $(Z_t)$ can be defined on a common probability space $(\Omega,\P)$ in such a way that for all $t>0$, 
\begin{equation} \label{eq:lift-prob}
		\P(Z_t \in B\mid Y_t = y) = \Pi(y,B) 
\end{equation}
 provided the relation holds at $t=0$. 
An interesting special case is when every $H(\cdot,y)$ is extremal, so that the map $\pi(y):= H(\cdot,y)$ is a surjection from $F$ onto $\ex \mathcal{V}_{1,+}$. In this case 
 $\Pi(y,B) = \delta_{\pi(y)}(B)$ and Eq.~\eqref{eq:lift-prob} becomes 
$$\P(Z_t = \pi(Y_t)) =1.$$
As a consequence,  we can think of  $(Y_t)$ as a lift of the cone dual from $\ex \mathcal{V}_{1,+}$ to the ``bigger'' space $F$. This is a kind of converse to Prop.~\ref{lem:lift-deg}. 

\begin{remark}
	Relations of the type~\eqref{eq:pilift} are often called \emph{intertwining} relations. For references and applications of intertwining in the context of duality, see \cite{diaconis-miclo09,swart11,huillet-martinez11} and Section~\ref{sec:intertwining}. An example of an intertwining relation with invertible kernel (note that our $\Pi$ is \emph{not} invertible) is the concept of \emph{thinning} in interacting particle systems, which has been related to duality theory as well. Here the kernel represents transition probabilities for a particle to be thrown away. 
\end{remark}

In general, the duals constructed for the proof of Theorem~\ref{prop:lift} have rather bad continuity properties. There will be a subset $F_0 \subset F$ such that $Q_t(y,F_0) =1$ for all $y\in F$ and all $t>0$; in particular, if $F_0$ is not dense in $F$, the dual process $(Y_t)$ constructed in 2. will leave any neighborhood of $y\in F\setminus \overline{F_0}$ immediately: the constructed process has branch points. 

If we have more information than in Theorem~\ref{prop:lift}, we can slow down the jumps from $F\setminus F_0$ to $F_0$ and obtain a dual with nice continuity properties. We prove a statement for finite state spaces only. 

\begin{theorem} \label{prop:continuous-dual}
	Let $E$ and $F$ be finite state spaces and $H\in \R^{E \times F}$. Let $(P_t)$ be a strongly continuous Markov semi-group in $E$. Suppose that $\mathcal{V}_{1,+}$ is a simplex and invariant under $(P_t)$. Then 
	$(P_t)$ admits a strongly continuous dual $(Q_t)$. 
\end{theorem} 

Now we come to proofs. The proof of Theorem~\ref{prop:lift} requires two lemmas. 

\begin{lemma}  \label{lem:convex-geometry} 
	Let $X$, $Y$ be locally convex vector spaces, $C\subset X$ convex and compact, and $T:X\to Y$ a continuous linear operator. Then $\ex T(C) \subset T(\ex C)$. 
\end{lemma} 

Later we are going to apply this lemma to $X=\meas(F)$ with the topology of weak convergence of measures, $Y= C_0(E)$ with the topology from the supremum norm, and $T: \nu \mapsto \int_F H(\cdot,y) \nu(\dd y)$.

\begin{proof}
	Let $y \in \ex T(C)$. Let $F_y:= \{ x\in X\mid Tx = y\}$. Then $F_y$ is a \emph{face} in $C$, i.e., if $x\in F_y$ is a convex combination of two points in $C$, then these two points must be in $F_y$. Indeed, if $x\in M_f$ is of the form $x=(1-t) x_1 + tx_2$ with $t \in (0,1)$ and $x_1,x_2\in C$, then $y = (1-t) Tx_1 + t T x_2$. Since $y$ is extremal in $C$, it follows that $y = Tx_1 = Tx_2$, thus $x_1$ and $x_2$ are in $F_y$.
	
	$F_y$ is closed (because $T$ is continuous) and, as a closed subset of the compact set $C$, itself compact. Therefore $\ex F_y \neq \emptyset$ (this is a part of the Krein-Milman theorem~\cite[Theorem VIII.4.4]{werner}). Moreover, because $F_y$ is a face, $\ex F_y = (\ex K) \cap F_y$  \cite[Lemma VIII.4.2]{werner}.  Thus we can pick $x\in \ex F_y$ and find that $y = Tx$ and $x\in \ex C$. This proves $\ex T(C)\subset T(\ex C)$. 
\end{proof}

\begin{lemma} \label{lem:lift} 
	Under the assumptions of Theorem~\ref{prop:lift}, let
	\begin{equation*}
		 F_0:= \{ y \in F\mid H(\cdot,y) \in \ex \mathcal{V}_{1,+} \}
	\end{equation*} 
	and $\pi: F_0 \to \ex \mathcal{V}_{1,+}$ be given by $\pi(y):= H(\cdot,y)$. Then 
	\begin{itemize} \itemsep0pt
		\item $F_0$ is a measurable subset of $F$.
		\item $\pi$ is surjective and continuous. 
		\item $\pi$ has a Borel-measurable right inverse: there is an injective, measurable map $\iota: \ex \mathcal{V}_{1,+}\to  F_0$ such that $ \pi \circ \iota =\id_{\ex \mathcal{V}_{1,+}}$.  
	\end{itemize} 
\end{lemma}
The right inverse $\iota$ will be used to prove the existence of a $\Pi$-lift as in Eq.~\eqref{eq:pilift}, see Eq.~\eqref{eq:i-lift} below.  

\begin{proof} 
	 The surjectivity is a consequence of Lemma~\ref{lem:convex-geometry}, 
	applied to $X = \meas(F)$, $Y = C_0(E)$ and $T: \nu \mapsto \int_F H(\cdot,y) \nu(\dd y)$, where $\meas(F)$ is endowed with the topology of weak convergence of measures and $C_0(E)$ the (supremum) norm topology. By Lemma \ref{lem:felleraux1}, $\mathcal{V}_{1,+}$ is a compact subset of $C_0(E)$. By Lemma~\ref{lem:felleraux2}, $T$ maps weakly convergent sequences to convergent sequences; since the topology of weak convergence on probability measures in Polish spaces is metrizable, it follows that $T$ is continuous. Thus Lemma~\ref{lem:convex-geometry} can be applied to the convex, compact set $\mathcal{V}_{1,+}$, and we find that every extremal $f \in \mathcal{V}_{1,+}$ is of the form $f =T\nu$ for some $\nu\in\ex \prob(F)$. Now, $\nu$ is extremal in $\prob(F)$ if and only if it is a Dirac measure $\nu = \delta_y$. Thus $f(\cdot)= H(\cdot,y)$ for some $y\in \bar F$, and $\pi$ is surjective. The continuity of $\pi$ follows from Lemma~\ref{lem:felleraux2}. 
	
The set	$\ex \mathcal{V}_{1,+}$ is a $G_\delta$ subset (countable intersection of open sets) of the compact metric space $\mathcal{V}_{1,+}$ \cite[Theorem 4.1.11]{bratteli-robinson-vol1} hence in particular, measurable. Thus $F_0 = T^{-1}(F)$, as the preimage of the measurable set under a continuous map, is measurable. 

 	Since $\pi$ is surjective, for every $z \in \ex \mathcal{V}_{1,+}$, there is a $y(z)\in F_0$ such that $\pi( y(z)) = z$. If $F_0$ is discrete, every choice $z \mapsto y(z)$ corresponds to a measurable map. If $F$ is compact, the existence of a measurable right inverse is guaranteed by a measurable choice theorem~\cite[Theorem 6.9.7]{bogachev2}. The injectivity of $\iota$ is for free: if $\iota(z) = \iota (z')$, then $z =\pi(\iota(z)) = \pi(\iota(z')) = z'$. 
\end{proof}

\begin{proof}[Proof of Lemma~\ref{lem:pimatrix}] 
	We only need to prove the measurability of $y\mapsto \Pi(y,B)$. For $f\in \mathcal{V}_{1,+}$, let $\hat \Pi(f,\cdot)$ be the unique probability measure on $\ex\mathcal{V}_{1,+}$ such that $f(x) = \int \hat \Pi(f,\dd e) e(x)$, for every $x\in E$. Let $S(\mathcal{V}_{1,+})$ be the set of convex, continuous functions on $\mathcal{V}_{1,+}$ and for $\phi\in C(\mathcal{V}_{1,+})$, define the \emph{upper envelope} by 
\begin{equation*}
	\tilde \phi(f):=\inf\{ \psi(f) \mid -\psi \in S(\mathcal{V}_{1,+}) \text{ and } \psi \geq \phi \text{ on } \mathcal{V}_{1,+} \}.
\end{equation*}
$\tilde \phi$ is concave and upper semi-continuous~\cite[Prop. 4.1.6]{bratteli-robinson-vol1}. 
Fix $f\in \mathcal{V}_{1,+}$. The measure $\hat \Pi(f,\cdot)$ can be considered as a measure on $\mathcal{V}_{1,+}$ which is supported on the measurable subset $\ex	\mathcal{V}_{1,+}$; it is maximal in the sense of~\cite[Def. 4.1.2]{bratteli-robinson-vol1}, see~\cite[Theorem 4.1.11]{bratteli-robinson-vol1}. It follows that 
\begin{equation*}
	\int _{\ex \mathcal{V}_{1,+}} \hat \Pi(f, \dd e) \phi(e) =  \int _{\ex \mathcal{V}_{1,+}} \hat \Pi(f, \dd e) \tilde \phi(e) 
\end{equation*}
\cite[Theorem 4.1.7]{bratteli-robinson-vol1}. If $\phi \in S(\mathcal{V}_{1,+})$, then $\tilde \phi$ is affine \cite[Theorem 4.1.15]{bratteli-robinson-vol1}, and therefore 
\begin{equation*}
	\tilde \phi(f) = \int _{\ex \mathcal{V}_{1,+}} \hat \Pi(f, \dd e) \tilde \phi(e) 
\end{equation*}
\cite[Corollary 4.1.18]{bratteli-robinson-vol1}. We deduce $(\hat \Pi \phi)(f) =\tilde \phi(f)$. In particular, $\hat \Pi \phi$ is upper semi-continuous and therefore measurable. Since differences of continuous convex functions are dense in $C(\mathcal{V}_{1,+})$ \cite[Lemma 4.1.14]{bratteli-robinson-vol1}, and every measurable function is a pointwise limit of continuous functions, it follows that $\hat \Pi$ maps Borel measurable functions to Borel measurable functions. In particular, for every measurable $B\subset \mathcal{V}_{1,+}$, $f\mapsto \hat \Pi(f,B)$ is measurable.

Since $\Pi(y,B) = \hat \Pi( H(\cdot,y),B)$, we find that $y\mapsto \Pi(y,B)$ is the composition of a measurable map and a continuous map, and therefore measurable, for every fixed measurable $B$. 
\end{proof}

\begin{proof}[Proof of 1. in Theorem~\ref{prop:lift}]
Endow $\ex \mathcal{V}_{1,+}$ with the uniform topology and the Borel $\sigma$-algebra $\mathcal{F}$ as described above. By assumption, $\mathcal{V}_{1,+}$ is a Choquet simplex and therefore has a unique integral representation over $(\ex \mathcal{V}_{1,+},\mathcal{F})$. Note that 
if $e_n\to e$ uniformly and $x_n\to x$, then $e_n(x_n)\to e(x)$. It follows that the function $H_C(x,e) = e(x)$ and the evaluation maps $e\mapsto e(x)$ are continuous, hence measurable. Furthermore, $\ex \mathcal{V}_{1,+}$ is a $G_\delta$ subset of the compact separable Hausdorff space $\mathcal{V}_{1,+}$ \cite[Theorem 4.1.11]{bratteli-robinson-vol1}, hence Polish.

The uniqueness of the cone dual follows from the uniqueness in the integral representation, see also Prop.~\ref{prop:uniqueness}. For $z\in \ex \mathcal{V}_{1,+}$ and $t>0$, let $R_t(e,\cdot)$ be the unique probability measure on $\ex \mathcal{V}_{1,+}$ such that $(P_t e)(x) = \int \tilde e(x) R_t(e,\dd \tilde e)$. 
 
In order to check the measurability, let $\hat \Pi(f,\dd e)$ be as in the proof of Lemma~\ref{lem:pimatrix}. Fix $B\subset \ex\mathcal{V}_{1,+}$ measurable and $t>0$.  Note $R_t(e,B) = \hat \Pi (P_t e, B)$. Thus $e\mapsto R_t(e,B)$ is the composition of the continuous map $e\mapsto P_t e$ with the measurable map $f\mapsto \hat \Pi(f,B)$. It follows that it is measurable. 
\end{proof}

\begin{proof}[Proof of 3. in Theorem~\ref{prop:lift}]
	Suppose that $(Q_t)$ is dual to $(P_t)$. Then, for every $x\in E$ and $y\in F$, we have 
	\begin{align*}
		\int_F Q_t(y,\dd y') H(x,y') & = \int_E P_t(x,\dd x') H(x',y) \\
				& = \int_E P_t(x,\dd x')\Bigl( \int_{\ex \mathcal{V}_{1,+}} e(x') \Pi(y, \dd e)\Bigr)\\
				& = \int_{\ex \mathcal{V}_{1,+}}\Pi(y, \dd e) \Bigl(\int_E P_t(x,\dd x') e(x') \Bigr) \\
				& = \int_{\ex \mathcal{V}_{1,+}}\Pi(y, \dd e) \Bigl(\int_{\ex \mathcal{V}_{1,+}} R_t(e,\dd \tilde e) \tilde e(x) \Bigr),
	\end{align*}
	but also 
	\begin{equation*}
		\int_F Q_t(y,\dd y') H(x,y') = \int_F Q_t(y,\dd y')\Bigl( \int_{\ex \mathcal{V}_{1,+}}  \Pi(y', \dd e) e(x') \Bigr).
	\end{equation*}
	Eq.~\eqref{eq:pilift} follows from the uniqueness of the integral representation. Conversely, if Eq.~\eqref{eq:pilift} holds, computations in the same spirit show that $(Q_t)$ is dual to $(P_t)$ with respect to $H$. 
\end{proof} 

\begin{proof}[Proof of 2. in Theorem~\ref{prop:lift}]
	Because of 3. we only need to show that the cone dual has a $\Pi$-lift. To this aim let $\iota:\ex \mathcal{V}_{1,+} \to F_0$ be a Borel measurable right inverse of $\pi$ as in Lemma~\ref{lem:lift}. 
For $y\in F$ and $A\subset F$ measurable, define 
\begin{equation} \label{eq:i-lift} 
	Q_t(y,A):= \int _{\ex \mathcal{V}_{1,+}} \Pi(y,\dd e) R_t\bigl(e, \iota^{-1}(A\cap F_0)\bigr). 
\end{equation}
For each $t>0$, $Q_t$ is a  Markov kernel. By definition, $Q_t(y,F_0) =1$. Let $f:F_0 \to \R$ bounded and measurable; then 
\begin{equation} \label{eq:lift-aux}
	\int_{F} Q_t(y,\dd y') f(y') = \int _{\ex \mathcal{V}_{1,+}} \Pi(y,\dd e) \Bigl( \int _{\ex \mathcal{V}_{1,+}}
		R_t\bigl(e, \dd \tilde e\bigr) f\bigl( \iota (\tilde e)\bigr) \Bigr).
\end{equation}
Let $g \in L^\infty(\ex \mathcal{V}_{1,+})$ and $f(y):=\int \Pi(y,\dd e) g(e)$. For $y = \iota(\tilde e) \in F_0$, the function $H(\cdot,y)$ is extremal and $\Pi(y,\cdot) = \delta_{\pi(y)} =\delta_{\tilde e}$. It follows that $f\circ \iota (\tilde e) = g(\tilde e)$ and  Eq.~\eqref{eq:lift-aux} yields
\begin{equation*}
	\int_{F} Q_t(y,\dd y') \int_{\ex \mathcal{V}_{1,+}} \Pi(y',\dd e) g(e) 
			= \int _{\ex \mathcal{V}_{1,+}} \Pi(y,\dd e) \Bigl( \int _{\ex \mathcal{V}_{1,+}}
		R_t\bigl(e, \dd \tilde e\bigr) g(\tilde e) \Bigr).
\end{equation*}
Therefore $Q_t$ satisfies Eq.~\eqref{eq:pilift}, which we write as $Q_t \Pi = \Pi R_t$. We have, for every $t,s>0$, 
\begin{equation} \label{eq:chap-ko-aux}
	Q_{t+s} \Pi =\Pi R_{t+s} = \Pi R_t R_s = Q_t \Pi R_s = Q_t Q_s \Pi. 
\end{equation}
Let $f\in L^\infty(F)$ and $g(e):= f(\iota (e))$. As mentioned above, for every $y\in F_0$, we have $(\Pi g)(y) = f(y)$. Together with $Q_t(y,F_0) =1$ and Eq.~\eqref{eq:chap-ko-aux}, this shows that $(Q_t)$ satisfies the Chapman-Kolmogorov equations. 
\end{proof}

Now we turn to the proof of Theorem~\ref{prop:continuous-dual}. Recall that a $Q$-matrix is a matrix with row sums $0$ and non-negative off-diagonal entries.

\begin{proof}[Proof of Theorem~\ref{prop:continuous-dual}]
	Assume $(P_t)$ is strongly continuous with generator $L$, i.e., $P_t = \exp(t L)$ for all $t>0$. Theorem~\ref{thm:feller} shows that the cone dual is strongly continuous. 
	Let $F_1\subset F$ be a set such that the columns $H(\cdot,y_1)$, $y_1\in F_1$, are extremal in $\mathcal{V}_{1,+}$, and every column $H(\cdot,y)$, $y\in F$, is a unique convex combination of columns indexed by $y_1\in F_1$. In the notation of Lemma~\ref{lem:lift}, $F_1:=\iota(\ex \mathcal{V}_{1,+}) \subset F_0$.

	 From here on we identify $F_1$ with $\ex \mathcal{V}_{1,+}$ and consider the cone dual $(R_t)$ as a process with state space $F_1$. The transition kernel $\Pi$ from Lemma~\ref{lem:pimatrix} becomes a $|F|\times|F_1|$ matrix. It has a block structure 
	\begin{equation*}
		\Pi = \begin{pmatrix} 
				\id \\
				\Pi_{21} 
			\end{pmatrix} 
	\end{equation*}
	with $\id$ a $|F_1|\times|F_1|$ identity matrix, and $\Pi_{21}$ a $|F_2|\times |F_1|$ matrix, $F_2:= F\setminus F_1$. The matrix $H$ is of the form 
	\begin{equation*}
		H= \begin{pmatrix} 
			H_1 & H_2 
			\end{pmatrix}, 
			\quad  H_2 = H_1 \Pi_{21}^\trans, 
	\end{equation*}
	where $H_1$ and $H_2$ are $|E|\times |F_1|$ resp. $|E|\times |F_2|$ matrices, and $\ker H_1\cap \{\mathbf{1}\}^\perp = \{0\}$. We can define a $|F|\times |F|$ stochastic matrix $\hat \Pi$ by the block structure 
	\begin{equation*} 
		\hat \Pi = \begin{pmatrix} 
						\id & 0 \\
						\Pi_{21} & 0 
					\end{pmatrix}. 
	\end{equation*}
	Note $H= H \hat \Pi$. 
 The dual $(Q_t)$ constructed for the proof of Theorem~\ref{prop:lift} is of the form 
	\begin{equation} \label{eq:jump-dual} 
		Q_t = \begin{pmatrix}
				R_t & 0 \\
				\Pi_{21} R_t  & 0 \\ 
			\end{pmatrix}. 
	\end{equation} 
	It describes a process that jumps from $F_2$ to $F_1$ right away; we want to slow down this jump in order to obtain a strongly continuous process. Thus we are looking for a $Q$-matrix $\hat L$ such that $\hat L H = H L^\trans$. Let $(Q_t)$ be a dual as in Eq. \eqref{eq:jump-dual}. Since $Q_t H^\trans = H^\trans P_t$, we find that $Q_t$, restricted to $\mathcal{W}:=\ran H^\trans\subset \R^F$, is strongly continuous. Moreover, by Prop.~\ref{prop:gen-form}, $(Q_t)$ leaves $\mathcal{W}$ invariant. As a consequence, there is a linear map $B$ from $\R \mathbf{1} +\mathcal{W}$ to itself such that 
$Q_t g = \exp(t B) g$ for all $t\geq 0$ and $g \in \R \mathbf{1} +\mathcal{W}$, and we have $B H^\trans = H^\trans L^\trans$. 
		
	Next, we note that $\hat \Pi$ is a projection ($\hat \Pi^2 = \hat \Pi$) with  
	$\ran \hat \Pi = \R \mathbf{1} + \mathcal{W}$. Indeed, the inclusion $ \R \{\mathbf{1}\} + \mathcal{W}\subset \ran \hat \Pi$ follows from $\hat \Pi \mathbf{1} = \mathbf{1}$ and $\hat  \Pi H^\trans = H^\trans$. The reverse inclusion is equivalent to 
$\{\mathbf{1}\}^\perp \cap \ker H \subset \ker {\hat \Pi}^\trans$ (remember that in finite dimensions, $(V+W)^\perp = V^\perp \cap W^\perp$ and $\ran A^\trans = (\ker A)^\perp)$). Indeed, if $v\in \{\mathbf{1}\}^\perp \cap \ker H$, there are two probability measures $\mu,\nu$ on $F$ (considered as column vectors) such that $v = \mu - \nu$ and $H \mu = H \nu$; using the definition of $\Pi$ and that every column of $H$ is a unique convex combination of columns $H(\cdot,y)$, $y \in F_1$, we deduce that ${\hat  \Pi}^\trans \mu = {\hat \Pi}^\trans \nu$. 
	
	As a consequence, $B\hat \Pi$ is a well-defined linear map in $\R^F$ and we can define, for $\lambda>0$ to be specified later, 
	\begin{equation*}
		\hat L := B \hat \Pi + \lambda (\hat \Pi - \id), \quad \hat Q_t:= \exp(t \hat L). 
	\end{equation*}
	We have $\hat L \mathbf{1} = 0$, and  
 $\hat L$ and $\hat Q_t$ have $F_1 \times F_2$ block structures of the form
	\begin{equation*}
		\hat L =\begin{pmatrix} 
				\hat L_{11} & 0 \\
				* & 0 
			\end{pmatrix} 
			+ \begin{pmatrix} 
				0 & 0 \\ 
				\lambda \Pi_{21} & - \lambda \id 
			\end{pmatrix}, \quad 
		\hat Q_t =\begin{pmatrix} 
				\exp(t \hat L_{11}) & 0 \\
				* & \exp(- t \lambda) \id
			\end{pmatrix}.
	\end{equation*}
	Here and below the star $*$ is a generic abbreviation for a matrix block not necessarily equal to zero, to which we do not assign names because the precise values do not matter; different star blocks need not be equal. 
	By construction,
	\begin{equation*}
		\hat L H^\trans = \hat L \hat \Pi H^\trans = B H^\trans = H^\trans L^\trans. 
	\end{equation*}
	It remains to check that $\hat L$ is a $Q$-matrix. We already know that $\hat L \mathbf{1} = 0$; choosing $\lambda$ large enough, we can ensure that the lower left block of $\hat L$ has only non-negative entries. Thus we are left with $\hat L_{11}$. 
	For $f\in \R^{F}$, 
	\begin{equation*}
		\hat Q_t \hat  \Pi f 
			= \hat Q_t \begin{pmatrix} 
					f_1 \\ \Pi_{21} f_2 \end{pmatrix} 
			= \begin{pmatrix} 
				\exp(t\hat L_{11}) f_1 \\ * \end{pmatrix}.
	\end{equation*}
	On the other hand, since $\ran \hat \Pi = \R \mathbf{1} + \mathcal{W}$, we have 
	$\hat Q_t \hat \Pi f = Q_t \hat \Pi f$. Therefore if $f_1 \geq 0$ componentwise, then also $\exp(t \hat L_{11}) f_1 \geq 0$ componentwise, for all $t\geq 0$. As a consequence, 
	$\hat L_{11}$ has non-negative off-diagonal entries. 	
\end{proof}

\subsection{Spectrum and unitary equivalence. Fourier transforms}

Let $(P_t)$ and $(Q_t)$ be Markov semi-groups in finite state spaces. Suppose that they are dual with respect to some non-degenerate duality function $H$. Then $H$ defines an invertible $E\times F$-matrix and we can write $P_t = H Q_t^\trans H^{-1}$, which implies in particular that $P_t$ and $Q_t$ have the same eigenvalues and eigenvalue multiplicities. As noted in the paragraphs preceding Theorem~3.2 in \cite{Moehle99}, this observation has been used to compute eigenvalues for some models of population genetics;\footnote{Moran \cite{moran58} exploited recurrence relations between moments corresponding to a moment type duality. The triangular matrix of recurrence coefficients is related to the structure of the block-counting process of the Kingman coalescent.} Besides this computational application, the observation is  interesting because of the relation between the spectral gap and mixing properties of the Markov chain.  Moreover, some chains allow for a stochastic interpretation of \emph{all} eigenvalues: for example,  
\cite[Remark 4.22]{DiaconisFill} relates the eigenvalues to parameters of geometric random variables characterizing how fast a chain becomes stationary.

Theorem~\ref{thm:unitary} below states that an analogous relation holds for infinite state spaces if both Markov processes are known to be reversible; this  applies to many models studied in the context of hydrodynamic limits \cite{{DeMasiPresutti91}}. Without the assumption of reversibility, non-degenerate duals can have drastically different spectrum: in the following example, taken from \cite[Sec. 3.5]{RedigEtAl09} one infinitesimal generator has discrete, real spectrum, while the other has the complex half-plane $\Re \lambda \leq 0$ as its spectrum. 

\begin{example} 
	Let $E=\{ (x_1,x_2) \in \R_+^2 \mid 0 \leq x_1+x_2 <1,\ x_1\leq x_2\}$, $F=\N_0\times \N_0$. 
	On $E$, let 
	\begin{equation*} 
		(P_t f)(x_1,x_2) := f\bigl( \frac{x_1+x_2}{2} + \frac{x_1-x_2}{2} e^{-t}, 
		 	\frac{x_1+x_2}{2} - \frac{x_1-x_2}{2} e^{-t} \bigr).
	\end{equation*} 	
	$(P_t)$ describes a (deterministic) process $X(t) = (X_1(t),X_2(t))$ where $X_1(t)+ X_2(t)$ stays constant and $X_1(t) - X_2(t) = (X_1(0)- X_2(0)) \exp( - t)$. 
	On $F$, we consider the semi-group $Q_t := \exp(t\hat L)$ with 
	\begin{align*}
		 (\hat L f)(n_1,n_2):= n_1 \bigl( f(n_1-1,n_2+1) - f(n_1,n_2)\bigr) + n_2 \bigl( f(n_1+1,n_2-1) - f(n_1,n_2) \bigr). 
	\end{align*}
	$(Q_t)$ describes a process of independent random walkers on two sites. 
	As shown in \cite[Sec. 3.5]{RedigEtAl09}, $(P_t)$ and $(Q_t)$ are dual with duality function $H(\vect{x},\vect{n}) = x_1^{n_1} x_2^{n_2}$. The associated bilinear form is non-degenerate. 
	
	In order to determine the spectrum of $\hat L$, we note that $\hat L$ leaves the finite state spaces $E_n:=\{(n_1,n_2)\in E \mid n_1+n_2 =n\}$, $n\in \N_0$ invariant; on each of these, $\hat L$ has the uniform distribution as a reversible measure. Thus $\hat L$ is block diagonal with finite-dimensional symmetric matrices as blocks; as a consequence, it has real, discrete spectrum -- in fact, using the duality between $\hat L$ and an Ornstein-Uhlenbeck process \cite[Remark 3.1]{RedigEtAl09}, one can show that
\begin{equation*}	
	 \sigma(\hat L) = \{ - n/2 \mid n \in \N_0\}
\end{equation*}	
	On the other hand, let  $\lambda \in \C$ with $\Re \lambda < 0$ and define the  
	 complex-valued function $f_\lambda$ by $f_\lambda(x_1,x_2):= (x_2-x_1)^{-\lambda} =  \exp(- \lambda \log (x_2-x_1) )$ for $x_2>x_1$ and $f_\lambda(x_1,x_2) = 0$ for $x_2=x_1$ is continuous on $E$. We have 
	\begin{equation*}
		(P_t f_\lambda)(x_1,x_2) =  \Bigl( (x_2-x_1) e^{-t} \Bigr)^\lambda = e^{-t \lambda} f_\lambda(x_1,x_2).
	\end{equation*} 
	Let $L$ be the infinitesimal generator of $(P_t)$ in $C(E;\C)$, the space of complex-valued continuous functions in $E$. We have just shown that $\lambda \in \sigma(L)$, for any $\lambda$ with $\Re \lambda <0$. The theory of contraction semi-groups  tells us that $\sigma(L)$ is a closed set contained in the complex half-plane $\Re z \leq 0$. Thus 
\begin{equation*}
	\sigma(L) = \{\lambda \in \C \mid \Re \lambda \leq 0\}.
\end{equation*}
\end{example} 

Now let us turn to the reversible case. Recall that if $\mu$ is a reversible measure for the semi-group $(P_t)$, then $P_t$ is self-adjoint as an operator in $L^2(E,\mu)$, the Hilbert space of complex-valued square-integrable functions with scalar product $\la f,g\ra= \int_E \overline{f} g \dd \mu$. A bounded operator $U:\Hi_1 \to \Hi_2$ between Hilbert spaces is  \emph{unitary} if $U^*U=\id_{\Hi_1}$ and $U U^*= \id_{\Hi_2}$; equivalently, if it is bijective and norm preserving, $||U x || = ||x||$. 

\begin{theorem} \label{thm:unitary}
	Let $(P_t)$ and $(Q_t)$ be dual with respect to $H$. Suppose that 
	that $(P_t)$ and $(Q_t)$ have reversible probability measures $\mu$ and $\nu$, and that the bilinear form associated with $H$ is non-degenerate. Then  there is a unitary operator $U:L^2(F,\nu) \to L^2(E,\mu)$ such that for all $t>0$, $Q_t = U^*P_t U$.
\end{theorem}

Thus $(P_t)$ and $(Q_t)$, as operators in $L^2(E,\mu)$ and $L^2(F,\nu)$, are unitarily equivalent. An immediate consequence of Theorem~\ref{thm:unitary} is the following corollary.

\begin{cor}
	Under the conditions of Theorem~\ref{thm:unitary}, suppose that $(P_t)$ and $(Q_t)$ are Feller semi-groups with infinitesimal generators $L$ and $\hat L$. Then $L$ and $\hat L$, as (self-adjoint) operators in $L^2(E,\mu)$ and $L^2(F,\nu)$, are unitarily equivalent; in particular, they have the same spectrum $\sigma(L) = \sigma(\hat L) \subset (-\infty,0]$.
\end{cor}

Before we turn to the proof of Theorem~\ref{thm:unitary}, let us point out that another interesting connection between duality and unitary transformation appears, without reversibility, for stochastic processes on locally compact abelian groups; the duality function is chosen as the kernel of the Fourier transform -- see \cite{saloff-coste01} for an introduction to random walks and diffusions on groups and  \cite[Chapter 8]{hewitt-ross70} for theoretical background on harmonic analysis. This type of duality was considered by Holley and Stroock \cite{holley-stroock76,holley-stroock79}, who noticed that this setup includes dualities between diffusions and random walks as well as dualities of interacting particle systems.

\begin{example}[Fourier duality for diffusions on a circle and random walk \cite{holley-stroock79}]
	Let $E=\{ \exp(\mathrm{i} \theta) \mid \theta \in \R\}=U(1)$, $F= \Z$, and consider the diffusion in $E$ with formal generator 
	$L = (1-\cos \theta) \partial^2/\partial \theta^2$, and the random walk on $\Z$ with generator $(\hat Lf)(n) = n^2( f(n+1)+f(n-1) - 2 f(n))$. These two processes are dual with respect to $H(\theta,n) = \exp(\mathrm{i}n \theta)$.
\end{example}

\begin{example}[Annihilating duality revisited \cite{holley-stroock79}]
	Let $E=\{-1,1\}^\Z$ and $F = \{A\subset \Z\mid |A|<\infty\}$. $E$ and $F$ have natural group structures given by componentwise multiplication  and symmetric difference $A\Delta B = (A\backslash B)\cup (B\backslash A)$. The analogue of the Fourier kernel is 
	$H(\sigma, A) = \prod_{x\in A} \sigma_x$. Processes that are dual with respect to this function correspond to annihilating dual particle systems: indeed, setting  $B(\sigma):= \{k\in \Z\mid \sigma_k =-1\}$, we have
	\begin{equation*}
		H(\sigma,A) = \prod_{x\in A} \sigma_x = (-1)^{|A\cap B(\sigma)|} 
			= 1 - 2\, \mathbf{1}_{\{ |A\cap B(\sigma)|\ \text{is odd} \}},
	\end{equation*}
	and we recognize the annihilating duality function from Section~1.2. 
\end{example}

Fourier tranforms on locally compact groups define, up to a multiplicative constant, unitary operators between $L^2$ spaces, the measures being chosen as Haar measures rather than invariant measures of processes \cite[Chapter 8]{hewitt-ross70}. 
Thus Fourier duality is, formally, a unitary transformation between semi-groups in  suitable $L^2$-spaces. 

%

\begin{proof}[Proof of Theorem~\ref{thm:unitary}]
	Let $T:L^2(F,\nu) \to L^2(E,\mu)$ be the integral operator
	$(T g)(x):=\int_F H(x,y) g(y) \nu(\dd y)$. Write $||\cdot||_p$ for the $L^p$-norms and $\la \cdot, \cdot\ra$ for the scalar product in $L^2(E,\mu)$ and $L^2(F,\nu)$.  Using Cauchy-Schwarz, we have 
	\begin{equation*} 
		||T g||_2 \leq \sqrt{\mu(E)} ||T g||_\infty \leq \sqrt{\mu(E) \nu(E)} ||H||_\infty ||g||_2, 
	\end{equation*} 
	thus $T$ is a bounded operator. $T$ is defined so that  $\la f, Tg \ra = B_H( f\mu, g\nu)$. The non-degeneracy of $B_H$ shows that both $T$ and $T^*$ are injective and have dense ranges, e.g., $\ran T = \{ T g \mid g \in L^2(F,\nu)\}$ is dense in $L^2(E,\mu)$.
	 
	Because of the reversibility of $\mu$ and $\nu$, we have $P_t^*(f\mu) = (P_t f) \mu$ and $Q_t^* (g \nu) = (Q_tg)\nu$; moreover, $P_t$ and $Q_t$ are self-adjoint in the respective $L^2$-spaces.  
	Therefore 
	\begin{equation*} 
		\la P_t f,T g \ra = B_H\bigl( (P_t f)\mu, g \nu \bigr)= B_H\bigl( P_t^*(f\mu), g\nu) 
			= \la f, T Q_t g \ra 
	\end{equation*}  
	for all $f\in L^2(E,\mu)$ and $g\in L^2(F,\nu)$. It follows that $P_t T = T Q_t$, i.e., $T$ intertwines $P_t$ and $Q_t$. 
	
	If $T$ is unitary, we are done. If $T$ is not unitary, we construct a unitary by adapting the constructions from the polar decomposition \cite[Section VI.4]{reed-simon}. 
	Let $A:= \sqrt{T^*T}$ be the positive semi-definite operator in $L^2(F,\nu)$ 
	with $A^2 =T^*T$. We note that $||A f||_2^2 = \la g, T^* T g\ra = ||T g||_2^2$, thus $A$ is injective because $T$ is. Moreover, $\ran A$ is dense: this follows from $\ran T^* T = \ran A^2 \subset \ran A$ and 
	\begin{align*} 
		g\in\ran (T^*T)^\perp &\Leftrightarrow \forall h \in L^2(F,\nu):\, \la T g, T h \ra = 0 \\
				& \Leftrightarrow Tg \in (\ran T)^\perp  \\
				& \Leftrightarrow T g =0 \quad \text{(because $\ran T$ is dense)} \\
				& \Leftrightarrow g =0 \quad \text{(because $T$ is injective).}
	\end{align*}
	By adapting the proof of Theorem VI.10 in \cite{reed-simon}, we find that there is 
	a unique unitary operator $U:L^2(F,\nu) \to L^2(E,\mu)$ such that $T = UA$. 
		
	Next, we want to deduce from $P_t T = T Q_t$ the intertwining relation $P_t U= UQ_t$, using that $P_t$ and $Q_t$ are self-adjoint. Write $P_t = \int \lambda \dd E_\lambda$ and $Q_t = \int \lambda \dd \hat E_\lambda$ for the spectral decompositions of $P_t$ and $Q_t$ \cite[Section VII]{reed-simon}; we suppress the $t$-dependence in the notation.  A simple induction over $n$ shows that $P_t^n T = T Q_t^n$ for all $n\in \N$, which in turn implies that the spectral projections are intertwined, $E_\lambda T = T \hat E_\lambda$. It follows that $T (\ker \hat E_\lambda) \subset \ker E_\lambda$ and $T(\ran E_\lambda) \subset \ran E_\lambda$. Applying the polar decomposition construction to the restrictions, we obtain unitaries $U_1:\ker \hat E_\lambda \to \ker E_\lambda$ and $U_2:\ran \hat E_\lambda \to \ran E_\lambda$. The uniqueness statement in the polar decomposition can be used to show that $U_1$ and $U_2$ must coincide with the restrictions of $U$ to the corresponding subspaces, and it follows that $E_\lambda U = UE_\lambda$. Since this holds for every $\lambda \in \R$, we deduce $P_t U = U Q_t$. 
\end{proof}

\section{Pathwise duality}\label{sec:pathwise}
In this section, we discuss various notions of pathwise duality, which strengthen the basic notion of dual processes.
Often one is interested to know whether a duality holds in some pathwise sense, which has to be specified. We will introduce strong pathwise duality as well as weaker notions. As a first step, we show
that in principle dual processes can always be coupled. In the next proposition, given two probability measures $\mu$ and $\nu$ on $E$ and $F$,  $\P_\mu$  and $\P^\nu$ refer to the laws of  $(X_s)$  and $(Y_s)$ with initial conditions $\mu$ and $\nu$. 

\begin{prop}\label{prop:coupling-duality}
	Let $(X_t)$ and $(Y_t)$ be two Markov processes with respective Polish state spaces $E$ and $F$, and $H:E\times F\to \R$ measurable and bounded. Then $(X_t)$ and $(Y_t)$ are dual with respect to  $H$ if and only if for all $t>0$ and every choice of initial conditions $\mu$ of $(X_s)$ and $\nu$ of $(Y_s)$ there are processes $(\tilde X_s)_{s\in [0,t]}$ and  $(\tilde Y_s)_{s\in [0,t]}$, defined on a common probability space $(\Omega,\mathcal{F},\P)$  such that 
\begin{itemize}
	\item [(i)] The finite-dimensional distributions $(\tilde X_s)_{s\in [0,t]}$ and $(\tilde Y_s)_{s\in [0,t]}$ under $\P$ agree with those of $(X_s)_{s\in [0,t]}$ under $\P_\mu$ and $(Y_s)_{s\in [0,t]}$ under $\P^\nu$, respectively. 
	\item [(ii)] For all $s\in [0,t]$, 
	\begin{equation} \label{eq:eh}
		\E H(\tilde X_0,\tilde Y_t) = \E H(\tilde X_s,\tilde Y_{t-s}) 
			= \E H(\tilde X_t,\tilde Y_0).  
	\end{equation}
\end{itemize} 
\end{prop}
The probability space $(\Omega,\mathcal{F},\P)$ may depend on $t$, $\mu$, and $\nu$. 

\begin{proof} 
We start with the easy direction. Suppose that for every $\mu$, $\nu$, and $t>0$ we can find $(\Omega,\mathcal{F},\P)$, $(\tilde X_s)_{s\in [0,t]}$ and $(\tilde Y_s)_{s\in [0,t]}$ satisfying (i) and (ii). 
Applying (i) and (ii) to  $\mu = \delta_x$, $\nu = \delta_y$ we have $\tilde X_0 = x$ and $\tilde Y_0= y$ $\P$-almost surely, and 
\begin{equation*}
	\E_x H(X_t,y) = \E H (\tilde X_t, \tilde Y_0) = \E H (\tilde X_0, \tilde Y_t) = \E^y H (x,Y_t). 
\end{equation*} 
It follows that $(X_t)$ and $(Y_t)$ are dual with respect to $H$. 

Conversely, suppose that $X = (\Omega_1, \mathcal{F}_1, (X_s)_{s\geq 0 }, \{\P_x\}_{x\in E} )$ and $Y = (\Omega_2,\mathcal{F}_2,(Y_s)_{s\geq 0}, \{\P^y\}_{y\in F} )$ are dual  with respect to $H$. 
Fix $t>0$ and $\mu$ and $\nu$, two probability measures on $E$ and $F$ respectively. Let $\P_\mu$ and $\P^\nu$ be the measures on $(\Omega_1, \sigma( \{X_s,\ s\geq 0\}))$ and $(\Omega_2, \sigma(\{Y_s,\ s\geq 0\})$ defined by $\P_\mu (A)= \int_E \mu(\dd x) \P_x(A) $ and $\P^\nu(B) = \int_F \nu(\dd y) \P^y(B)$. 

Let $\Omega:=\Omega_1\times \Omega_2,$ ${\mathcal F}:=  \sigma( \{X_s,\ s\geq 0\}) \otimes \sigma(\{Y_s,\ s\geq 0\})$, and $\P:=\P_\mu \otimes \P^\nu.$  Set $\tilde X_s(\omega_1,\omega_2):= X_s(\omega_1)$ and  $\tilde Y_s(\omega_1,\omega_2): = Y_s(\omega_2)$. Then $(\Omega,\mathcal{F},\P)$,  $(\tilde X_s)_{s\in [0,t]}$,
 and $(\tilde Y_s)_{s\in [0,t]}$ satisfy (i).  Moreover, for every $s\in [0,t]$, $\tilde X_s$ and $\tilde Y_{t-s}$ are independent. The joint law of $(\tilde X_s,\tilde Y_{t-s})$ is therefore $\mu P_s \otimes \nu Q_{t-s}$, with $(P_s)$ and $(Q_s)$ the semi-groups of $X$ and $Y$. In view of Eq.~\eqref{eq:duality_semigroup}, we thus have for every $s\in [0,t]$ 
\begin{equation*}
	\begin{split}
		\E H(\tilde X_s, \tilde Y_{t-s}) & = \int_{E\times F} \Bigl \{ 
				\int_E P_s(x,\dd x') \Bigl(\int_F Q_{t-s}(y,\dd y') H(x',y')\Bigr)
\Bigr\} \mu(\dd x) \nu(\dd y) \\
			& = \int_{E\times F} \Bigl \{ 
				\int_E P_s(x,\dd x') \Bigl(\int_E P_{t-s}(x',\dd y'') H(x'',y)
\Bigr\} \mu(\dd x) \nu(\dd y) \\
			& = \int_{E\times F} \Bigl \{ 
				\int_E P_t(x,\dd x'') H(x'',y) \Bigr\} \mu(\dd x) \nu(\dd y) \\
		&=\E H(\tilde X_t,\tilde Y_0). 
	\end{split}
\end{equation*}
This proves (ii).
\end{proof}

For applications and to infer properties of one process to its dual, simple duality and the trivial coupling of $(X_t)$ and $(Y_t)$ on the product space is often not enough. We would like to find a stronger, pathwise coupling of the two processes. In general, two processes which are dual with respect to a duality function $H$ are called \emph{pathwise dual} if they can be coupled using an (explicitly constructed) auxiliary driving stochastic process, so that, for given initial conditions, one proces is running forward, the other running backward in time driven by the same realization of the auxiliary process. This concept has been widely used in many concrete cases, but it seems to us that there is no general treatment or even a generally accepted precise definition of this somewhat vague notion in the literature so far. It should be mentioned that in the context of duality with respect to a measure (recall definition \ref{def:duality2}), there is a powerful notion of pathwise duality via Kuznetsov realization of Markov processes, (cf. \cite{getoor}, Thm. 2.4), which will not be discussed here. 

In this chapter, we define some notions of pathwise duality with respect to a function, and discuss a few general concepts for the construction of pathwise duals, in particular the \emph{graphical representation} for dual interacting particle systems. However, there are few general results, and in many questions we will restrict ourselves to some concrete examples, mostly from interacting particle systems.

\subsection{Strong pathwise duality}\label{sec:pathwise_strong}
To start with, we propose the following definition of \emph{strong} pathwise duality; the reader should contrast it with Proposition~\ref{prop:coupling-duality}. In this context, some statements can be formulated more easily in terms Markov families instead of Markov processes. Recall that every Markov family defines a Markov process, via the canonical construction \cite{dynkin}.

\begin{definition}[Strong pathwise duality] \label{def:strongly-pathwise}
	Let $(X_t)$ and $(Y_t)$ be two Markov processes with Polish state spaces $E$ and $F$, and $H:E\times F\to \R$ measurable and bounded. Suppose that for every $t>0$ there are families of processes $\{(X_s^x)_{s\in [0,t]}\}_{x\in E}$ and $\{(Y_s^y)_{s\in [0,t]}\}_{y\in F}$  defined on a common probability space $(\Omega,\mathcal{F},\P)$  such that the following holds: 
	\begin{itemize}
		\item [(i)] For all $x\in E$ and $y\in F$, the finite dimensional distributions of $(X_s^x)_{s\in [0,t]}$ and $(Y_s^y)_{s\in [0,t]}$ under $\P$ agree with those of $(X_s)_{s\in [0,t]}$ under $\P_x$ and $(Y_s)_{s\in [0,t]}$ under $\P^y$, respectively. 
		\item [(ii)] For all $s\in [0,t]$ and all $x\in E$, $y\in F$, 
		\begin{equation*}
			H(x,Y_t^y) = H(X_s^x,Y_{t-s}^y)= H(X_t^x,y) \quad \P\text{-a.s.}
		\end{equation*} 
	\end{itemize} 
	Then $(X_t)$ and $(Y_t)$ are called \emph{strongly pathwise dual} with respect to $H$. 
\end{definition} 
In most examples, $(\Omega,\mathcal{F},\P)$ can be chosen independent of $t$, but the precise form of the maps $X_s^x$, $Y_s^x$ will depend on the fixed time horizon $t$. Moreover, in the examples below $X_s^x$ and $Y_{t-s}^y$ are independent, and there is a collection 
 of independent random variables $(Z_s)_{s\in [0,t]}$ -- for example, Poisson arrows in graphical representations -- such that $X^x_s$ is measurable with respect to $\sigma( Z_\alpha,\ \alpha \in [0,s])$ and $Y^y_{t-s}$ is measurable with respect to $\sigma(Z_\alpha,\ \alpha \in [s,t])$. 

\begin{lemma}
Let $(X_t)$ and $(Y_t)$ be strong pathwise duals with respect to $H.$ Then they are dual.
\end{lemma}

\begin{proof} 
For every $x\in E, y\in F,$ by strong pathwise duality there exist $(X_s^x), (Y^y_s)$ such that
$\E_x[H(X_t, y)]=\E[H(X^x_t, Y^y_0)]=\E[H(X_0^x, Y_t^y)]=\E^y[H(x, Y_t)].$
\end{proof}
As opposed to the coupling in Proposition \ref{prop:coupling-duality}, which worked for each fixed set of initial conditions, we now ask for a coupling that works for all initial conditions at once. The important task is to explicitly describe $(\Omega, \mathcal F, \P)$ in a useful way. There are many situations where this can be achieved in terms of a graphical representation.

\begin{example} [Absorbed and reflected random walks]
Let $(X_n)$ be the Markov process associated with simple symmetric random walk on $\N_0$ absorbed upon first hitting $0.$ Let $(Y_n)$ be discrete time simple symmetric random walk reflected at 0 via the transition rule $P(0,1)=P(0,0)=1/2$. As for Brownian motion, it is a classical result that $(X_n)$ and $(Y_n)$ are Siegmund duals (compare also Section \ref{sec:monotonicity}). We give a simple pathwise construction of this duality. Let $W_n, n\in\N_0,$ be iid random variables 
on a space $(\Omega, \mathcal F, \P)$ with $\P(W_1=-1)=\P(W_1=1)=\frac{1}{2}.$ We think of $W_n$ as a sequence of arrows pointing upwards if $W_n=1$ and downwards if $W_n=-1$ (cf. Figure 1). Fixing $N\in\N,x,y\in\N$ and a realization $(W_n)_{0\leq n\leq N-1},$ we define $X^x_0:=x,$ and 
$$X^x_{n+1}:=\begin{cases} X^x_n+W_n &\mbox{if }X^x_n>0\\
										0&\mbox{otherwise.}
										\end{cases}$$
In words, $(X^x_n)$ is constructed by following the arrows up and down until the first hitting time of 0, after which it stays in 0. $(Y^y_n)$ is constructed using the same realization of $(W_n),$ but going backward in time, and following the arrows in the converse direction, unless $Y^y_n=0,$ in which case the process either stays in 0 or is reflected. That is, we set $Y^y_0:=y,$ and
$$Y^y_{n}=\begin{cases}Y^y_{n-1}-W_{N-n}	&\mbox{if } Y^y_{n-1}>0,\\
				\max\{Y^y_{n-1}-W_{N-n},0\}&\mbox{otherwise.}
				\end{cases}$$
By construction, $(X^x_n)$ and $(Y^y_n)$ have the same finite dimensional distribution as $(X_n)$ and $(Y_n)$ under $\P_x, \P^y,$ respectively.
Since we have used the same arrows in the construction of both proceses, the paths of $(\hat{Y}^y_n)_n$ defined by $\hat{Y}^y_n:=Y_{N-n}$ and $(X^x_n)$ don't intersect, and therefore $x\leq Y^y_N$ if and only if $X^y_N\leq y,$ hence 
\[1_{\{x\leq Y^y_N\}}=1_{\{X^x_n\leq Y^y_{N-n}\}}=1_{\{X^x_N\leq y\}}\quad a.s.\]

\begin{center}
\setlength{\unitlength}{1cm}
\begin{picture}(12,8)
\put(1,1){\line(0,1){5.5}}
\put(1,1){\line(1,0){10}}
\put(11,1){\line(0,1){5.5}}
\multiput(2,0.9)(1,0){9}{\line(0,1){0.2}}
\put(0.9,0.6){$0$}
\put(10.9,0.6){$N$}
\put(1,7){\vector(0,1){1}}
\put(3,7){\vector(0,1){1}}
\put(7,7){\vector(0,1){1}}
\put(8,7){\vector(0,1){1}}
\put(10,7){\vector(0,1){1}}
\put(2,8){\vector(0,-1){1}}
\put(4,8){\vector(0,-1){1}}
\put(5,8){\vector(0,-1){1}}
\put(6,8){\vector(0,-1){1}}
\put(9,8){\vector(0,-1){1}}
\put(0,7.2){$(W_n)$}
\put(1,3){\circle*{0.1}}
\put(0.6,3){$x$}
\put(0.4,4){$\hat{Y}_0$}
\put(11,4){\circle*{0.1}}
\put(11.1,4){$y_2$}
\put(11.1,1){$X_N$}
\linethickness{0.4mm}
\put(1,3){\line(1,1){1}}
\put(2,4){\line(1,-1){1}}
\put(3,3){\line(1,1){1}}
\put(4,4){\line(1,-1){1}}
\put(5,3){\line(1,-1){1}}
\put(6,2){\line(1,-1){1}}
\put(7,1){\line(1,0){4}}
\put(11,4){\line(-1,-1){1}}
\put(10,3){\line(-1,1){1}}
\put(9,4){\line(-1,-1){1}}
\put(8,3){\line(-1,-1){1}}
\put(7,2){\line(-1,1){1}}
\put(6,3){\line(-1,1){1}}
\put(5,4){\line(-1,1){1}}
\put(4,5){\line(-1,-1){1}}
\put(3,4){\line(-1,1){1}}
\put(2,5){\line(-1,-1){1}}
\put(11,2){\circle*{0.1}}
\put(11.1,2){$y_1$}
\put(11,2){\line(-1,-1){1}}
\put(10,1){\line(-1,1){1}}
\put(9,2){\line(-1,-1){1}}
\put(8,1){\line(-1,0){1}}
\put(7,1){\line(-1,1){1}}
\put(6,2){\line(-1,1){1}}
\put(5,3){\line(-1,1){1}}
\put(4,4){\line(-1,-1){1}}
\put(3,3){\line(-1,1){1}}
\put(2,4){\line(-1,-1){1}}
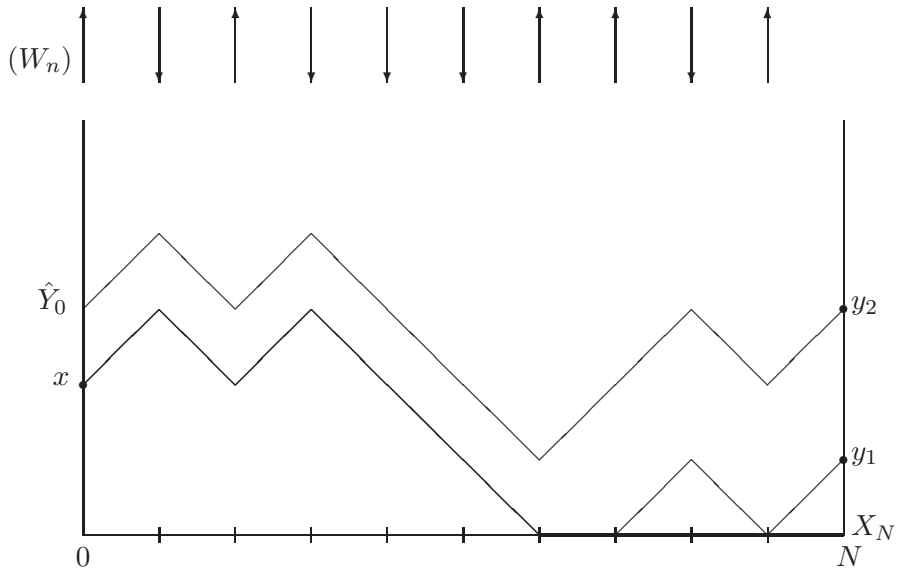
\put(1,0){\rm {Figure 1: \small{Absorbed and reflected random walks}}}
\end{picture}
\end{center}
\end{example}

This is an example of the following general result:

\begin{prop} [Siegmund duals are pathwise dual] Let $(X_t)$ and $(Y_t)$ be right continuous Markov processes in discrete or continuous time on a totally ordered state space, which are dual with respect to $H(x,y)=\mathbf{1}_{\{x\leq y\}}.$ Then they are strongly pathwise dual.
\end{prop}

The proof of this result is given in
Clifford and Sudbury \cite{CliffordSudbury85}, where a general construction in the spirit of the above example is given.\\
Clearly strong pathwise duality implies duality, but not every duality is strongly pathwise. We now give an example of a duality which is not strongly pathwise, but later we will see that also in this case there is an underlying pathwise construction.\\

\begin{example}[Dual but not strongly pathwise dual] Let $(X_t)$ be the Wright-Fisher diffusion taking values in $[0,1]$ and $(N_t)$ the block-counting process of a Kingman coalescent (values in $\N$), see e.g. \cite{etheridge}. It is well-known that $(X_t)$ and $(N_t)$ are dual with respect to $H(x,n)= x^n$. However, the 
 duality is not strongly pathwise. To see why, let $x\in (0,1)$, $n\in \N$
and $t>0$. Suppose that $(X_s(\omega))_{s\geq 0}$ started in $x$ and $(N_s(\omega))_{s\geq 0}$ started in $n$ are defined on a common probability space $(\Omega,\P)$. Let 
\begin{equation*}
	\mathcal{E}:=\bigl \lbrace \omega \in \Omega \mid X_t(\omega)^n = x^{N_t(\omega)} \bigr \rbrace. 
\end{equation*}
We know that $N_t(\omega) \leq n$ almost surely, since the number of blocks for  coalescent decreases. For $\omega \in \mathcal{E} \cap \{N_t \leq n\}$, 
we have $X_t(\omega)^n = x^{N_t(\omega)}  \geq x^n$ and therefore  $X_t(\omega) \geq x$. 
Thus 
\begin{equation} \label{eq:counter}
	\P(\mathcal{E}) \leq \P( X_t(\omega) \geq x) <1.
\end{equation}
The last inequality is strict because the paths of the Wright-Fisher diffusion \noemi{are} not monotone functions, unless started in the absorbing states $0$ and $1$. Eq.~\eqref{eq:counter} shows that the moment duality of $(X_t)$ and $(N_t)$  is not strongly pathwise. \noemi{However, there is an underlying pathwise structure in a sense that will become more evident in Sections \ref{sec:pathwise_weak} and \ref{sec:pathwise_rescaled}.}
\end{example}

The construction given in the example of absorbed and reflected random walks is probably the easiest example of a general construction of the underlying auxiliary processes for (strong) pathwise dualities, which, depending on the field, are also called \emph{graphical representation} or \emph{driving sequence}. We indicate the general idea to construct such a representation.

\begin{example}[Pathwise duality, discrete time]
Let $(X_n)_{n\in\N}, (Y_n)_{n\in\N}$ be Markov processes in discrete time with state space
 $E$ and $F$ respectively. Assume that there exists a sequence of iid random variables $(W_n)_{n\in\N}$ on $\Omega,$ and transition functions
$f:E\times \R\to E, g:F\times \R\to F, H:E\times F\to\R$ such that

$$X_n=f(X_{n-1}, W_n),\; Y_n=g(Y_{n-1}, W_n),$$
that is, given the values of $W_n$ and $X_{n-1},$ the value of $X_n$ is uniquely determined (and not random any more). See for example \cite{AsmussenSigman}. Assume further that
\begin{equation}
\label{eq:dual_mechanisms_discrete}
H(f(x, w),y)=H(x,g(y,w))\,\forall x\in E,y\in F, w\in\R,
\end{equation}
then $(X_n)$ and $(Y_n)$ are pathwise dual: Fix a realization of $(W_n).$ Use $(W_n)_{n=,0...,N}$ as a driving sequence for $(X_n)_{n=0,...,N}$ and $(W_{N-n})_{n=0,...,N}$ for $(\hat{Y}_n)_{n=0,...,N}.$ Clearly $\hat{Y}$ is equal in distribution to $Y,$ and by (\ref{eq:dual_mechanisms_discrete}), 
$$H(x,\hat{Y}_N)=H(x,g(\hat{Y}_{N-1}, W_1))=H(f(x,W_1),\hat{Y}_{N-1})=H(X_1, \hat{Y}_{N-1})\,a.s.$$
Iterating this proves strong pathwise duality. The transition probabilites of $(X_n)$ are given by
$$P(X_n=y\,|\,X_{n-1}=x)=\P(f(x,W_n)=y)=\P(f(x,W_1)=y).$$ 
Hence the Markov chains are necessarily time-homogeneous. 
\end{example}

\begin{example}[Continuous time, discrete space]
Clearly this example in discrete time can be extended to continuous time and discrete state space, by considering independent Poisson processes with rates $\lambda_i, 1\leq i\leq k$ for some $k\geq 1,$ adapted to the jump rates of the processes, and transition functions $f_i:E\to E, g_i:F\to F.$ Assume that $(X_t),(Y_t)$ change their state at a jump time $\tau$ of the $i-$th Poisson process according to the rule
\begin{equation}\label{eq:pw_jump}X_\tau=f_i(X_{\tau-}), Y_\tau=g_i(Y_{\tau-}).\end{equation}
This is certainly well-defined if $|E|,|F|<\infty.$ For countably infinite state space, the construction works if the interactions are sufficiently local \cite{Harris}. If $H(f_i(x),y)=H(x,g_i(y))$ for all $x\in E, y\in F,$ then $X$ and $Y$ are (strongly pathwise) dual with duality function $H.$
\end{example}

In the case of spin systems $(E=\{0,1\}^G)$ and coalescing or annihilating duals, this kind of construction goes back to Harris \cite{Harris} and is of widespread use. It is usually referred to as \emph{graphical representation}, since it can be represented by drawing a line of length $t$ for every element of $G,$ and representing the Poisson processes by arrows between pairs of lines. A  detailed account can be found in Griffeath \cite{Griffeath} and Liggett \cite{Liggett05}, see also \cite{Sudbury}. Usually, the interpretation of the mechanisms $f_i, g_i$ is such that one thinks of a particle at the tail of an arrow in the graphical representation having some effect on the configuration at the tip, for example by jumping there, or by branching, and subsequent coalescence, or annihilation, or death. The rates of the Poisson processes then naturally have the interpretation of giving a rate \emph{per particle} for some event to happen. \\
We now give a variant of this construction, which differs only slightly from the previous one in the sense that the Poisson process gives us the rate of an event happening\emph{ per pair of sites }on a graph, which may or may not be occupied by a particle. We discuss this construction here in some detail, in the case of a complete graph. 

\begin{example}[Pathwise duality for interacting particle systems via basic mechanims]

Let $N\in\N,$ and let $E_N:=\{0,1\}^N.$ Consider duality functions of the form $$H(x,y)=q^{|x\wedge y|}$$
for $q\in\R\setminus\{1\}.$ We call such a duality a \emph{$q-$duality}. Special cases are $q=0,$ which leads to a coalescing duality, and $q=-1$, which is equivalent to an annihilating duality, see \cite{SudburyLloyd95} for a discussion of this type of duality functions.\\
The graphical representation is constructed as follows. For each $i\in\{1,...,N\},$ draw a vertical line of length $T,$ which represents time up to a finite end point $T.$ For each such pair $(i,j), i,j\in\{1,...,N\}$ run $m\in\N$ independent Poisson processes with parameters $(\lambda_{ij}^k), k=1,...,m.$ At the time of an arrival draw an arrow from the line corresponding to $i$ to the line corresponding to $j,$ marked with the index $k$ of the process. For each $k,$ define functions $f^k,g^k:\{0,1\}^2\to\{0,1\}^2.$ A Markov process $(X_t^N)$ with c\`adl\`ag paths is then constructed by specifying an initial condition $x=(x_i)_{i=1,...,N},$ and the following dynamics: $X_t^N=x$ for all $t<\tau_1,$ where $\tau_1$ is the time of the first arrow in the graphical representation (which is clearly well-defined, since we consider a finite time horizon, and a finite number $N$ of lines). If this arrow points from $i$ to $j$ and is labelled $k,$ then the pair $(x_i,x_j)$ is changed to $f^k(x_i,x_j),$ and the other coordinates remain unchanged. Go on until the next arrow, and proceed exactly in the same way. The dual process $(Y_t^N)$ is constructed using the same Poisson processes, but started at the final time $T>0,$ running time backwards, inverting the order of all arrows, and using the functions $g^k$ instead of $f^k.$ 
\end{example}
Following \cite{AlkemperHutzenthaler}, we call the functions $f^k,g^k$ \emph{basic mechanisms}. The transition functions in the classical graphical representation of interacting particle systems as in \eqref{eq:pw_jump} can be composed by suitable basic mechanisms. For $x=(x_1,x_2)\in\{0,1\}^2$ we use the notation $x^\dagger:=(x_2,x_1)$; the dagger accounts for the reversal of an arrow.
Two basic mechanisms $f,g:\{0,1\}^2\to\{0,1\}^2$ are called \emph{$q$-dual mechanisms} if and only if
\begin{equation}
\label{eq:dual_mechanism}
q^{|x\wedge (g(y^\dagger))^\dagger|}=q^{|f(x)\wedge y|}\,\,\,\,\forall x,y\in\{0,1\}^2.
\end{equation}
As an example, let $f^R:\{0,1\}^2\to\{0,1\}^2$ be defined by $f^R(0,*)=(0,0), f^R(1,*)=(1,1), *\in\{0,1\}.$ We call this the resampling mechanism, it corresponds to a voter-transition, where the particle in the second place takes the opinion (or type) of the particle in the first position. It is straightforward to check that $f^R$ is 0-dual to the coalescence mechanism $f^C, $ with $f^C(1,*)=(0,1), *\in \{0,1\},$ $f(0,0)=(0,0), f^C(0,1)=(0,1)$, and $(-1)-$dual to the annihilation mechanism $f^A,$ with $f^A(0,0)=(0,0), f^A(0,1)=(0,1), f^A(1,0)=(0,1), f^A(1,1)=(0,0).$\\
It is clear that two processes constructed using $q-$dual mechanisms are strongly pathwise $q-$dual processes. 
\begin{center}
\setlength{\unitlength}{1cm}
\begin{picture}(7,5)
\put(1.5,1.5){\line(0,1){3}}
\put(5.5, 1.5){\line(0,1){3}}
\put(1.4,1){$x_1$}
\put(1.4, 4.7){$y_1$}
\put(5.4,1){$x_2$}
\put(5.4, 4.7){$y_2$}
\put(1.5,3.6){\vector(1,0){4}}
\multiput(1.7, 3.4)(0.3,0){13}{\line(1,0){0.2}}
\put(1.7, 3.4){\vector(-1,0){0.2}}
\put(0,3.8){$(f(x))_1$}
\put(5.6,3.8){$(f(x))_2$}
\put(0,3.0){$(g(y^\dagger))_2$}
\put(5.6, 3.0){$(g(y^\dagger))_1$}
\put(0.9,0.3){\rm{Figure 2: \small{Basic mechanism}}}
\end{picture}
\end{center}
\begin{lemma}\cite{JansenKurt}\label{lem:dual_mech_proc}
Fix $m\in\N$, $q\in\R\setminus\{1\}$ and $T>0$. For every $k=1,...,m$, let $f^k, g^k$ be $q-$dual basic mechanisms.  Let $(X^N_t), (Y_t^N)$ be Markov processes with state space $E_N,$ constructed using the mechanisms $f_k$, $k=1,\ldots,m$ and $g_k$, $k=1,\ldots,m,$ respectively, and Poisson processes with the same symmetric parameters $\lambda^k_{ij}= \lambda_{ji}^k$, $k=1,....,m$, $i,j=1,...,N.$ Then there exists a process $(\hat{Y}_t^N)$ such that for all $0\leq t\leq T,$
\begin{equation} \label{eq:dual_mech_proc}
	\hat{Y}_t^N\stackrel{d}{=}Y_t^N\,\,\,\mbox{
and }\,\,\,q^{|{X}_t\wedge \hat{Y}_0|}=q^{|X_0\wedge \hat{Y}_t|}\,\,\,\,a.s.
\end{equation}
\end{lemma}

\begin{proof}
Since we assume that the Poisson processes have the same rates, we can construct $\hat{Y}^N_t$ from the graphical representation of $(X_t^N),$ using the same realization of the Poisson processes, reversing time and the directions of all the arrows. It is clear from the construction that then $\hat{Y}_t^N\stackrel{d}{=}Y_t^N,$ and $q^{|{X}_t\wedge \hat{Y}_0|}=q^{|X_0\wedge \hat{Y}_t|}$ hold (see Figure 2). For some more details, in the case of coalescing mechanisms, compare the proof of Proposition 2.3 of \cite{AlkemperHutzenthaler}.
\end{proof}
 
We give a list of relevant dual mechanisms in the appendix. As an example, we consider the voter model.

\begin{example}[Voter model and coalescing random walks]
Let $E_N$ be the complete graph with $N$ vertices. In the above construction, we choose  $\lambda_{ij}=\lambda>0$ for all $1\leq i,j\leq N$, $i\neq j,$ and we use the resampling mechanism $f^R$ described above. This means, that at each jump of the $\lambda_{ij},$ the process $(X_t)=(X_t^1,...,X_t^N)$ changes in the following way: $X^j$ takes on the same value as $X^i,$ and all the other values remain unchanged. 
We know that this \noemi{process} is dual to a system of coalescing random walks, given by $(Y_t)=(Y_t^1,...,Y_t^N),$ where $Y^i_t=1$ if there is a particle at time $t$ at site $i.$ This process can indeed be constructed using the same driving Poisson processes, and the mechanism $f^C,$ which is the coalescing dual mechanism to $f^R.$ It has the effect that  at each arrival of $\lambda_{ij}$ the particle at site $j$ jumps to site $i$ and merges with the particle at that site. With this procedure, we obtain the well-known duality $\mathbf{1}_{\{X_0\wedge Y_t=\mathbf{0}\}}=\mathbf{1}_{\{X_t\wedge Y_0=\mathbf{0}\}}\;\;a.s..$
\end{example}

\subsection{Weaker notions of pathwise duality}\label{sec:pathwise_weak}

In this section, we weaken the notion of pathwise duality from above. Note that simple duality and strong pathwise duality can be cast into the following general form: If $(X_t)$ and $(Y_t)$ are strongly pathwise dual, then for every $t>0$ and for all $x\in E, y\in F$ they can be realized on a common probability space $(\Omega, \mathcal F, \P)$ such that for all $s\in [0,t]$, 
$$\E[H(X_t,Y_0)\;|\; \mathcal F]=\E[H(X_s, Y_{t-s})\;|\;\mathcal F]=\E[H(X_0,Y_t)\;|\;\mathcal F]\quad \P\mbox{-a.s.}$$
If they are dual in the usual sense, then for fixed initial conditions $x,y$ they can be realized on a common probability space $(\Omega, \mathcal F, \P)$ (note that $\P=\P_{x,y}$ usually depends on $x$ and $y$), such that 
$$\E[H(X_t,Y_0)\;|\; \sigma(\{\emptyset, \Omega\}) ]=\E[H(X_s,Y_{t-s})\;|\; \sigma(\{\emptyset,\Omega\}) ]=\E[H(X_0,Y_t)\;|\;\sigma(\{\emptyset,\Omega\}) ]\quad\P\text{-a.s.}$$

Interpolating between these two extreme cases and using families $(X_s^x)$ and $(Y_s^y)$ as in Definition~\ref{def:strongly-pathwise}  leads to the following definition.
\begin{definition} [Conditional pathwise duality] \label{def:conditional-pathwise}
	Let $(X_t)$ and $(Y_t)$ be two Markov processes with Polish state spaces $E$ and $F$, and $H:E\times F\to \R$ measurable and bounded. Suppose that for every $t>0$ there are families of processes $\{(X_s^x)_{s\in [0,t]}\}_{x\in E}$ and $\{(Y_s^y)_{s\in [0,t]}\}_{y\in F}$  defined on a common probability space $(\Omega,\mathcal{F},\P)$ and a $\sigma$-algebra $\mathcal{D}\subset \mathcal{F}$ such that the following holds: 
	\begin{itemize}
		\item [(i)] For all $x\in E$ and $y\in F$, the finite dimensional distributions of $(X_s^x)_{s\in [0,t]}$ and $(Y_s^y)_{s\in [0,t]}$ under $\P$ agree with those of $(X_s)_{s\in [0,t]}$ under $\P_x$ and $(Y_s)_{s\in [0,t]}$ under $\P^y$, respectively. 
		\item [(ii)] For all $s\in [0,t]$ and all $x\in E$, $y\in F$, 
		\begin{equation} \label{eq:conditional_duality}
			\E [H(X^x_t, y)\;|\;\mathcal{D}]=\E [H(X_s^x, Y_{t-s}^y)\;|\;\mathcal{D} ]=\E [H(x,  Y^y_t)\;|\;\mathcal{D}]\quad \P\text{-a.s.}
		\end{equation} 
	\end{itemize} 
	Then $(X_t)$ and $(Y_t)$ are called \emph{conditionally pathwise dual} with respect to $H$. 
\end{definition}

An example of this concept is given in Ex. 4.9 below.

Clearly conditional pathwise duality implies duality, and strong pathwise duality implies conditional  pathwise duality. Moreover,  by the tower property of the conditional expectation, if Eq.~\eqref{eq:conditional_duality} holds for a $\sigma$-algebra $\mathcal{D}$, it holds for every smaller $\sigma$-algebra $\mathcal{D'}\subset \mathcal{D}$.

Another interesting situation arises when one looks at functions of processes.  More precisely, we place ourselves in the situation of the following lemma. 

\begin{lemma} \label{lem:subprocess}
	Let $(X_t)$ be a Markov process with Polish state space $E$ and semi-group $(P_t)$. Let $\hat E$ be another Polish space and $f:E\to \hat E$ a measurable surjective function. Suppose that there is a transition kernel $\Lambda: \hat E\times \mathcal{B}(E) \to [0,1]$ and a Markov semi-group $(\hat P_t)$ on $\hat E$ such that 
\begin {itemize}
	\item For every $\hat x\in \hat E$, the probability measure $\Lambda(\hat x,\cdot)$ is supported on the preimage of $\hat x$, i.e., $\Lambda(\hat x, f^{-1}(\{\hat x\}) \bigr) = 1$.
	\item $\Lambda$ intertwines $(P_t)$ and $(\hat P_t)$: for all $t\geq 0$,  $\hat x \in E$, and measurable $A\subset E$, 
	\begin{equation} \label{eq:subprocess-intertwining}
		 \int_E \Lambda(\hat x, \dd x) P_t(x, A) = \int_{\hat E} \hat P_t(\hat x, \dd \hat z) \Lambda(\hat z,A).
	\end{equation}  
\end{itemize} 
	Then for every $\hat x\in \hat E$, the finite-dimensional distributions of $(f(X_t))_{t\geq 0}$ when $(X_t)$ has initial law $\mu^{\hat x} = \Lambda(\hat x,\cdot)$ are those of a Markov process $(\hat X_t)$ with semi-group $(\hat P_t)$ started in $\hat x$. 
\end{lemma} 
This notion is applied in Examples 4.7 and 4.8. By a slight abuse of notation we sometimes write $\hat X_t= f(X_t)$. 

\begin{proof} [Proof of Lemma~\ref{lem:subprocess}]
 Let $\hat x\in E$, $s\geq 0$, $t>0$, and $\hat A,\hat B\subset \hat E$ measurable. Then 
	\begin{align*}
		\E_{\mu^{\hat x}} \bigl[ f(X_s)\in \hat A,\, f(X_{t+s}) \in \hat B\bigr]
			& = \int \Lambda(\hat x, \dd x_0) P_s(x_0, \dd x_1)\mathbf{1}_{\hat A} \bigl( f(x_1)\bigr) P_t(x_1,\dd x_2) \mathbf{1}_{\hat B} \bigl( f(x_2)\bigr) \\
		& = \int  \hat P_s(\hat x, \dd \hat x_1)\Lambda(\hat x_1, \dd x_1)\mathbf{1}_{\hat A} \bigl( f(x_1)\bigr) P_t(x_1,\dd x_2) \mathbf{1}_{\hat B} \bigl( f(x_2)\bigr)\\
		& = \int  \hat P_s(\hat x, \dd \hat x_1) \mathbf{1}_{\hat  A} ( \hat x_1 ) \Lambda(\hat x_1, \dd x_1) P_t(x_1,\dd x_2) \mathbf{1}_{\hat B} \bigl( f(x_2)\bigr)\\
		& = \int \hat P_s(\hat x, \dd \hat x_1) \mathbf{1}_{\hat  A} \bigl( \hat x_1 \bigr) \hat P_t(\hat x_1,\dd \hat x_2) \Lambda(\hat x_2,\dd x_2) \mathbf{1}_{\hat B} \bigl( f(x_2)\bigr)\\
		& = \int \hat P_s(\hat x, \dd \hat x_1) \mathbf{1}_{\hat  A} ( \hat x_1)
		 \hat P_t(\hat x_1,\dd \hat x_2) \mathbf{1}_{\hat B}( \hat x_2). 
	\end{align*}   
	Hence when $(X_t)$ has initial law $\mu^{\hat x} = \Lambda(\hat x,\cdot)$, it has the same two-dimensional distributions as a Markov process $(\hat X_t)$ with semi-group $(\hat P_t)$ started in $\hat x$. The other finite-dimensional distributions can be computed in a similar way, and the claim follows. 
\end{proof} 

The kernel $\Lambda(\hat x,\dd x)$ has the interpretation of  a conditional probability, 
\begin{equation} \label{eq:lambda-meaning}
	\Lambda( \hat x,A) = \P( X_t \in A\mid f(X_t) = \hat x).   
\end{equation} 
Lemma~\ref{lem:subprocess} can be rephrased as follows. Suppose that the process $(X_t)$ preserves the conditional structure $\Lambda$, i.e., whenever the initial law of $(X_t)$ is such that Eq.~\eqref{eq:lambda-meaning} holds at $t=0$, then it holds for all $t\geq 0$. (In the examples below, this condition becomes: if $X_0$ is exchangeable, then $X_t$ is exchangeable for all $t\geq 0$.) Then $f(X_t)$ is Markovian with transition semi-group 
\begin{equation*}
	\hat P_t(\hat x,\hat A) = \P_{\Lambda(\hat x,\cdot)} \bigl( f(X_t) \in A\bigr). 
\end{equation*} 
%

Now let $(X_t)$, $(Y_t)$, $(\hat X_t)$, $(\hat Y_t)$ be Markov processes with respective Polish state spaces $E, F, \hat E, \hat F$ and semi-groups $(P_t)$, $(Q_t)$, $(\hat P_t)$, $(\hat Q_t)$. Let $f:E\to \hat E$ and $g:F\to \hat F$ be measurable surjective maps, and $\Lambda:\hat E\times \mathcal{B}(E)\to [0,1]$, $K:\hat F\times \mathcal{B}(F)\to [0,1]$ transition kernels such that $\Lambda, (P_t), (\hat P_t),f$ satisfy the support and intertwining conditions of Lemma~\ref{lem:subprocess}, and $K, (Q_t), (\hat Q_t),g$ satisfy the analogous conditions. 
Fix $H:E\times F\to \R$ measurable and bounded.
\begin{prop} \label{prop:funct-duality}
	In the situation described above the following holds: 
If $(X_t)$ and $(Y_t)$ are dual with respect to $H$, then $(\hat X_t)$ and $(\hat Y_t)$ are dual with respect to 
		\begin{equation*}
		              \hat H(\hat x,\hat y):= \int_{E\times F} \Lambda(\hat x ,\dd x) K(\hat y,\dd y) H(x,y).
            	\end{equation*} 
\end{prop} 
In view of Eq.~\eqref{eq:lambda-meaning} and its analogue for $(Y_t)$ and $K(\hat y,\dd y)$, the new function $\hat H$ should be thought of as a conditional expectation of $H$, see  Eq.~\eqref{eq:hath-meaning} below. 
\begin{proof} 
	First we note that $\hat H$ is measurable; this follows from the measurability of $\hat x\mapsto \Lambda(\hat x,A)$ and $\hat y\mapsto K(\hat y,B)$ for all measurable $A\subset E$ and $B\subset F$. Suppose that $(X_t)$ and $(Y_t)$ are dual with respect to $H$. Writing $P_t\otimes \mathrm{id}$ and $\mathrm{id}\otimes Q_t$ for the actions of the semi-groups on the first and second variables of $H$, we thus have $(P_t \otimes \mathrm{id}) H = (\mathrm{\id} \otimes Q_t) H$. Moreover, 
\begin{align*} 
	(\hat P_t \otimes \mathrm{id})\,  \hat H &= (\hat P_t \otimes \mathrm{id}) (\Lambda \otimes K)\,  H 
	= (\hat P_t \Lambda \otimes K) \, H  
	 = (\Lambda P_t \otimes K) \, H \\
	& =(\Lambda\otimes K) (P_t\otimes \mathrm{id}) \, H 
	= (\Lambda\otimes K) (\mathrm{id} \otimes Q_t) \, H 
	 = (\Lambda \otimes \hat Q_t K) \, H \\
	& = (\mathrm{id} \otimes \hat Q_t)\,  \hat H. 
\end{align*} 
Therefore $(\hat X_t)$ and $(\hat Y_t)$ are dual with respect to $\hat H$. 
\end{proof} 

It is instructive to give a probabilistic construction when $(X_t)$ and $(Y_t)$ are in strong pathwise duality with respect to $H$. Fix $t>0$ and let $(\Omega,\mathcal{F},\P)$ and $(X_s^x)$, $(Y_s^y)$ be as in Definition~\ref{def:strongly-pathwise}. Assume in addition that the maps $(x,\omega)\mapsto X_s^x(\omega)$ are measurable, similarly for $Y_s^y$. 
Enlarging the probability space we may assume that there are families of random variables $V_{\hat x}:\Omega \to E$ and $W_{\hat y}: \Omega \to F$ such that $V_{\hat x}$ has law $\Lambda(\hat x,\cdot)$, $W_{\hat y}$ has law $K(\hat y,\cdot)$, and the variables $V_{\hat x}$, $W_{\hat y}$ are independent between themselves and independent of the variables $X_s^x$, $Y_s^y$. We define 
\begin{equation*} 
	\hat X_s^{\hat x}(\omega):= f\Bigl(X_s^{V_{\hat x}(\omega)} (\omega)\Bigr), \quad 
	\hat Y_s^{\hat y}(\omega):= g\Bigl( Y_s^{W_{\hat y}(\omega)} (\omega)\Bigr).
\end{equation*} 
Thus we randomize the parameters $x$ and $y$, and apply the surjective maps $f$ and $g$. 
Lemma~\ref{lem:subprocess} shows that the finite-dimensional distribution of $(\hat X_s^{\hat x})_{s\in [0,t]}$ are those of $(f(X_s))_{s\in [0,t]}$ when $X_0$ has law $\Lambda(x',\cdot)$, and similarly for $(\hat Y^{\hat y}_s)$.  Moreover $\hat X_0^{\hat x} = \hat x$ and $Y_0^{\hat y} = \hat y$ $\P$-almost surely. 

Assume in addition that for every $s\in [0,t]$, the families $\{X_s^x\}_{x\in E}$  and $\{Y_{t-s}^y\}_{y\in F}$ are independent (this assumption is satisfied in the examples below). 
Then $\{X_s^{V_{\hat x}}\}_{x\in E}$ and $\{Y_{t-s}^{W_{\hat y}}\}_{y\in F}$ are independent as well, and in view of Eq.~\eqref{eq:lambda-meaning} we have 
\begin{equation*}
	\P\Bigl( X_s^{V_{\hat x}}\in \hat A,  Y_{t-s}^{W_{\hat y}(\omega)} \in \hat B\, \Big|\,  
		\sigma(\hat X_s^{\hat x},\hat Y_{t-s}^{\hat y} ) \Bigr)
		= \Lambda( \hat X_s^{\hat x}, \hat A) K(\hat Y_{t-s}^{\hat y},\hat B),\quad \P\text{-a.s.}
\end{equation*} 
It follows that 
\begin{equation} \label{eq:hath-meaning} 
 	\hat H(\hat X_s^{\hat x}, \hat Y_{t-s}^{\hat y}) 
= \E \Bigl[H(X_s^{V_{\hat x}},Y_{t-s}^{W_{\hat y}})\, \Big|\, 
\sigma({\hat X}_s^{\hat x},{\hat Y}_{t-s}^{\hat y} )
\Bigr],\quad \P\text{-a.s.} 
\end{equation}
As a consequence, 
\begin{equation*}
	\E \hat H(\hat X_s^{\hat x}, \hat Y_{t-s}^{\hat y}) = \E H(X_s^{V_{\hat x}},Y_{t-s}^{W_{\hat y}}) = \int_{E\times F} \Lambda( \hat x,\dd x) K(\hat y,\dd y) \E H(X_s^x,Y_{t-s}^y)
\end{equation*} 
is independent of $s\in [0,t]$, for all $\hat x\in \hat E$ and $\hat y\in \hat F$. 
This provides a pathwise construction for the duality of $(\hat X_t)$ and $(\hat Y_t)$, though the latter duality need not be strongly or conditionally pathwise in the sense of Definitions~\ref{def:strongly-pathwise} and~\ref{def:conditional-pathwise}. 

 \begin{example}[Moran model and block-counting process] Recall the strong pathwise duality of the voter model and coalescing random walk obtained at the end of the last section from a graphical representation. Fix $N\in \N.$ 
Let $A_t:=\{i:X_t^i=1\}$ and $B_t=\{i: Y_t^i\}=1,$ and recall from Example 1.2 that coalescing duality in this context means duality with respect to
$H(A,B)=1_{\{A\cap B=\emptyset\}}.$ Note that $|A_t|=N Z_t(\{1,...,N\}),$ where $Z_t$ is the empirical process of $X_t.$ We choose $f=g=|\cdot|,$ and define $\Lambda(a,\cdot)$ as the uniform distribution over all configurations $A$ such that $|A|=a,$ and analogously define $K(b,\cdot).$ This choice of measure means that we choose \emph{exchangeable} initial conditions. Assume that the driving Poisson processes $\lambda_{ij}$ in the graphical representation have the same intensity for all $i,j.$ Then $A_t$ and $B_t$ remain exchangeable for all times $t,$ and the intertwining relation of Lemma \ref{lem:subprocess} is satisfied. The function $\hat{H}$ from Prop. \ref{prop:funct-duality} then is
$$\hat{H}_N(a,b)=\P(|A\cap B|\;\big|\; |A|=a, |B|=b),$$ 
and therefore this proposition yields that $(|A_t|)$ and $(|B_t|)$ are dual with respect to $\hat{H}_N.$ This duality was derived in \cite{Moehle99}. Note that 
\begin{equation*}\begin{split}\hat{H}_N(a,b)=&\frac{(N-|A|)(N-|A|-1)\cdots (N-|A|-|B|+1)}{N(N-1)\cdots (N-|B|+1)} 
= \mathrm{Hyp}(N,|A|,|B|)(0), 
\end{split}
\end{equation*}
where $\mathrm{Hyp}(N,|A|,|B|)$ denotes the hypergeometric distribution function. 
\end{example}

\begin{example}[$q-$dual processes] Generalizing Example 4.7, we consider set-valued $q-$dual processes $(A_t), (B_t)$ (cf. Example 4.5) constructed from a graphical representation. This means we work with Poisson processes $\lambda_{ij}(s), i,j=1,...,N, 0\leq s\leq t,$ realized on some probability space $(\Omega^P, \mathcal F^P, \P^P).$ If the intensities are the same for each pair $i,j,$ then the sigma-algebra generated by the Poisson processes is exchangeable. Generalizing the situation of Example 4.7, we let $Z_{N,|A|,|B|}\sim\mathrm{ Hyp} (N, |A|,|B|),$ and define
$$\tilde{H}_{N,q}(|A|,|B|):=\E[q^{Z_{N,|A|,|B|}}]$$
the generating function of $Z_{N,|A|,|B|}.$ For fixed $a,b\in \N$ we can realize $(A_s), (B_s)$ on $(\Omega^P, \mathcal F^P, \P^P).$ Choosing $\Lambda, K$ as in Example 4.7, we obtain (cf. Lemma 4.5) the duality
$$\E [\tilde H_{N,q}(|A_t|,|B_0|)] = \E [\tilde H_{N,q}(|A_s|,|B_{t-s}|)], \quad 0\leq s\leq t.$$
\end{example}

In the graphical construction of $q$-dualities we have so far used different Poisson processes for different types of transitions or mechanisms. By basic properties of Poisson processe, we could use one Poisson process per pair of lines, even if we allow for more than one mechanism. The graphical representation then has only one type of arrow, and the processes are constructed from the arrow-configuration by following time forwards, and whenever an arrow is encountered, the type of transition is determined by an additional random variable. We give an example in the above context of $q-$duality of particle systems.

\begin{example}[Non-deterministic mechanisms, conditional duality]
In the set-up of $q-$dual mechanisms on $\{0,1\}^G$ we assume that $(X_t), (Y_t)$ are constructed from a graphical representation of Poisson processes (cf. Ex. 4.8) using two different mechanisms $f^1$ and $f^2.$ We want a mechanism $f^1$ to happen at rate $\alpha_1,$ and $f^2$ at rate $\alpha_2.$ Assume that the Poisson processes for each pair has the same rate $\lambda=a_1+\alpha_2.$ This means that at each time, the law of the arrow in the graphical representation is exchangeable. Assume that at a given time $\tau$ there is an arrow from $i$ to $j.$ We will give it type 1 corresponding to mechanism $f^1$ with probability $q=\frac{\alpha_1}{\alpha_1+\alpha_2},$ and type 2 with probability $1-q.$ We could think of the arrow as a random mechansim which is described by the transition matrix
$$P(x,y)=q1_{\{y=f^1(x)\}}+(1-q)1_{\{y=f^2(x)\}}.$$
In \cite{JansenKurt}, a concrete example is given where such a construction leads to a $q-$duality, which is pathwise in the sense that it is constructed from one realization of the graphical representation, but not strongly pathwise, as it is obtained by averageing over the random mechanisms. It can be viewed as a conditional duality by fixing a sequence $(Z_n)_{n\in\N}$ of iid Bernoulli random variables with parameter $q,$ independent of the Poisson processes, and defining $\mathcal D:=\sigma(\{Z_n, n\in\N\}).$ Then $(X_t)$ and $(Y_t)$ are dual conditional on $\mathcal D.$
\end{example}
This construction makes use of the thinning property of Poisson processes. A related approach, also in the context of interacting particle systems, is described in \cite{SudburyLloyd97}, \cite{AthreyaSwart}, using thinnings of the (particle) processes instead of the Poisson processes, leading to similar results.

\subsection{Rescaled processes}\label{sec:pathwise_rescaled}
So far we considered processes with discrete state space, mostly interacting particle systems. A natural question to ask is whether a (pathwise) duality is preserved in some sense after rescaling. Such ideas were exploited for example in \cite{Swart, AlkemperHutzenthaler} and in a more sophisticated way in \cite{DonnellyKurtz1, DonnellyKurtz2}. One simple approach is to approximate the hypergeometric distribution showing up in the context of $q-$duality by the binomial distribution. One obtains (see \cite{JansenKurt} for the proof)
\begin{prop}
\label{thm:prototype_general}
Let $(X^N_t)$, $(Y^N_t)$ be Markov processes with state space $E_N$ that 
are $q_N-$dual for some $q_N\in[-1,1)$.  
Choose exchangeable initial conditions $X_0^N, Y_0^N\in E_N$, fixing $|X_0^N|=k_N, |Y_0^N|=n_N$, and suppose that $X_t^N$ and $Y_t^N$ stay exchangeable for all $t>0$.
Assume that $n_N/N\to 0$ and $\E[|Y^N_{t_N}|/N]\to 0$ as $N\to \infty$, for some time scale $t_N\geq 0$.  
Then
$$\lim_{N\to\infty}\E\left[\left(1+(q_N-1)\frac{|X_0^N|}{N}\right)^{|Y^N_{t_N}|}\right]=
\lim_{N\to\infty}\E\left[\left(1+(q_N-1)\frac{|X_{t_N}^N|}{N}\right)^{|Y^N_{0}|}\right],$$
provided that the limits exist.
\end{prop}

%
Depending on the scaling, this result may lead to a moment duality, if $\frac{|X_t^N|}{N}\to X_t,$ and $|Y_t^N|\to Y_t,$ as we then get $\E[(1+(q-1) X_0)^{Y_t}]=\E[(1-(q-1)X_t)^{Y_0}].$ If $X^N$ and $Y^N$ have the same scaling, we may get a Laplace duality, see \cite{AlkemperHutzenthaler} for an example.

\begin{definition}[$q-$moment duality]\label{def:q-moment}
Let $(p_t)_{t\geq 0}$ and $(n_t)_{t\geq 0}$ be moment dual Markov processes with values in $\R$ and $\N$ respectively. Assume that there exist $q$dual interacting particle systems $(X_t^N), (Y_t^N)$ for $q\in[-1,1)$ such that $\left(1-(q-1)\frac{X^N_t}{N}\right)$ converges weakly to $(p_t),$ and $(Y_t^N)$ converges weakly to $(n_t).$ Then we say that $(p_t)$ and $(n_t)$ are \emph{$q-$moment dual}.\\
If $q=0,$ we say that $(p_t)$ and $(n_t)$ are \emph{coalescing moment duals}, for $q=-1$ we call them \emph{annihilating moment duals}.
\end{definition}
In the next section we will see that moment dualities obtained in this manner retain some of the properties of the approximating particle systems, such as monotonicity.
An example of a moment duality obtained in this way is the duality of the $(1,b,c,d)-$braco-process and the $(1,b,c,d)-$resem-process, see \cite{AlkemperHutzenthaler, AthreyaSwart}.

\begin{remark}[Non-consistency]
Note that using basic mechanisms gives a (strong) pathwise construction of all the approximating processes $(X_t^N), (Y_t^N).$ In passing to the limit, this pathwise construction suffers two problems: The duality with respect to the hypergeometric distribution is not strongly pathwise, and in particular we need to pass from $N$ to $N+1,$ which necessitates a new choice of the realization of the graphical representation. Hence, the construction given here is only pathwise for finite $N,$ and not for the limiting processes. 
\end{remark}

\begin{remark}[Lookdown-construction] 
If this step from $N$ to $N+1$ can be done in a consistent way, keeping the graphical representation of step $N$ and add arrows in step $N+1,$ one would obtain a pathwise construction for the limiting duality. This was achieved by Donnelly and Kurtz \cite{DonnellyKurtz1} via the so-called lookdown-construction, which provides a pathwise construction of the duality of the Fleming-Viot process and Kingman's coalescent, and has since then be successfully applied to many different situations \cite{DonnellyKurtz1, DonnellyKurtz2, Birkner_etal}. It is outside the scope of the present paper to do the lookdown-construction full justice. It roughly works as follows: As before, we have a graphical representation, with lines of length $T$ representing time for each particle, labelled by $i\in\N.$ In the lookdown construction, these are usually drawn horizontally instead of vertically. We also have Poisson processes with rates $\lambda_{ij},$ and we draw arrows from $i$ to $j,$ but now they are only allowed to point in one direction: top to bottom, which for exchangeable models is always possible. We can say that $\lambda_{ij}=0$ if $i\leq j.$ Mechanisms now work in only one direction: If an arrow form $i$ to $j$ is encountered, meaning that $i>j,$ then the site $i$ changes, according to a mechanism and according to the state at $j$ that is seen when following the arrow down -- hence the name "lookdown construction". In order to apply the lookdown construction, it is crucial to choose the rates in a consistent way, that is, that the arrows for the first $N$ lines can be kept if more lines are added. This means that we can study the genealogy of a sample of size $n<N$ within the construction of $N$ line. The result of \cite{DonnellyKurtz1} tells us that it is possible to take the limit as $N\to\infty.$
\begin{center}
\setlength{\unitlength}{1cm}
\begin{picture}(12,5)
\multiput(1,1)(0,1){5}{\line(1,0){10}}

\multiput(2,2)(5,0){2}{\vector(0,-1){1}}
\multiput(1.7,3)(4.2,0){2}{\vector(0,-1){2}}
\put(9,3){\vector(0,-1){1}}
\put(3,4){\vector(0,-1){3}}
\put(10.2,4){\vector(0,-1){3}}
\put(1.1,4){\vector(0,-1){1}}
\put(5,4){\vector(0,-1){2}}
\put(2.2,5){\vector(0,-1){2}}
\put(3.5,5){\vector(0,-1){4}}
\put(4,5){\vector(0,-1){1}}
\put(6,5){\vector(0,-1){3}}
\put(62,5){\vector(0,-1){1}}
\put(8.3,5){\vector(0,-1){2}}
\put(9.6,5){\vector(0,-1){1}}
\put(10.5,5){\vector(0,-1){4}}
\put(0.5,1){1}
\put(0.5,2){2}
\put(0.5,3){3}
\put(0.5,4){4}
\put(0.5,5){5}
\put(1,1){\vector(1,0){10}}
\put(11.2,0.8){$T$}
\put(2.5,0){\rm{Figure 3: \small{Lookdown-construction}}}
\end{picture}
\end{center}
\end{remark}

\section{Monotonicity}\label{sec:monotonicity}
An interesting question is which properties of Markov processes are preserved under duality. One such property is (stochastic) monotonicity. Its intrinsic relation to duality has been investigated for example in \cite{Siegmund76, AsmussenSigman, moehle-cones}. In particular, there is a connection between monotonicity and the existence of a Siegmund dual. We state this connection in Theorem \ref{prop:monotone_dual}, and prove it using the results from Section \ref{sec:fa}. Moreover, we discuss other duality functions, in particular $q-$moment duals in Corollary \ref{cor:monotone_moment}.

\begin{definition}[Monotonicity]
A stochatic process $(X_t)$ on a partially ordered state space $E$ is called (stochastically) \emph{monotone}, if $x\leq y$ implies $\P_x(X_t\geq z)\leq \P_y(X_t\geq z), x,y\in E.$
\end{definition}
 
 \begin{remark}
 Note that $\P_x(X_t\geq z)\leq \P_y(X_t\geq z)$ if and only if $\E_x[f(X_t)]\leq E_y[f(X_t)]$ for every continuous increasing function $f.$ Hence, a Feller process $(X_t)$
with semigroup $P_t$  is monotone if and only if $P_tf$ is continuous and monotone for every continuous and monotone function $f.$
 \end{remark}
 
\begin{remark}
In the language of spin systems, monotone processes are often called \emph{attractive}.
\end{remark}

Monotonicity is in some situations equivalent to having a coalescing dual or Siegmund dual \cite{Siegmund76}, see also \cite{asmussen, CliffordSudbury85}. We give below a proof of this result which illustrates the connection to the invariance of the set $\mathcal V_{1,+}=\{\int_F H(\cdot, y)\nu(\dd y)\;|\; \nu\in\mathfrak M_{1,+}(F)\},$ from Section \ref{sec:fa}. \\
In the case of Siegmund duality, that is, $E=F=\R$ and $H(x,y)= \mathbf 1_{\{x\geq y\}}$ we see that $f\in  \mathcal{V}_{1,+}$ if and only if there exists $\nu\in\mathfrak M_{1,+}(F)$ such that 
$$f(x)=\int \mathbf{1}_{\{x\geq y\}}\nu(\dd y)=\int_{\{y\leq x\}}\nu(\dd y)=\nu(\{y\leq x\}).$$ Hence,
$\mathcal V_{1,+}$ consists of the cumulative distribution functions of probability measures, i.e.,  right-continuous, monotone increasing functions with $\lim_{x\to -\infty} f(x) = 0$ and $\lim_{x\to \infty} f(x) = 1$. The set $\mathcal{V}$ consists of the right-continuous functions with bounded variation.

\begin{theorem}
\label{prop:monotone_dual}
\begin{itemize}
\item[(a)] Let $E$ be partially ordered. Let $(X_t)$ and $(Y_t)$ be dual with respect to $H(x,y)=\mathbf1_{\{x\geq y\}}.$
Then $(X_t)$ and $(Y_t)$ are monotone.
\item[(b)] Let $(X_t)$ be a monotone process with state space $E=\R, E=[0,\infty)$ or $E=\N,$ such that $x\mapsto \P_x(X_t\geq y)$ is right-continuous for every $t\geq 0, y\in E$. Then there exists a process $(Y_t)$ on $E\cup \{\infty\}$ such that $(X_t)$ and $(Y_t)$ are dual with respect to $H(x,y)=\mathbf1_{\{x\geq y\}}.$
\end{itemize}
\end{theorem}

\emph{Proof of Theorem \ref{prop:monotone_dual}.}
$(a).$ Let $x\leq y.$ Then
$\P_x(X_t\geq z)=\P_z(Y_t\leq x)\leq \P_z(Y_t\leq y)=\P_y(X_t\geq z).$
Hence $(X_t)$ is monotone. Exchangeing the roles of $X$ and $Y$ completes the proof.
\\
$(b).$ If $x\mapsto \P_{x}(X_t\geq y)$ is monotone increasing and right-continuous, it is a distribution function of a probability measure $Q_{t,y}(\cdot)$ on $F.$ Let $f\in \mathcal V_{1,+},$ and let $\nu\in\mathfrak M_{1,+}(F)$ such that $f(x)=\int H(x,y)(\nu)(\dd y),$ with $H(x,y)=\mathbf 1_{\{x\geq y\}}$ (cf. the previous example). Let $(P_t)$ denote the semigroup of $(X_t).$ We get 
\begin{equation}
P_tf(x)=\E_x\left[\nu((-\infty, X_t])\right]=\int_{F}\int_E \mathbf{1}_{\{z\leq x\}} Q_{t,y}(\dd z)\nu(\dd y).\end{equation}
Defining a probability measure $\nu_t$ on $F$ by  $\int_E Q_{t,y}(\cdot)\nu(\dd y),$ we see that $\mathcal V_{1,+}$ is invariant under $P_t.$ By Proposition \ref{prop:non-deg}, this implies existence of a Siegmund dual \noemi{(note that Assumption \ref{ass:w-dense} is satisfied).}
\hfill$\Box$

\begin{cor}\label{cor:monotonicity_coal_dual}
Let $(X_t)$ and $(Y_t)$ be coalescing dual spin systems. Then $(X_t)$ and $(Y_t)$ are monotone. 
\end{cor}
\begin{proof}
By Proposition \ref{prop:monotone_dual}, $(X_t)$ and $(1-Y_t)$ are both monotone.  For spin systems, $x\leq y$ if and only if $ x\geq 1-y,$ hence monotonicity of $(Y_t)$ follows from the monotonicity of $(1-Y_t).$
\end{proof}

The converse is not true, as can be seen by analysing once more $\mathcal V_{1,+}.$ In the case of coalescing duals, we obtain with $H(A,B) = 1_{\{A\cap B = \emptyset\}}$ that $f\in \mathcal{V}_{1,+}$ if and only if $f(X)=\sum_{B\subset A^c}\nu(B),$ or equivalently $f (A) =\sum_C \lambda(C) \mathbf{1}_{\{A\subset C\}}$ for some $\lambda\in \mathfrak M_{1,+}(\mathcal{P}(\Lambda)).$ In particular, $\mathcal{V}_{1,+}$ is a subset of the functions that are monotone decreasing with respect to the partial order given by $A \leq B$ iff $A\subset B$. If $(X_t)$ has a dual with respect to $H$, then the semi-group maps every 
indicator function $\mathbf{1}_{\{\cdot \subset C\}}$ into $\mathcal{V}_{1,+}$ and 
therefore we see again that $(X_t)$ is decreasing. But the converse is wrong, because the semi-group of a monotone Markov process might map a function $\mathbf{1}_{\{\cdot\subset B\}}$ to a monotone function not in $\mathcal{V}_{1,+}$. \\

For duality functions other than the Siegmund (or coalescing) duality, monotonicity need not be preserved under duality:

\begin{example}(Monotone process with non-monotone dual)
Consider the voter model on $\Z^d.$ This is clearly a monotone process, which is either seen directly, or by Proposition \ref{prop:monotone_dual}, since it has coalescing random walk as a coalescing dual. It also has an annihilating dual, which is annihilating random walk. This is seen from Lemma \ref{lem:ann_dual} in the appendix. This dual is not monotone: Let $x,y\in\{0,1\}^{\Z^d},$ with $x=\delta_i+\delta_j, i\neq j,$ and $y=\delta_i.$ Then we have $x\geq y,$ but annihilating random walk started at $x$ will almost surely reach state $\bf 0,$ while the process started in $y$ will always have at least one 1. 
\end{example}

If we have a graphical representation, monotonicity of the process follows from monotonicity of the mechanism. This will help us to prove monotonicity of coalescing moment duals, cf. section \ref{sec:pathwise_rescaled}

\begin{definition}[Monotone mechanism]
Let $E$ be partially ordered. A function $f:E\to E$ is monotone if $x\leq y\Rightarrow f(x)\leq f(y),x,y\in E.$ Similarly, a basic mechanism $f:\{0,1\}^2\to\{0,1\}^2$ is called monotone if $x\leq y\Rightarrow f(x)\leq f(y), x,y\in \{0,1\}^2.$
\end{definition}

\begin{prop}
Let $(X_t)$ be a process defined from a graphical representation as described in section \ref{sec:pathwise_strong} using Poisson-processes and basic mechanisms. $(X_t)$ is monotone if and only if all basic mechanisms are monotone.
\end{prop}

\begin{proof}
If all basic mechanisms are monotone, then it is clear that the resulting process is monotone. For the converse direction, consider a monotone process $X$. Let $f$ be a non-monotone basic mechansim used in the construction of $X.$ Let $\tau>0$ be the time of the first arrow in the graphical representation.  Assume that $f$ is the first type of transition to happen in the graphical representation, this happens with positive probability. We couple two versions of $X,\tilde{X}$ where we can choose initial conditions $X_0=x,\tilde{X}_0=y\in E, x\leq y,$ such that $\tilde{X}_\tau\leq X_\tau.$ This is a contradiction to the monotonicity of $X.$
\end{proof}

\noemi{
Consequently, processes that are derived as in \ref{prop:funct-duality} via a monotone function are also monotone.}
Conservation of monotonicity is preserved under rescaling of the duality, hence moment dualities obtained by rescaling a coalescing duality (section \ref{sec:pathwise_rescaled}) also preserves monotonicity.

\begin{cor}\label{cor:monotone_moment}
Let $(p_t),(n_t)$ be coalescing moment duals. Then they are both monotone.
\end{cor}

\begin{proof}
By assumption, there are coalescing dual processes $(X^N_t), (Y^N_t)$ on $\{0,1\}^N$ such that $(p_t)$ is the weak limit of $\left(\frac{N-|X^N_t|}{N}\right)$ and $n_t$ is the weak limit of $(|Y^N_t|).$ $(X^N_t)$ and $(Y^N_t)$ are monotone by corollary \ref{cor:monotonicity_coal_dual}. Therefore, $(|Y^N_t|)$ and $(N-|X_t^N|)=(|1-X_t^N|)$ are also monotone. We get, for $x\leq y, x,y\in \R,$
$$\P_x(p_t\geq z)=\lim_{N\to\infty}\P_{\lfloor Nx\rfloor}(N-|X_t^N|\geq\lfloor NZ\rfloor)\leq \lim_{N\to\infty}\P_{\lfloor Ny\rfloor}(N-|X_t^N|\geq\lfloor NZ\rfloor)=\P_y(p_t\geq z).$$
Hence $(p_t)$ is monotone, and similarly for $(n_t).$
\end{proof}

Other moment dualities need not imply  monotonicity. An example of a moment duality of non-monotone processes is given by rescaling the branching annihilating process from Appendix \ref{app:mechanisms} and its dual, leading to a Wright-Fisher diffusion with a certain type of balancing selection, cf. \cite{JansenKurt}.

\section{Symmetries and intertwining} \label{sec:hilbert} 

This section complements the functional analytic theory from Section \ref{sec:fa} and explains the relations of duality with the notions of intertwining of Markov processes \cite{DiaconisFill,carmona-petit-yor98,huillet-martinez11, swart13}, symmetries, and methods borrowed from quantum mechanics. Although most of the material is known, there are some new aspects: the main point of Section~\ref{sec:intertwining} is that some of the quantum symmetries of the physics literature can be given a stochastic interpretation with the help of an argument from \cite{huillet-martinez11}. Section~\ref{sec:quantum} carves out the rigged Hilbert space / Gelfand triple structure (see Eq.~\eqref{eq:gelfand}) of the representations of creation and annihilation operators used for example in \cite{doi76,sandow-schuetz94,RedigEtAl09}, noting that birth and death are time reversals with respect to natural reference measures.

\subsection{Intertwining of Markov processes} \label{sec:intertwining}

Let $P$ and $Q$ be stochastic matrices in finite state spaces that are dual with non-degenerate duality function $H$, so that $PH = HQ^\trans$ with invertible matrix $H$. Then there is a bijection between the set of duality functions for $P$ and $Q$ and the matrices commuting with $P$  (\emph{symmetries}): if $SP = PS$, then $SH$ is a duality function for $P$; conversely, if $\tilde H$ is a duality function, then $S:= \tilde H H^{-1}$ commutes with $P$. An interesting special case is obtained if $H = \mathrm{diag} (1/\pi(x))$ for some probability measure $\pi$ with $\pi(x)>0$: in this case every other duality function is of the form $\tilde H(x,y)= S(x,y)/\pi(y)$ for some matrix $S$ commuting with $P$. Relations of this type have been studied in \cite{Moehle99,RedigEtAl09}. In both articles, the symmetries enter in an algebraic way and need not have a stochastic interpretation.

In contrast, the notion of \emph{intertwining} starts from a stochastic matrix $\Lambda$, sometimes referred to as a \emph{link} between two Markov processes $(P_t)$ and $(Q_t)$. 
The relation of intertwining and duality has been studied by Carmona, Petit, and Yor \cite[Section 5.1]{carmona-petit-yor98} and by Diaconis and Fill \cite[Section 5]{DiaconisFill}: if a stochastic matrix $\Lambda$ intertwines $(P_t)$ and $(Q_t)$, i.e, $P_t \Lambda = \Lambda Q_t$ for all $t$, and $(\tilde Q_t)$ is the time reversal of $Q_t$ with respect to a measure $\pi$ such that $\pi(x)>0$ for all $x$, then $H(x,y) := \Lambda(x,y)/\pi(y)$  is a duality function for $(P_t)$ and $(\tilde Q_t)$. At first sight, it looks like the restriction that $\Lambda$ is a stochastic matrix excludes many dualities.  Huillet and Martinez \cite{huillet-martinez11}, however, have given a general argument {proving} that for irreducible Markov chains, in finite state spaces, \emph{all} duality functions are associated with stochastic matrices. 

Quickly summarized, the argument is as follows. Let $P$ and $Q$ be irreducible stochastic matrices. Let $\pi$ be the stationary distribution of $Q$. By adding to $H$ multiples of the matrix with all entries equal to $1$, we can ensure that $H$ has non-negative entries only, and that the function $h:=H\pi$ is strictly positive. Furthermore, because $H$ is a duality function and $\pi$ is invariant, $h$ is harmonic for $P$, i.e., $Ph =h$. If $P$ is irreducible, $h$ must be constant and, up to a constant factor, the matrix $\Lambda(x,y):= H(x,y) \pi(y)$ is stochastic and intertwines $P$ and $\tilde Q$, the time reversal of $Q$ with respect to $\pi$. If $P$ is not irreducible, $\Lambda(x,y):= h(x)^{-1} H(x,y)\pi(y)$ intertwines the Doob $h$-transform $P_h$ with the time reversal $\tilde Q$, see also \cite[Section 5]{DiaconisFill}. 
 
The previous argument can be applied to many dualities of interacting particles with invariant measures, showing that some of the ``algebraic'' or ``quantum'' symmetries of \cite{sandow-schuetz94,RedigEtAl09} can be interpreted in terms of commuting Markov chains. 
We refrain from an abstract description and content ourselves with the following example. 

\begin{example}[Symmetric simple exclusion process] 
	Let $M \in \N$, $E= \{0,1\}^M\equiv \mathcal{P}(\{1,\ldots,M\})$, and $(P_t)$ the symmetric simple exclusion process (a particle at site $k\in \{1,\ldots,M\}$ jumps to a neighboring site $k\pm 1$ at rate $1$, provided the site is unoccupied and in $\{1,\ldots,M\}$). The process is self-dual with duality function $H(A,B) = \mathbf{1}_{\{A \subset B\}}$. It has many reversible measures, for example, $\pi = \otimes_1^M \mathrm{Ber}(1/2)$, a product of Bernoulli measures with parameter $1/2$. The harmonic function $h:= H\pi$ is given by 
\begin{equation*}
		h(A) = \frac{1}{2^{M}} \bigl| \{ B \subset \{1,\ldots,M\}\mid A \subset B\} \bigr|
				= \frac{1}{2^{|A|}}. 
\end{equation*}
and the stochastic matrix $\Lambda$ associated with $H$ and $\pi$ is given by 
\begin{equation*}
	\Lambda(\vect{\eta},\vect{\eta'}) = \prod_{k=1}^M q(\eta_k,\eta'_k),\quad 
	\begin{pmatrix} q(0,0) & q(0,1) \\ q(1,0) & q(1,1) \end{pmatrix} 
	= \begin{pmatrix} 
				1/2 & 1/2 \\
				0 & 1
		\end{pmatrix}. 
\end{equation*}
$\Lambda$ describes a chain where at each unoccupied site, a particle is born with probability $1/2$.  
Since $(P_t)$ has no transitions between configurations with differing particle numbers, 
the Doob $h$-transform $P_t ^h(A,B) = h(A)^{-1} P_t(A,B) h(B)$ is equal to $P_t$. Thus we obtain, in the end, that $P_t \Lambda = \Lambda P_t$: the symmetric simple exclusion process commutes with the birth mechanism described by $\Lambda$. This provides a stochastic interpretation of the ``quantum'' symmetry $\exp( \sum_k S_k^+)$ (see p.~\pageref{eq:spin-chain} below for an explanation of the notation) studied in \cite{sandow-schuetz94,RedigEtAl09}.
\end{example}

\subsection{Quantum many-body representations of interacting particle systems} \label{sec:quantum}

There is a close relationship between problems from quantum mechanics and probability, which helps the understanding of duality \cite{sandow-schuetz94, SudburyLloyd95, RedigEtAl09}.
Let $\Hi$ be a separable Hilbert space, for example $L^2(E,\mu)$, a space of complex-valued square-integrable functions.  We write $\la \cdot, \cdot\ra$ for the scalar product, linear in the second entry and conjugate linear in the first entry.
The quantum mechanics analogue of contraction semi-groups $(P_t)_{t\geq 0}$ are one-parameter unitary groups $(U_t)_{t \in \R}$ in $\Hi$. Stone's theorem \cite{reed-simon} says that every such $(U_t)$ can be written as $U_t = \exp( -\mathrm{i} t H)$, for a unique self-adjoint operator $H$. The \emph{Hamilton operator} $H$ and Stone's theorem take the place of the infinitesimal generator and the Hille-Yosida theorem in probability. 
Some quantum mechanical Hamiltonians admit stochastic representations, by which we mean the following: 

\begin{definition} \label{def:stochrep}
	Let $H$ be a self-adjoint operator in a Hilbert space $\Hi$ and $(P_t)$ a Markov semi-group with symmetrizing $\sigma$-finite measure $\mu$ in a Polish space $E$; we assume that $\mu(\mathcal{O})>0$ for every non-empty open set $\mathcal{O}\subset E$. We say that $(U,E,\mu,(P_t))$ is a stochastic representation of $H$ if $U: \Hi\to L^2(E,\mu)$ is a unitary operator and 
	\begin{equation*} 
		\forall t\geq 0:\quad \exp(-t H) = U^* P_t U. 
	\end{equation*} 
\end{definition}

%
\begin{example}[Ornstein-Uhlenbeck process and harmonic oscillator]
	Let $\Hi = L^2(\R,\dd x)$ ($\dd x$ is Lebesgue measure) and 
	\begin{equation*} 
		H= - \frac{1}{2} \frac{\dd^2}{\dd x^2} + \frac{1}{2} x^2 - \frac{1}{2},
	\end{equation*} 
	the so-called \emph{harmonic oscillator}. $H$ with a suitable definition is self-adjoint.  As is well-known, it has a stochastic representation given by the \emph{Ornstein-Uhlenbeck process}:  Let $\varphi(x):= \pi^{-1/4} \exp(- x^2/2)$ and define $U: L^2(\R,\dd x) \to L^2(\R, \varphi^2 \dd x)$ by $(Uf)(x) := f(x)/\varphi(x)$. Then 
	\begin{equation*} 
		UHU^* = - \frac{1}{2} \frac{\dd^2}{\dd x^2} + x \frac{\dd}{\dd x} = - L,
	\end{equation*}
	with $L$ the infinitesimal generator of the Ornstein-Uhlenbeck process.
\end{example}

\begin{example}[Simple symmetric exclusion process and quantum spin chain \cite{gwa-spohn92,sandow-schuetz94,lloyd-sudbury-donnelly96,RedigEtAl09}] 
	Fix $M\in\N$ and let $\Hi = \C^2 \otimes \cdots \otimes \C^2$ ($M$ times) with the standard scalar product. Consider the Hermitian $2\times 2$-matrices $S^x,S^y,S^z$ determined by 
	\begin{equation} \label{eq:spin1/2}
	S^+ = S^x + \mathrm{i} S^y := \begin{pmatrix} 0& 1 \\ 0 & 0 \end{pmatrix},\ 
			S^- = S^x - \mathrm{i} S^y := \begin{pmatrix} 0& 0 \\ 1 & 0 \end{pmatrix},\ 
		S^z := \begin{pmatrix} 1/2 & 0 \\ 0 & -1/2 \end{pmatrix}.
\end{equation} 
	$S^\pm$ correspond to birth and death of a particle, and $S^+ + S^-$ corresponds to spin flip. 
	Let $S_k^\alpha$ be the operator $S^\alpha$ acting on the $k$-coordinate, e.g., 
	$S_1^x (e_1\otimes \cdots \otimes e_M)= (S^x e_1) \otimes e_2 \otimes \cdots \otimes e_M$. Let 
	\begin{equation}  \label{eq:spin-chain}
		H = -  \sum_{k=1}^{M-1} \Bigl( S_k^x S_{k+1}^x + S_k^y S_{k+1}^y+ S_k^z S_{k+1}^z   - \frac{1}{4}\Bigr) 
		   =  - \sum_{k=1}^{M-1} \Bigl( S_k^+ S_{k+1}^- + S_k^- S_{k+1}^+ + S_k^z S_{k+1}^z   - \frac{1}{4}\Bigr), 
	\end{equation} 
	the Hamiltonian for an isotropic spin $1/2$-quantum spin chain. Then for every $t>0$, the matrix $\exp(- t H)$ in the canonical basis of $\Hi$ is doubly stochastic; $-H$ is the infinitesimal generator of the simple symmetric exclusion process. The relationship is precisely of the form given in Definition~\ref{def:stochrep}, with underlying state space $E= {\{0,1\}}^M$ and $\mu$ the counting measure (or, if we normalize and change the scalar product in $\Hi$ by a multiplicative constant, the uniform distribution on $E$). 
\end{example}

\begin{remark} The previous example is closely related to the topic of stochastic representations of quantum spin chains, going back to T{\'o}th \cite{toth93} and Aizenman and Nachtergaele \cite{aizenman-nachtergaele94}, see also \cite{goldtschmidt-ueltschi-windridge11} and the references therein. In this context the emphasis is usually put on stochastic geometric aspects that arise in a way similar to percolation problems in graphical representations of interacting particle systems \cite{Harris, Griffeath, Liggett05}.   
\end{remark}


Theorem~\ref{thm:unitary} just says that two dual reversible Markov processes give rise to unitarily equivalent Hamiltonians, $H_2 = UH_1 U^*$; put differently, if a Hamiltonian has two different stochastic representations, then the representing Markov processes are dual with non-degenerate duality function, provided some regularity conditions.\footnote{In particular,  if $(U,E,\mu,(P_t))$ and $(V,F,\nu,(Q_t))$ are the two reperesentations, we need to assume that $VU^*: L^2(E,\mu) \to L^2(F,\nu)$ is of the form $(VU^*f)(x) =\int D(x,y) f(y) \mu(\dd y)$ for some function $D$, which will become the duality function.} Self-dualities correspond to unitaries $U$ such that $UHU^* = H$ -- i.e., \emph{symmetries} of the Hamiltonian. The quantum mechanist's ansatz for finding dualities is, therefore, to look for symmetries and alternative representations of a Hamiltonian.
To this aim it is convenient to reinterpret the infinitesimal generators of stochastic processes as Hamiltonians and, for interacting particle systems, to rewrite the operator 
using notation of quantum-many body theory, notably creation and annihilation operators. In probabilistic terms, this means that we break down transitions into birth and death of particles; for example, in Eq.~\eqref{eq:spin-chain}, the hopping of a particle from site $k$ to site $k+1$ is written as death at site $k$ followed by birth at site $k+1$. 

We refer the reader to \cite{RedigEtAl09} for a systematic treatment of the method for finding dualities; the remainder of this section clarifies the representation theoretic structure underpinning the method. 

First, let us stress that the quantum mechanics translation can be useful in a broader context than in Definition~\ref{def:stochrep}, allowing for non-reversible Markov processes and ``Hamiltonians'' that are not necessarily self-adjoint -- for example, a non-symmetric simple exclusion process leads to a non-self adjoint Hamiltonian \cite{gwa-spohn92}. In fact, we might wish to study processes before even knowing whether they have reversible or invariant measures. For this reason it is useful to work with an a priori reference measure (e.g., Poisson or Bernoulli) and to make explicit that we work in three different spaces: bounded measurable functions, signed measures, and a Hilbert space. 

Thus let $E$ be a Polish space and $\mu_0$ a reference probability measure on $E$. Let $\Hi:= L^2(E,\mu_0)$. 
We have the following embeddings: 
\begin{equation}  \label{eq:gelfand}
	L^\infty(E) \stackrel{\iota}{\rightarrow} \Hi \stackrel{\iota^*}{\rightarrow} \meas(E) 
\end{equation}
given by $\iota f:= f$  and $\iota^* f:=f \mu_0$. The notation $\iota^*$ is justified by the identity $\la \iota f,\varphi\ra_{\Hi} = \la f, \iota^*\varphi\ra_{L^\infty\times \meas}$, where $\la f,\mu\ra_{L^\infty \times \meas}: = \int_E f\dd \mu$.  A bounded operator $A$ in $\Hi$ acts on measures that have an $L^2$-density with respect to $\mu_0$ and on functions via 
\begin{equation*} 
	\ell(A) (f\mu_0) := (A f) \mu_0,\quad r(A) f:= A^* f. 
\end{equation*} 
The following identity holds: 
\begin{equation*} 
	\la r(A) f, \mu \ra_{L^\infty \times \meas} = \la f,A \frac{\dd \mu}{\dd \mu_0} \ra_{\Hi} = \la f, \ell(A) \mu \ra_{L^\infty \times \meas}.
\end{equation*} 
When  $E$ is finite and $\mu_0$ is the counting measure, we may think of functions as row vectors, measures as column vectors (note the unfortunate inversion with respect to the usual probabilist's conventions), and operators as matrices; then $\ell(A) \mu = A\mu$ and $r(A) f = f A$, i.e., $\ell(A)$ and $r(A)$ correspond to multiplication from the left and multiplication from the right, respectively. 

If $(P_t)$ is a Markov process with infinitesimal generator $L$, we may associate an operator $H$ in $\Hi$ by asking that $P_t^* (f\mu_0) = (\exp(-t H) f) \mu_0$; thus formally, $\ell(H) = - L^*$. In practical applications, the distinction between an operator in Hilbert space and its representation in the spaces of measures or functions is often suppressed, and   we find the relation $H= - L^*$, compare also \cite[Eq. (74)]{RedigEtAl09}.

For interacting particle systems, there is a standard dictionary towards quantum many-body processes. For example, exclusion processes (at most one particle per lattice site) can be expressed in terms of spin $1/2$-operators as in Eq.~\eqref{eq:spin-chain}. The other two cases (partial exclusion and unbounded number of particles per site) are explained in the next two examples. The point we wish to stress is that there are natural reference measures, turning birth and death into time reversals or, on a quantum mechanical level, ensuring that creation and annihilation are Hilbert space adjoint operators in some Hilbert space. This fact, to the best of our knowledge, has not been highlighted before. 

\begin{example}[Partial exclusion and spin $m/2$-chains] 
	 Fix $m\in \N$ and let $E:=\{0,1,\ldots,m\}$; think of $E$ as the single-site state space for an interacting particle system with at most $m$ particles per lattice site. Let
	$\mu_0(\eta) := 1/(2^m \binom{m}{\eta})$ be the binomial measure with parameter $1/2$,  and $\Hi = L^2(E,\mu_0)$. We specify $J^\pm$  and $J^z$ via their action on measures,
	\begin{equation*} 
		\ell(J^+) \delta_\eta := (m-\eta) \delta_{\eta+1},\quad 
		\ell(J^-) \delta_\eta:= \eta \delta_{\eta-1},\quad 
		\ell(J^z) \delta_\eta:= (\eta - \frac{m}{2}) \delta_\eta,
	\end{equation*}
$\eta=0,1,\ldots,m$, see \cite[Eq. (21)]{RedigEtAl09}. For $m=1$ (\emph{simple exclusion}), we recover the $2\times 2$ matrices from Eq.~\eqref{eq:spin-chain}.
The simple stochastic interpretation is the following (see the ladder graph interpretation in \cite[Section 4]{RedigEtAl09}): think of $\eta\in E$ as a configuration of $\eta$ filled boxes out of $m$ available boxes. Then $J^+$ corresponds to 
	birth of a particle in one of the $m-\eta$ empty boxes, and $J^-$ corresponds to death of one of the $\eta$ particles. The reader may check that $(J^+)^* = J_-$, i.e., the birth and death mechanisms are time reversals of each other with respect to the binomial reference measure; this is where the choice of the reference measure enters. 
\end{example} 

\begin{example}[Doi-Peliti formalism for systems with unbounded occupation numbers] 
	 Let $\Lambda$ be a discrete set and $E= \N_0^\Lambda$. For $k \in \Lambda$, let $\vect{e_k}$ be the vector in $E$ with all entries vanishing except for the $k$'th which i equal to $1$. Define the action of creation and annihilation operators on measures as 
	 \begin{equation*} 
	 	\ell(c_k^\dagger) \delta_{\vect{n}} = \delta_{\vect{n} + \vect{e_k}},\quad \ell(c_k) \delta_{\vect{n}} = n_k \delta_{\vect{n} - \vect{e_k}}.
	 \end{equation*} 
	 The action on functions is 
	 \begin{equation*} 
	 	\bigl(r(c_k^\dagger) f \bigr) ( \vect{n}) = f(\vect{n} + \vect{e_k}),\quad 
	 	\bigl(r(c_k) f \bigr) ( \vect{n}) = n_k f(\vect{n} - \vect{e_k}). 
	 \end{equation*} 
	 $c_k^\dagger$ is creation (birth) of a particle, and $c_k$ is annihilation (death) of one of the $n_k$ particles at site $k$. 
	 This representation is often used by physicists in order to reformulate stochastic problems as a model of bosons, and is part of the so-called \emph{Doi-Peliti formalism} \cite{doi76,peliti85}. The underlying Hilbert space is, to the best of our knowledge, in general not specified. A natural choice is $\Hi = L^2(E,\mu_0)$ with $\mu_0 = \otimes_{k\in \Lambda} \mathrm{Poiss}(1)$ the tensor product of Poisson measures with parameter $1$. This reference measure has the crucial property that the birth and death mechanisms are time reversals of each other with respect to $\mu_0$, so that the underlying 
operators $c_k$ and $c_k^\dagger$ are adjoints in $\Hi$, $c_k^\dagger = c_k^*$.	 
\end{example}

\begin{remark}
	An  alternative to the bosonic creation and annihilation operators for unbounded occupation numbers is the use of operators connected to representations of $\mathrm{SU}(1,1)$ \cite{RedigEtAl09}. 
\end{remark}

We conclude this section with a remark on group representations; the remark is not of direct relevance for the probabilistic setting, but might be of interest from an algebraic point of view. Groups enter in two ways: first, in relation to self-dualities, the set of unitaries $U$ commuting with a given Hamiltonian $H$ form a group, the \emph{symmetry group}. Second, many of the basis operators considered above are connected to the Lie algebras of Lie groups, and choices of alternative representations of the Hamiltonian are closely related to alternative representations of the Lie group; see \cite{kirillov08} for relevant background on Lie groups. In quantum mechanics, one is especially  interested in unitary representations. Therefore we note that our functional analytic setup for interacting particles typically yields \emph{three} representations, a (unitary, left) representation in the Hilbert space, a (left) representation on measures, and a right representation on functions. ``Right'' refers to the product rule $r(AB) = r(B) r(A)$. The unitarity of the representation depends on the choice of the reference measure $\mu_0$, mirroring the time reversal relations mentioned earlier.

\appendix

\section{Non-degeneracy of some standard duality functions} \label{app:standard}

Here we prove Lemma~\ref{lem:non-degenerate}. Remember that $B_H$ is non-degenerate if the left and right null spaces are trivial; equivalently, if the integral transforms $\nu \mapsto \int_F H(\cdot,y)\nu(\dd y)$ and $\mu \mapsto \int_E H(x,\cdot) \mu(\dd x)$ are injective. 

\begin{proof}[Siegmund duality.]
	Let $H:\R\times \R\to \R$ be given by 
	$H(x,y):= 1$ if $x\leq y$ and $H(x,y):=0$ otherwise. Suppose that a signed Borel measure on $\R$ belongs to the left null space of the bilinear form $B_H$. Then for all $y\in \R$, 
	\begin{equation*}
		0 = B_H(\mu,\delta_y) = \int \mu(\dd x) H(x,y) = \mu\bigl( (-\infty,y] \bigr), 
	\end{equation*}
It follows that $\mu \equiv 0$. Thus the left null space of $B_H$ is equal to $\{0\}$. The right null space can be treated in a similar way.  
\end{proof}

\begin{proof}[Moment duality]
	Let $H:[0,1]\times \N_0 \to \R$, $H(x,n):= x^n$. Let $\mu\in \meas(\R)$ be in the left null space of $B_H$. Then $B_H(\mu,\delta_n) = \int_0^1 x^n \mu(\dd x) =0$ for all $n\in \N_0$. Write $\mu = \lambda_+ \mu_+ - \lambda_- \mu_-$ for the Hahn-Jordan decomposition, with $\mu_+$ and $\mu_-$ probability measures on $\R$. We have $\mu([0,1]) = \int_0^1 x^0 \mu(\dd x) =0$, hence $\lambda_+ = \lambda_- = :\lambda$. If $\lambda =0$, then obviously $\mu =0$. If $\lambda \neq 0$, we have two probability measures $\mu_{\pm}$ with all their moments identical. 
Since a probability measure on a bounded interval is uniquely defined by its moments~\cite[Section VII.3]{feller2}, it follows that $\mu =0$. 

	Now let $(a_n)_{n\in \N_0} \in \meas(\N_0) = \ell^1(\N_0)$ be in the right null space. Then for every $x\in [0,1]$, 
	\begin{equation*}
		0 = B_H\bigl(\delta_x, (a_n)\bigr) = \sum_{n=0}^\infty a_n x^n, 
	\end{equation*} 
	which implies $a_n=0,\forall  n\in\N_0.$ 
\end{proof} 

\begin{proof}[Laplace duality] 
	Let $H:[0,\infty) \times [0,\infty)\to \R$ be given by $H(x,\lambda):= \exp(-\lambda x)$.  
	Let $\mu \in \meas(\R)$ be in the left null space of $B_H$ and $g(\lambda):= \int_0^\infty \exp(-\lambda x) \mu(\dd x)$ its Laplace transform. Then for all $\nu \in \meas(\R)$, 
	\begin{equation*}
		B_H(\mu,\nu) = \int_0^\infty g(\lambda) \nu(\dd \lambda) = 0,
	\end{equation*}
	and therefore the Laplace transform $g$ of $\mu$ is $0$. It follows~\cite[Section XIII]{feller2} that $\mu =0$. 
	The right null space can be treated in a similar way. 
\end{proof} 

\begin{proof}[Coalescing dual] 
	Let $\Lambda$ be a finite set, $E= F= \mathcal{P}(\Lambda)$ (the collection of subsets of $\Lambda$)  and 	
	\begin{equation*}
		H(A,B)=
		\begin{cases}
	 		1, &\quad A\cap B = \emptyset,\\
			0, &\quad A\cap B \neq  \emptyset. 
		\end{cases}
	\end{equation*} 
	As observed in Sudbury-Lloyd \cite{SudburyLloyd95}, $H(A,B)$ is best understood by exploiting a tensor product structure. To this aim we identify $\mathcal{P}(\Lambda)$ with $\{0,1\}^\Lambda$ by associating with each set $A\subset \Lambda$ its sequence of occupation numbers $(n_A(x))_{x\in\Lambda}\in \{0,1\}^\Lambda$. Next, we identify $\R^{\mathcal{P}(\Lambda)}$ with $\otimes_{x\in \Lambda} \R^2$ via 
$$\delta_A \equiv\bigotimes_{x\in\Lambda} e_{n_A(x)},\quad e_0:= \begin{pmatrix} 1 \\ 0 \end{pmatrix}, \ 
e_1:= \begin{pmatrix} 0 \\ 1 \end{pmatrix}. $$ 	
Set \begin{equation*}
		h = \begin{pmatrix} 
				1 & 1 \\
				1 & 0 
			\end{pmatrix}  = (h_{ij})_{i,j=0,1}. 
	\end{equation*}
	Then 
	\begin{equation*}
		H(A,B) = \prod_{x\in \Lambda} \chi(A \cap B \cap \{x\} = \emptyset ) = 
			\prod_{x\in \Lambda} h_{n_A(x),n_B(x)}
	\end{equation*}
	whence $H = \otimes_{x\in \Lambda} h$. 
	Since $h$ is invertible, $H$ is invertible too, with inverse $\otimes_{x\in \Lambda} (h^{-1})$, and 
	$B_H$ is non-degenerate. 
\end{proof} 

\begin{proof}[Annihilating dual] 
	Let $\Lambda$ be a finite set, $E= F= \mathcal{P}(\Lambda)$ (the collection of subsets of $\Lambda$)  and 	
	\begin{equation*}
		H(A,B)= (-1)^{|A\cap B|}.
	\end{equation*}
	Just as for the coalescing dual, following~\cite{SudburyLloyd95}, we identify $\R^{\mathcal{P}(\Lambda)}$ with a tensor product space and $H \equiv \otimes_{x\in \Lambda} h$, where this time 
	\begin{equation*}
		h = \begin{pmatrix} 
			1 & 1 \\ 
			1 & - 1
			\end{pmatrix}. 
	\end{equation*}
	Since $h$ is invertible, $H$ is invertible too and $B_H$ is non-degenerate.  
\end{proof} 

\section{Some dual mechanisms}\label{app:mechanisms}
We give a list of some dual basic mechanisms for interacting particle systems, and an example for the application of Proposition \ref{thm:prototype_general}. \noemi{For more details see \cite{JansenKurt}.}
\begin{center}
\begin{tabular}{|c||c|c|c|c||l|}
\hline 
 & $f(0,0)$ & $f(0,1) $& $f(1,0)$ & $f(1,1)$ &   \\ 
\hline 
\hline
$f^R$ &(0,0) & (0,0) & (1,1) & (1,1) &resampling \\ 
\hline 
$f^C$ &(0,0) & (0,1) & (0,1) & (0,1) & walk-coalescence\\
\hline 
$f^A$ &(0,0) &(0,1) &(0,1)&(0,0)& walk-annihilation\\
\hline
$f^D$&(0,0) & (0,0) & (0,1) & (0,1) & death\\ 
\hline 
$f^{BC}$&(0,0)&(0,1)&(1,1)&(1,1)& branching-coalescence\\
\hline
$f^{BA} $& (0,0) & (0,1) & (1,1) & (1,0) & branching-annihilation \\ 
\hline 
\end{tabular}
\end{center}


\begin{lemma}\label{lem:ann_dual}\begin{itemize}
\item[(a)] Two basic mechanisms $f, g$ are annihilating dual mechanisms if and only if
$$|x\wedge (g(y^\dagger))^\dagger|\mbox{ is odd } \Leftrightarrow |f(x)\wedge y|\mbox{ is odd.}$$
\item[(b)] With the notation of the above table, we have the following:
\begin{itemize}
\item[(i)] $f^R$ and $f^A$ are annihilating duals
\item[(ii)] $f^D$ is an annihilating self-dual
\item[(iii)] $f^{BA}$ is an annihilating self-dual.
\end{itemize}
\end{itemize}
\end{lemma}

For a proof see \cite{JansenKurt}

Let us now consider an example, taken from \cite{JansenKurt}. We will assume that the following mechanisms occur: 
$f^R$ occurs with rate $\frac{r_N}{N}$ for each ordered pair $(i,j), i,j\in \{1,...,N\},$
$f^{BA}$ with rate $\frac{b_N}{N},$
and assume $\frac{r_N}{N}\to\alpha\geq 0, b_N=b_N^a\to \beta\geq 0,$ as $N\to\infty.$  
Hence, the process $|X_t^N|, t\geq 0,$ makes the following transitions:
\begin{eqnarray}
k\to k+1 &\mbox{ at rate } &\frac{r_N+b_N}{N}k(N-k),\label{eq:ktok+1}\\
k\to k-1 &\mbox{ at rate } &\frac{r_N}{N}k(N-k)+\frac{b_N}{N}k(k-1),\label{eq:ktok-1}
\end{eqnarray}
This implies that the rescaled discrete process $\frac{|X_t^N|}{N}$ has discrete generator 
\begin{equation*}
\begin{split}
\tilde{G}_Nf\left(\frac{k}{N}\right)=&\frac{r_N}{N}k(N-k)\left(f\left(\frac{k+1}{N}\right)+f\left(\frac{k-1}{N}\right)-2f\left(\frac{k}{N}\right)\right)\\
&+\frac{b_N}{N}k(k-1)\left(f\left(\frac{k-1}{N}\right)-f\left(\frac{k}{N}\right)\right)+
\frac{b_N}{N}k(N-k)\left(f\left(\frac{k+1}{N}\right)-f\left(\frac{k}{N}\right)\right).
\end{split}
\end{equation*}
Assume now $\frac{k}{N}\to x$ as $N\to \infty.$ Then 
we see that $\tilde{G}_Nf(k/N)$converges to
$$\tilde{G}f(x)=\beta x(1-2x)f'(x)+\alpha x(1-x)f''(x),$$
which is the generator of the one-dimensional diffusion given by the SDE
$$dX_t=\beta X_t(1-2X_t)dt+\sqrt{2\alpha X_t(1-X_t)}dB_t.$$
This is a Wright-Fisher diffusion with local drift $\beta x(1-2x),$ which has the effect of pushing $X_t$ towards the values 0 and $1/2.$
Consider now the dual process. According to Lemma \ref{lem:ann_dual}, $(Y_t^N)$ where $f^A$ happens at rate $\frac{r_N}{N}, f^{BA}$ at $\frac{b_N}{N}$  is an annihilating dual of $(X_t^N).$ Its generator is
\begin{equation*}
\begin{split}G_Nf(k)=&\frac{b_N}{N}\cdot k(N-k)\left(f(k+1)-f(k)\right)+\frac{b_N}{N}\cdot k(k-1)\left(f(k-1)-f(k)\right)\\
&+\frac{a_N}{N}k(k-1)\left(f(k-2)-f(k)\right).
\end{split} \end{equation*}
As $N\to\infty,$ this converges to 
$$Gf(k):=\beta k\left(f(k+1)-f(k)\right)+\alpha k(k-1)\left(f(k-2)-f(k)\right),$$
which is the generator of a branching annihilating process on $\N.$ 

%

\paragraph{Acknowledgments} We thank Jochen Blath and Michael Scheutzow for interesting discussions, and anonymous referees for careful reading and useful remarks.  Both authors acknowledge support from the Hausdorff institute in Bonn. S. J. thanks  DFG Forschergruppe 718 ``Analysis and Stochastics in Complex Physical Systems'' and ERC Advanced Grant 267356 VARIS of Frank den Hollander for financial support. N.K. was supported by the DFG project BL 1105/2-1 and the DFG priority program 1590 ``Probabilistic structures in evolution''.

\subsubsection*{Addresses}

Sabine Jansen\\
Fakult{\"a}t f{\"u}r Mathematik\\
Ruhr-Universit{\"a}t Bochum\\
44780 Bochum, Germany\\
E-Mail: sabine.jansen@ruhr-uni-bochum.de\\

\noindent
Noemi Kurt\\
Institute of Mathematics, TU Berlin, Sekr. MA 7-5\\
Stra\ss e des 17. Juni 136\\
10963 Berlin, Germany.\\
E-Mail: kurt@math.tu-berlin.de.

\end{document}